\documentclass[12pt]{amsart}
\usepackage{amsmath, amssymb,verbatim,appendix}
\usepackage{amscd}
\usepackage{amsthm}
\usepackage{stmaryrd}
\usepackage{enumerate}
\usepackage{comment}
\usepackage[latin1]{inputenc} 
\usepackage{tikz}
\usetikzlibrary{shapes,arrows}
\usepackage{cite}
\usepackage{url}
\usepackage{float}

\usepackage{setspace,a4wide,color,xcolor,graphicx}

\usepackage[pagebackref=true]{hyperref}

\newtheorem{theorem}{Theorem}[section]
\newtheorem{lemma}[theorem]{Lemma}
\newtheorem{corollary}[theorem]{Corollary}
\newtheorem{proposition}[theorem]{Proposition}

\newtheorem{observation}[theorem]{Observation}

\newtheorem{fact}[theorem]{Fact}

\numberwithin{equation}{section}

\theoremstyle{definition}

\newtheorem{definition}[theorem]{Definition}

\newtheorem{notation}[theorem]{Notation}

\def\E{\mathbb{E}}

\def\R{\mathbb{R}}

\def\F{\mathbb{F}}

\def\VC{\mathrm{VC}}

\def\e{\epsilon}
\def\a{\alpha}
\def\d{\delta}

\newcommand{\calL}{\mathcal{L}}

\newcommand{\calB}{\mathcal{B}}

\newcommand{\calP}{\mathcal{P}}

\newcommand{\calJ}{\mathcal{J}}
\newcommand{\calQ}{\mathcal{Q}}
\newcommand{\calR}{\mathcal{R}}
\newcommand{\calS}{\mathcal{S}}

\newcommand{\Span}{{\mathrm{Span}}}
\newcommand{\At}{{\mathrm{At}}}
\newcommand{\rk}{{\mathrm{rk}}}

\newcommand{\ind}{\mathrm{ind}}

\def\disc{\operatorname{disc}}

\def\poly{\operatorname{poly}}

\newenvironment{proofof}[1]{\indent{\itshape Proof of #1}.\;}{\qed}

\begin{document}
\title[]{On the linear complexity of subsets of $\F_p^n$ bounded $\VC_2$-dimension}
\author{H. Sheats}
\author{C. Terry}\thanks{The second author was partially supported by NSF CAREER Award DMS-2115518 and a Sloan Research Fellowship}

\address{Department of Mathematics, Statistics and Computer Science, University of Illinois at Chicago, Chicago, IL, USA}
\email{hshea3@uic.edu}
\email{caterry@uic.edu}
\maketitle 

\begin{abstract}
Previous work of the second author and Wolf showed that given a set $A\subseteq \F_p^n$ of bounded $\VC_2$-dimension, there is a high rank quadratic factor $\calB$ of bounded complexity such that $A$ is approximately equal to a union of atoms of $\calB$.  The main ingredients in that proof were  a counting lemma for a local version of the $U^3$-norm, along with a quadratic arithmetic regularity lemma of Green and Tao.  This approach yielded bounds of tower type bounds on the linear and quadratic complexities. It was later shown by the same authors that the quadratic complexity can be improved to $\log_p(\poly(\e^{-1}))$, however that proof provided no improvement on the linear component.  In this paper we prove that the bound on the linear complexity can be improved to a triple exponential in the case of linear rank functions, and a quadruple exponential for polynomial rank functions of higher degree.  

Our strategy is based on the one developed by Gishboliner, Wigderson, and Shapira to prove the analogous result in the hypergraph setting. Step 1 is to prove a``cylinder" version of the quadratic arithmetic regularity lemma, which says that given a set $A\subseteq G=\F_p^n$, there is a partition of $G$ into atoms of (possibly distinct) quadratic factors of high rank and bounded complexity, so that most atoms in the partition are uniform with respect to the set $A$, in the sense of a certain local $U^3$ norm. Step 2 is to show that if $A$ has bounded $\VC_2$-dimension, then it has density near $0$ or $1$ on all atoms which are uniform in the sense of Step 1.  We then obtain the structure theorem for $A$ by taking a high rank, common refinement of the factors appearing in the previous steps.  Step 1 relies on a recent  local version of the $U^3$ inverse theorem due to Prendiville, and is necessarily phrased in terms of a local $U^3$ norm implicit in that paper.  On the other hand, Step 2 relies on a counting lemma for a different local $U^3$ due to Terry and Wolf, which we prove here is approximately the same as the local $U^3$ norm used in Step 1.    
\end{abstract}

\section{Introduction}

 There is now a large body of work studying the behavior of bounds in Szemer\'{e}di's graph regularity lemma and its generalizations. We will focus this introduction on the parts of the story most relevant to the current paper, namely arithmetic regularity and strong hypergraph regularity, and  refer the interested reader to \cite{Sh,CF,Gowers.1997, Moshkovitz.2019, Moshkovitz.2016, kolomos, sze, FL} for more background.

 \subsection{Arithmetic regularity}\label{ss:areg} Analogues of the graph regularity lemma for subsets of groups were first developed by Green in \cite{Green.2005}.  In $\F_p^n$, this regularity lemma says that for any subset $A\subseteq \F_p^n$, there exists a subgroup $H\leq \F_p^n$ of bounded index, so that most cosets of $H$ are ``$\e$-regular" with respect to $A$, where ``$\e$-regular" is defined in terms of the Gowers $U^2$-norm, localized to the coset in question.  The bounds in the graph regularity lemma are known to be necessarily of tower type (see \cite{Gowers.1998, FL}), and Green proved the same is true of the bounds in this arithmetic analogue  \cite{Green.2005}.  
 
By the mid 2010's, several papers had appeared showing that in the setting of graphs, various combinatorial restrictions yield improved versions of the regularity lemma (see e.g. \cite{Alon.2007, Lovasz.2010, Malliaris.2014, Semi1, Semi2}).  The first result along these lines in the arithmetic setting is due to the second author and Wolf, who showed that $k$-stable subsets of $\F_p^n$ admit $\e$-regular subgroups  with index $\exp_p(\e^{-O_k(1)})$ and no irregular cosets.  Several more papers on ``tame" arithmetic regularity lemmas have since appeared,  including extensions and refinements in the stable case  \cite{Conant.2021, Terry.2020, Conant.2017p4}, in addition to results in the more general setting of bounded VC-dimension \cite{Alon.2018is, Sisask.2018, Conant.2018zd, CT}.\footnote{While the setting of bounded VC-dimension is more general than the stable setting, the structure theorems obtained in the stable setting are stronger.}  These results  are all naturally stated as structure theorems for the set in question. In $\F_p^n$, they say the set $A$ is well approximated by a union of cosets of a bounded index subgroup, with the distinctions among them boiling down to differences in the meaning of ``well approximated."

\subsection{Hypergraph regularity}\label{ss:hg} Several variants of the regularity lemma exist for hypergraphs (see \cite{Nagle.2013} for a survey of the $3$-uniform case). The version relevant to this paper is the ``strong" form of the regularity lemma for $3$-uniform hypergraphs, as developed  by subsets of the authors Gowers, Frankl, Kohayakawa, Nagle, R\"{o}dl, Skokan, and Schacht \cite{ Gowers.2006, Gowers.2007, Frankl.2002, Rodl.2005, Nagle.2006,Rodl.2005a}.  For results related to the analogous problem for \emph{weak} regularity, we refer the reader to \cite{FrPS, Semi1,Semi2, AFP, CS1, CS2, Terry.2024a, Terry.2024b, GSW1}. 

A strong regular decomposition for a $3$-uniform hypergraph consists of a pair $\calP=(\calP_{vert},\calP_{pairs})$, where  $\calP_{vert}$ partitions the vertex set,  and   $\calP_{pairs}$ partitions the pairs of vertices. This gives rise to two complexity parameters, $t$ for the vertex partition, and $s$ for the pairs partition.   The statement of the regularity lemma contains a constant error parameter $\e>0$ in addition to an ``error function"  $\psi:\mathbb{N}\rightarrow (0,1]$. In practice, the sets in $\calP_{pairs}$ are required to be $\psi(s)$-regular as graph relations, where $s$ is the pairs complexity. In essentially all applications, it suffices to take $\psi$ a sufficiently fast growing polynomial in $s^{-1}$.  The first major result on the bounds for strong hypergraph regularity was due to Moshkovitz and Shapira \cite{Moshkovitz.2019}, who showed that even for polynomial choices of $\psi$, the vertex partition must grow at a wowzer rate, matching the corresponding upper bound from the proofs of the regularity lemma.\footnote{In fact their theorems are much more general, giving lower bounds for $k$-uniform hypergraphs for all $k\geq 3$.}  This opened up the  field to a new problem: which combinatorial restrictions give rise to better bounds in the strong regularity lemma for $3$-uniform hyergraphs?  

The first result along these lines is due to the second author, who showed $3$-uniform hypergraphs of bounded $\VC_2$-dimension (a higher arity analogue of VC-dimension)\footnote{See Definition \ref{def:vc2}.} admit polynomial bounds on the pairs partition.  A systematic investigation of these problems in the $3$-uniform case was later undertaken by the second author in \cite{Terry.2024c,Terry.2024d}, with \cite{Terry.2024c} studying the growth of $|\calP_{vert}|$ and  \cite{Terry.2024d} studying the growth of $|\calP_{pairs}|$.  This work left open several interesting questions.  The two most important problems related to the vertex partition were subsequently resolved by Gishboliner, Shapira, and Wigderson  in \cite{GSW1} and \cite{GSW2}. As the latter is important for this paper,  we explain its contribution in more detail below.

When studying the bounds in the strong hypergraph regularity lemma for $3$-uniform hypergraphs, a choice must be made for what $\psi$ are allowed.  In \cite{Terry.2024c,Terry.2024d}, the choice was made to study the problem where arbitrary $\psi$ are allowed.  Interestingly, only the jump to the fastest growth rate was sensitive to this choice.   In the case of the vertex partition, combining \cite{Terry.2024c} with  \cite{GSW1}\footnote{The work of \cite{Terry.2024c} uses an earlier paper, \cite{Terry.2024a}, to prove  a jump from double exponential to wowzer. This can  be improved to a jump from single exponential to wowzer by replacing the use of \cite{Terry.2024a} there with \cite{GSW1}.} implies that when arbitrary $\psi$ are allowed, there is a jump all the way from  exponential to wowzer growth, characterized by whether or not a property is ``close" to finite slicewise VC-dimension.\footnote{The jump to wowzer in \cite{Terry.2024c} uses the fact that the strong graph regularity lemma requires wowzer bounds (see \cite{CF, Sh}). A similar result was recently proved for the strong linear arithmetic regularity lemma by Gladkova \cite{Gladkova.2025}.}   This suggests that in the regime where arbitrary $\psi$ are allowed, $\VC_2$-dimension does not characterize a jump.  This was surprising, as $\VC_2$-dimension characterizes many other dichotomies for $3$-uniform hypergraphs, including in the growth of the pairs partition \cite{Terry.2018, Terry.2022, Terry.2024d}, and in the labeled enumeration problem \cite{Terry.2018}.  On the other hand, the work in \cite{Terry.2024c} does not address the question of what happens with the vertex partition between the exponential and wowzer ranges, in the regime where only reasonably slow $\psi$ are allowed. This is a problem of great interest, as such $\psi$ are the ones used in applications.

 Gishboliner, Shapira, and Wigderson resolved this problem by showing that after modestly restricting the growth rate of $\psi$,  an additional jump appears between exponential and wowzer, characterized by $\VC_2$-dimension. One direction of this is implicit in prior work of Moshkovitz and Shapira (i.e. that unbounded $\VC_2$-dimension implies wowzer growth, even with polynomial $\psi$). The contribution of  \cite{GSW2} is the other direction. Specifically, they prove that when the $\VC_2$-dimension is bounded and $\psi$ is sub-tower, there is a double tower bound on the vertex partition.  We outline their proof in broad strokes below, as it guides our strategy in this paper:

 \vspace{2mm}

\noindent{\bf Step 1:} Prove a Duke-Leffman-Rodl ``cylinder" style hypergraph regularity lemma, which admits better bounds (namely tower) when the $\psi$ is assumed to be polynomial. 
\vspace{2mm}

\noindent{\bf Step 2:} When the hypergraph has bounded $\VC_2$-dimension, deduce it has density near $0$ or $1$ on all regular triads from Step 1. This is a straightforward  corollary of Gowers' counting lemma, which crucially requires only polynomially growing $\psi$ (see \cite{Gowers.2006}). 

\vspace{2mm}

\noindent{\bf Step 3:} Take the common refinement of all the components arising from the cylinder decompositions from Step 1.

\vspace{2mm}

\noindent{\bf Step 4:} Apply Szemer\'{e}di's regularity lemma to regularize the resulting pairs partition.  As densities close to $0$ and $1$ are mostly preserved under refinements, the resulting decomposition still has the property that most triads have density near $0$ or $1$.  As triads with density near $0$ or $1$ are always regular, this yields a regular decomposition. 

\vspace{2mm}

Because Step 1 yields a tower bound, and Step 4 applies the graph regularity lemma one additional time, this strategy yields a double tower bound.  
 
\subsection{Quadratic arithmetic regularity and main results}
There exists a hierarchy of arithmetic regularity lemmas in rough correspondence with the strong regularity lemmas for $k$-uniform hypergraphs.  The arithmetic analogue of the strong regularity lemma for 3-uniform hypergraphs is the \emph{quadratic arithmetic regularity lemma}.  In $\F_p^n$, a regularity lemma of this type comes equipped with \emph{quadratic factor}, consisting of a \emph{linear} component and a \emph{quadratic} component (see Subsection \ref{ss:factors} for detailed definitions). This gives rise to two complexity parameters, $\ell$ for the linear component, and $q$ for the quadratic component.   The  quadratic arithmetic regularity lemma making the analogy to hypergraph regularity most explicit is a formulation from \cite{Terry.2021d}, which says that given any $A\subseteq \F_p^n$, there is a high rank quadratic factor $\calB$ of bounded complexity so that on  most atoms of $\calB$, the set $A$ is uniform in the sense of a local $U^3$-norm defined there, which we will denote by $\|\cdot \|_{U^3(d)}^{TW}$ (see Subsection \ref{ss:TW} for more details). 

\begin{theorem}[Proposition 4.2 in \cite{Terry.2021d}\footnote{This result first appeared in early versions of \cite{Terry.2021a}, and was later moved to its own paper \cite{Terry.2021d}.}]\label{thm:u3reg} 
    For all odd primes $p$, all $\delta\in (0,1)$, and all growth functions\footnote{A \emph{growth function} is an increasing function from $\mathbb{R}$ into $\mathbb{R}^{\geq 0}$.} $\rho$, there is a constant $M(p,\delta,\rho)$ such that the following holds.
    
    For all $A\subseteq \F_p^n$,  there are integers $\ell,q$ and a quadratic factor $\calB = (\calL, \calQ)$ on $\F_p^n$ of complexity $(\ell, q)$ and rank at least $\rho(\ell+q)$ such that 
    \begin{enumerate}
    \item $\ell, q\leq M$
    \item for at least a $(1-\delta)$-fraction of the tuples $d\in (\F_p^{\ell}\times \F_p^q)^3\times (\F_p^q)^3$, we have 
     $$
    \|1_A-\alpha_{B(\Sigma(d))}\|_{U^3(d)}^{TW}<\delta,
    $$
     where $\alpha_{B(\Sigma(d))}$ is the density of $A$ on the atom $B(\Sigma(d))$.
    \end{enumerate}
\end{theorem}

In \cite{Terry.2021d}, Theorem \ref{thm:u3reg} is deduced from a quadratic arithmetic regularity lemma of Green and Tao from \cite{GT}.  As a result, when the rank function $\rho$ is polynomial, the proof of Theorem \ref{thm:u3reg} in \cite{Terry.2021d} give a  tower type bound for $M$  (see Section \ref{sec:warmup} for details).

The formulation of Theorem \ref{thm:u3reg} allows us to make a precise analogy between quadratic arithmetic regularity and hypergraph regularity.  Specifically,  Theorem \ref{thm:u3reg} implies that a special version of the hypergraph regularity lemma holds for the $3$-uniform hypergraph  $(\F_p^n, \{xyz: x+y+z\in A\})$, where $\calP_{vert}$ consists of the atoms of a quadratic factor $\calB$, and $\calP_{pairs}$ consists of the fibres of the associated bilinear form (see Appendix B of \cite{Terry.2021a}). The vertex complexity is then $p^{\ell+q}$, the pairs complexity is $p^q$, and the error function $\psi$ corresponds to the so-called \emph{rank function} $\rho$.  It is important to note that this tells us such hypergraphs admit regularity lemmas with better bounds on the vertex complexity than is possible in general. In particular, the bound $p^{\ell+q}$ in Theorem \ref{thm:u3reg} is tower type, while \cite{Moshkovitz.2019} showed that in general, there exist $3$-uniform hypergraphs requiring a vertex partition of wowzer size.   On the other hand, recent results of Gladkova \cite{Gladkova.2025} show one cannot do better than tower bounds on the linear complexity in a version of Theorem \ref{thm:u3reg} stated for functions (rather than sets).

Using Theorem \ref{thm:u3reg} and a counting lemma for the local $U^3$-norm appearing there, the second author and Wolf proved a structure theorem for sets of bounded $\VC_2$-dimension, showing they are approximated by unions of quadratic atoms \cite{Terry.2021d}.  This   result relied on Theorem \ref{thm:u3reg}, and thus yielded tower bounds for both $\ell$ and $q$ (when the rank function is polynomial).   The second author and Wolf later showed \cite{Terry.2021f} that the bound on the quadratic complexity of the factor can be drastically improved, yielding the following theorem.

\begin{theorem}[Theorem 1.3 in \cite{Terry.2021f}]\label{thm:TW1}
For all odd primes $p$, all integers $k\geq 1$, all growth functions $\rho$, and all sufficiently small $\delta\in (0,1)$, there is a constant $C=C(p,k,\rho,\delta)$ so that the following holds for all sufficiently large $n$.    

For all $A\subseteq G=\F_p^n$ with $\VC_2$-dimension at most $k$, there exists a quadratic factor $\calB$ on $\F_p^n$ of complexity $(\ell,q)$ and rank at least $\rho(\ell+q)$  such that 
\begin{enumerate}
\item $0\leq \ell\leq C$,
\item $0\leq q\leq \log_p(\delta^{-k+o(1)})$, and
\item there is a union $Y$ of atoms  of $\calB$, such that $|A\Delta Y|\leq \delta |G|$.
\end{enumerate}
\end{theorem}

As explained above, Appendix B of \cite{Terry.2021a} shows $p^q$ is the correct analogue of the pairs complexity  in Theorem \ref{thm:TW1}.  With this in mind, we see that the upper bound $\log_p(\delta^{-k+o(1)})$ translates to a bound of $\delta^{-k+o(1)}$ for the pairs complexity of the corresponding hypergraph decomposition. This matches what one would expect in the bounded $\VC_2$-dimension setting from the analogous hypergraph results (see \cite{Terry.2022, Terry.2024d}). 

 The proof of Theorem \ref{thm:TW1} proceeds by first approximating the set $A$ with a quadratic factor using the structure theorem of \cite{Terry.2021d}, then building a new factor with the same linear component, but much smaller quadratic component.  Thus, while this strategy improves the quadratic complexity, it leaves the linear complexity unchanged (namely bounded by a tower when $\rho$ is polynomial). 

The goal of this paper is to improve the bound on the linear component in Theorem \ref{thm:TW1}. Our result can be seen as an arithmetic analogue of \cite{GSW2}. However, as the general bound on the linear complexity in Theorem \ref{thm:u3reg} is tower type already, our analogous theorem must attain a sub-tower bound.  Specifically, we show that  when the rank function is polynomial, one can improve the bound on $\ell$ in Theorem \ref{thm:TW1} to a quadruple exponential, and when $\rho$ is linear (as suffices for most applications), one can obtain a triple exponential. The machinery of \cite{Terry.2021f} allows us to obtain this improvement on the linear complexity simultaneously with the improvement on the quadratic complexity from Theorem \ref{thm:TW1}, yielding the following main result.

\begin{theorem}\label{thm:main}
For all odd primes $p$, all integers $k\geq 1$, all polynomial growth functions $\rho$, and all sufficiently small $\delta\in (0,1)$, the following holds for all sufficiently large $n$. 

For all $A\subseteq G=\F_p^n$ with $\VC_2$-dimension at most $k$, there exists a quadratic factor $\calB$ on $\F_p^n$ of complexity $(\ell,q)$ and rank at least $\rho(\ell+q)$ such that 
\begin{enumerate}
\item $0\leq \ell\leq \exp_p(\exp_p(\exp_p(\exp_p(O_{k,p,\rho}(\delta^{-O_{k,p}(1)}))))$,
\item $0\leq q\leq \log_p(\delta^{-k-o(1)})$, and
\item there is a union $Y$ of atoms  of $\calB$, such that $|A\Delta Y|\leq \delta |G|$.
\end{enumerate}
Moreover, if $\rho$ has degree $1$, then one can ensure $\ell\leq \exp_p(\exp_p(\exp_p(O_{k,p,\rho}(\delta^{-O_{k,p}(1)})))$.
\end{theorem}

Our overall strategy is based on the one used in \cite{GSW2} by Gishboliner, Shapira, and Wigderson to prove the hypergraph analogue. We give a rough outline of the  arithmetic version of this argument for a set $A\subseteq \F_p^n$ of bounded $\VC_2$-dimension, after which we discuss the additional ingredients needed to make this strategy work.

\vspace{2mm}

\noindent{\bf Step 1:} Prove a ``cylinder" version of the quadratic arithmetic regularity lemma with double exponential bounds for linear rank functions. This theorem will partition the group into atoms of (possibly distinct) high rank factors, so that the set $A$ is appropriately uniform on most atoms in the partition.   

\vspace{2mm}

\noindent{\bf Step 2:} Show that since $A$ has bounded $\VC_2$-dimension, it has density near $0$ or $1$ on all of the uniform atoms from Step 1. 

\vspace{2mm}

\noindent{\bf Step 3:} Take the common refinement of all the factors used to build the cylinder partition from Step 1.  
 \vspace{2mm}

 \noindent{\bf Step 4:} Take a high rank refinement of the resulting factor.   As refinements mostly preserve densities near $0$ or $1$, this will result in a factor where $A$ has density near $0$ or $1$ on most atoms. From this we deduce that $A$ looks approximately like the union of the ``density near $1$" atoms.
 \vspace{2mm}

There are several technical hurtles related to Steps 1 and 2 which must be overcome to make this strategy work. We are able to resolve these difficulties, in part due to several recent results from the literature.  
 
 First, proving a cylinder version of the quadratic arithmetic regularity lemma requires a localized version of the $U^3$ inverse theorem which was proven recently by Prendiville \cite{P} (see Subsection \ref{ss:P} for more details).  Using this result, we accomplish Step 1, where the resulting notion of ``uniform on most atoms" is defined in terms of a local version of the $U^3$-norm implicit in Prendiville's work. We give a  slightly informal statement of this cylinder regularity lemma below and refer the reader to Section \ref{sec:cylinder} for details. While a result of this type in for the $U^2$ norm was proven by Fox, Tidor, and Zhao in \cite{TFZ}, this is to our knowledge the first such result for the $U^3$ norm.

 \begin{theorem}[``Cylinder" quadratic arithmetic regularity lemma]\label{thm:cylinderintro}
    Fix $p$ an odd prime, $\rho$ a polynomial growth function, and $\delta\in (0,1)$.  For all $A\subseteq G=\F_p^n$, there exists a partition $\calP$ of $G$ such that for each $P\in \calP$ there is a quadratic factor $\calB_P = (\calL_P, \calQ_P)$ of complexity $(\ell_P, q_P)$  such that the following hold.
   \begin{enumerate}
   \item For all $P\in \calP$, $P$ is an atom of $\calB_P$, $\calB_P$ has rank at least $\rho(\ell_P+q_P)$,  and
   $$
  \ell_P\leq \exp_p(\exp_p(\exp_p(O_{\rho,p}(\delta^{-1}))))\text{ and }q_P\leq O_{p}(\delta^{-O_p(1)});
   $$
   \item  $|\bigcup_{P\in \calJ}P|\geq (1-\delta)|G|$, where $\calJ$ is the set of $P\in \calP$ which are $\delta$-uniform with respect to $A$, in terms of the local $U^3$ norm from \cite{P}. 
   \end{enumerate}
   Moreover, if $\rho$ has degree $1$, one can ensure that for all $P\in \calP$, $\ell_P\leq \exp_p(\exp_p(O_{\rho,p}(\delta^{-1})))$.
\end{theorem}

Theorem \ref{thm:cylinderintro} is proven via an energy increment argument which iteratively applies Prendiville's local $U^3$ inverse theorem.  Clearly the bounds generated by such an argument  will depend on the bounds in Prendiville's result. Luckily for us, these are polynomial, a fact which in turn depends on  the recent resolution of the polynomial Freiman Rusza conjecture in $\F_p^n$ \cite{Gowers.2024}.  As Prendiville's inverse theorem requires a rank assumption, the proof of Theorem \ref{thm:cylinderintro} interweaves steps in which we apply the inverse theorem and steps in which we perform operations to make the factors have high rank.  It requires some delicacy to analyze the bounds resulting from such a process in enough detail to obtain the triple exponential bound.  This extra care is not required in the hypergraph case \cite{GSW2},  because there one need only obtain a tower bound at this step. 

Turning to Step 2, one runs into another hurtle, namely proving an appropriately general counting lemma  for the local $U^3$-norm in the conclusion of Theorem \ref{thm:cylinderintro}.  One has no  choice over which local norm is used in Step 1, as it must match the hypotheses of Prendiville's local inverse theorem.  For this reason, we denote this local norm by $\|\cdot \|_{U^3(b)}^P$, where the ``P" stands for Prendiville.  On the other hand, a \emph{different} local $U^3$ norm, which we denote by $\|\cdot \|_{U^3(d)}^{TW}$, was defined in \cite{Terry.2021d}, where an accompanying counting lemma was proved.  In fact, exactly the statement needed for Step 2 was proved in \cite{Terry.2021d}, except for the local norm $\|\cdot \|_{U^3(d)}^{TW}$ rather than the local norm $\|\cdot \|_{U^3(d)}^{P}$. To make Step 2 work, we show   these norms are in fact approximately the same (see Subsection \ref{ss:norms}).  A consequence of this is that \cite{Terry.2021d} implies a counting lemma for $\|\cdot \|_{U^3(b)}^P$, which may be of general interest (see Propositions 3.19 and D.11 in \cite{Terry.2021d}).

We end this introduction with remarks on lower bounds. First, it was shown in \cite{Terry.2021f} that there exist sets requiring $q$ to grow at a rate of at least a power of $\delta^{-1}$ in Theorem \ref{thm:u3reg}, showing Theorem \ref{thm:TW1} does not hold in general.  Regarding the linear component, a recent result of Gladkova shows the version of Theorem \ref{thm:u3reg} for \emph{functions} requires tower type lower bounds on on the linear complexity.  It seems likely this can be turned into an example of a subset of $\F_p^n$ requiring tower sized linear complexity in Theorem \ref{thm:u3reg} (showing the bound on $\ell$ in Theorem \ref{thm:main} cannot hold in general). However, we do not undertake this task in this paper.  We believe the bound on $q$ in Theorem \ref{thm:main} is roughly optimal, however, it seems possible the bound for $\ell$ can be improved further.  It remains open what happens when faster rank functions are allowed. In the setting of $3$-uniform hypergraphs, it was shown in \cite{Terry.2024c} that there exist $3$-uniform hypergraphs of bounded $\VC_2$-dimension requiring wowzer lower bounds on the vertex partition for fast growing $\psi$. Gladkova has proved arithmetic analagues of several ingredients in that argument, however it remains open whether that strategy can be fully executed in the arithmetic setting (see Section 4.1 of \cite{Gladkovathesis}).

\subsection{Acknowledgements} The second author thanks Julia Wolf for many helpful discussions about topics related to this paper, and especially for clarifying the bounds in the quadratic arithmetic regularity lemma of \cite{Green.2007}.

\subsection{Notation}\label{ss:notation}
Throughout the rest of this paper, $p$ is an odd prime.   By convention, the natural numbers start at $0$. For an integer $n\geq 1$, we let $[n]=\{1,\ldots, n\}$.

Given a set $X$, $\E_{x\in X}$ means $\frac{1}{|X|}\sum_{x\in X}$.  Given a finite index set $I=\{\alpha_1,\ldots, \alpha_t\}$, and sets $X_i$ for $i\in I$, we write $\E_{\begin{subarray}{l}  x_i\in X_i\\i\in I\end{subarray}}$ to mean $\E_{x_1\in X_{\alpha_1}}\ldots \E_{x_t\in X_{\alpha_t}}$.

 Given real numbers $r,s$, we write $r=s\pm \e$ to mean $|r-s|\leq \e$. Given real valued functions $f,g,h$ on the same domains, we write $f\gg g$ to mean there exists a constant $C$ such that $f(x)\geq Cg(x)$ for all $x$, and we write $f=h(1\pm O(g))$ to mean there is a constant $C$ so that for all $x$, $|f(x)-h(x)|\leq Ch(x)g(x)$.  Given a parameter $\theta$, we write $f\gg_{\theta} g$ and/or $f=h(1\pm O_{\theta}(g))$ to denote that the relevant constant $C$ depends on the parameter $\theta$.   Throughout the paper, we will refer to the following notion of a \emph{growth function}.

\begin{definition}[Growth functions]\label{def:growth}
A \emph{growth function} is any increasing function $\rho:\mathbb{R}\rightarrow \mathbb{R}^{\geq 0}$.  We say a growth function is \emph{polynomial} if there exists a constant $C>0$ and an integer $d\geq 1$ so that for all $x\geq 1$, $\rho(x)\leq Cx^d$.  The minimal $d$ for which this holds is then called the \emph{degree} of $\rho$.
\end{definition}

\subsection{Outline}

In Section \ref{sec:prelim}, we introduce the definition of the $\VC_2$-dimension of a subset in a group, the definitions of linear and quadratic factors, and state several lemmas for computing sizes of sets defined by such factors.  We also cover notation needed to translate the results of \cite{P} into the terminology used in this paper.  In Section \ref{sec:inversenorms}, we introduce the localized inverse theorem of Prendiville \cite{P}, and introduce both the local $U^3$ norm implicit in that paper, as well as the one defined by the second author and Wolf in \cite{Terry.2021d}. Also in this section, we prove these local norms are approximately equal, allowing us to translate a crucial theorem about sets of bounded $\VC_2$-dimension from \cite{Terry.2021d} into a statement involving the local norm from \cite{P}. In Section \ref{sec:rank} we set out definitions and prove lemmas which we will need to analyze the bounds in our cylinder regularity lemma, Theorem \ref{thm:cylinder}.  In Section \ref{sec:warmup}, as a warm up  to Theorem \ref{thm:cylinder},  we use Prendiville's local inverse theorem to give an energy increment proof of a quadratic arithmetic regularity lemma. In Section \ref{sec:cylinder} we prove our cylinder style quadratic arithmetic regularity lemma, Theorem \ref{thm:cylinder}.  Finally, in Section \ref{sec:sumup}, we prove our main theorem.
 
\section{Preliminaries}\label{sec:prelim}
 In this section we cover the definition of $\VC_2$-dimension, the higher arity version of VC-dimension which is the topic of this paper. We then introduce preliminaries related to linear and quadratic factors.

 \subsection{$\VC_2$-dimension}

 There are several ways to generalize the notion of VC-dimension in graphs to the setting of hypergraphs. This paper focuses on an analogue for $3$-uniform hypergraphs called $\VC_2$-dimension.  For more papers studying this notion, we refer the reader to the literature \cite{Terry.2021b,Chernikov.2019, Shelah.2014, Chernikov.2020, Terry.2021a,Terry.2021d,Terry.2021f, Terry.2022, Terry.2018,Terry.2024c, Terry.2024d, GSW2, Hempel.2016,Chernikov.2019b, Chernikov.2024, coregliano.2025b}. For related papers on other higher arity generalizations of VC-dimension, we refer the reader to \cite{Terry.2024a,Terry.2021b,Terry.2024b,GSW1,Chernikov.2020, FrPS,CS2, Chernikov.2025, coregliano.2024,  coregliano.2025}, for example.  This paper is interested in specifically subsets of groups, and we will restrict our attention to this context in what follows.  For context, we begin by defining the VC-dimension of a subset of a finite group.

\begin{definition}
Given an integer $k\geq 1$, a finite group $G$, and a subset $A\subseteq G$, we say \emph{$A$ has VC-dimension at least $k$} if there exist $a_1,\ldots, a_k\in G$ and $\{b_S:S\subseteq [k]\}\subseteq G$ so that for all $i\in [k]$ and $S\subseteq [k]$,
$$
\text{$a_i\cdot b_S\in A$ if and only if $i\in S$.}
$$
The \emph{VC-dimension of $A$} is then defined to be
$$
\VC(A)=\max\Big(\{0\}\cup \{k\in \mathbb{Z}^{\geq 1}: \text{$A$ has VC-dimension at least $k$}\}\Big).
$$
\end{definition}

As mentioned in the introduction, subsets of groups of bounded VC-dimension have been studied  in \cite{Alon.2018is, Sisask.2018, Conant.2018zd, CT}.  We now define the $\VC_2$-dimension of a subset of a group.

 \begin{definition}\label{def:vc2}
Given an integer $k\geq 1$, a finite group $G$, and a subset $A\subseteq G$, we say \emph{$A$ has VC-dimension at least $k$} if there exist $a_1,\ldots, a_k,b_1,\ldots, b_k\in G$ and $\{c_S:S\subseteq [k]\times [k]\}\subseteq G$ so that for all $i,j\in [k]$ and $S\subseteq [k]$,
$$
\text{$a_i\cdot b_j\cdot c_S\in A$ if and only if $(i,j)\in S$.}
$$
The \emph{$\VC_2$-dimension of $A$} is then defined to be
$$
\VC_2(A)=\max\Big(\{0\}\cup \{k\in \mathbb{Z}^{\geq 1}: \text{$A$ has $\VC_2$-dimension at least $k$}\}\Big).
$$
 \end{definition}

Subsets of elementary abelian $p$-groups of bounded $\VC_2$-dimension have been previously studied in \cite{Terry.2021a,Terry.2021b,Terry.2021f}. For more history on these subjects, we refer the reader back to the introduction.
 
\subsection{Linear and quadratic factors}\label{ss:factors}

In this section we introduce notation for the linear and quadratic factors used in this paper. Our setup is largely based on \cite{Green.2007}, and is designed to be consistent with \cite{Terry.2021a,Terry.2021d,Terry.2021f}.  We then cover background required to translate results from \cite{P} into our setup. Finally, we cover several lemmas which approximately compute the sizes of sets arising from quadratic factors.

We begin by defining linear factors and linear atoms.  We define these in two cases, based on whether or not the complexity is trivial.

\begin{definition}[Linear factors and atoms]
Given an integer $\ell \geq 1$, a \emph{linear factor on $\F_p^n$ of complexity $\ell$} is a set $\calL=\{r_1,\ldots, r_{\ell}\}\subseteq \F_p^n$ of linearly independent vectors. Given $a=(a_1,\ldots, a_{\ell})\in \F_p^{\ell}$, define
$$
L(a)=\{x\in \F_p^n: \text{ for each }i\in [\ell], x\cdot r_i=a_i\}.
$$
The sets of form $L(a)$ for $a\in \F_p^{\ell}$ are called \emph{atoms of $\calL$}, and we refer to the tuple $a$ as the \emph{label} for the atom $L(a)$.  We let $\At(\calL)$ denote the set of atoms of $\calL$. 

The \emph{linear factor on $\F_p^n$ of complexity $0$} is $\calL=\emptyset$. It has a unique atom, $L(0):=\F_p^n$, where the label $0$ is the unique element of the $0$-dimensional vector space $\F_p^0$.  We then let $\At(\calL)=\{\F_p^n\}=\{L(a): a\in \F_p^0\}$.
\end{definition}

The reader may notice that our notation for the atoms of a factor depends on a fixed enumeration of the vectors in $\calL$.  For this reason it would be more correct to define factors as tuples of vectors rather than sets of vectors.  In practice, no confusion will arise from this minor abuse of notation. 

Every linear factor comes with an associated linear form, whose fibres make up its atoms.

\begin{definition}[Linear form associated to a linear factor]\label{def:linform}
If  $\ell\geq 1$ is an integer, and $\calL=\{r_1,\ldots, r_{\ell}\}$ is a linear factor on $G=\F_p^n$ of complexity $\ell$, define $\beta_{\calL}:G\rightarrow \F^{\ell}_p$ by setting 
$$
\beta_{\calL}(x)=(x\cdot r_1,\ldots, x\cdot r_{\ell}),
$$
for each $x\in G$. 

If $\calL$ is the linear factor on $G=\F_p^n$ of complexity $0$, define $\beta_{\calL}:G\rightarrow \F^0_p$ by setting $\beta_{\calL}(x)=0$ for all $x\in G$. 
\end{definition}

In the notation of Definition \ref{def:linform}, it is clear $\beta_{\calL}$ is a linear form, and for each $a\in \F_p^{\ell}$, $L(a)=\beta_{\calL}^{-1}(a)$.  It will be useful to know that any linear map from $\F_p^n$ into $\F_p$ comes from such a map in the following sense.  

\begin{fact}
Any linear map $\phi:\F_p^n\rightarrow \F_p$ has the form $\phi(x)=\beta_{\calL}(x)+c$ for some linear factor $\calL$ on $\F_p^n$, and some constant $c\in \F_p$.  
\end{fact}
\begin{proof}
Let $e_1,\ldots, e_n$ denote the standard basis vectors in $\F_p^n$. Setting $\calL=\{\phi(e_1),\ldots, \phi(e_n)\}$ and $c=\phi(0)$, it is easy to check that for all $x\in \F_p^n$, $\phi(x)=\beta_{\calL}(x)+c$, as desired.
\end{proof}

We now define purely quadratic factors. We again do this in two cases, based on whether or not the complexity is trivial.

\begin{definition}[Purely quadratic factors]
Given an integer $q \geq 1$, a \emph{purely quadratic factor on $\F_p^n$ of complexity $q$} is a set $\calQ=\{M_1,\ldots, M_{q}\}$ of pairwise distinct symmetric $n\times n$ matrices with entries from $\F_p$.   Given $a=(a_1,\ldots, a_q)\in \F_p^q$, define
$$
Q(a)=\{x\in \F_p^n: \text{ for each }i\in [q], x^TM_ix=a_i\}.
$$
Sets of the form $Q(a)$ are called \emph{atoms of $\calQ$}, and we refer to the tuple $a\in \F_p^q$ as the \emph{label} for the atom $Q(a)$. We let $\At(\calQ)$ denote the set of atoms of $\calQ$.  

The \emph{purely quadratic factor on $\F_p^n$ of complexity $0$} is $\calQ=\emptyset$. It has a unique atom, $Q(0):=\F_p^n$, where the label $0$ is the unique element of the $0$-dimensional vector space $\F_p^0$. We then let $\At(\calQ)=\{\F_p^n\}=\{Q(a): a\in \F_p^0\}$.
\end{definition}

Every purely quadratic factor comes with a corresponding bilinear form.

\begin{definition}[Bilinear forms associated to purely quadratic factors]\label{def:betaq}
Let $q\geq 1$ be an integer, and let $\calQ=\{M_1,\ldots, M_q\}$ be a purely quadratic factor on $G=\F_p^n$ of complexity $q$.   Define $\beta_{\calQ}:G^2\rightarrow \F^q_p$ by setting 
$$
\beta_{\calQ}(x,y)=(x^TM_1y,\ldots, x^TM_qy),
$$
for each $(x,y)\in G^2$. 

If $\calQ$ is the purely quadratic factor on $G=\F_p^n$ of complexity $0$, define $\beta_{\calQ}:G^2\rightarrow \F^0_p$ by setting $\beta_{\calQ}(x,y)=0$ for all $(x,y)\in G^2$. 
\end{definition}

In the notation of Definition \ref{def:betaq}, it is not difficult to see that $\beta_{\calQ}$ is a symmetric, bilinear form. The fibres, $\beta_{\calQ}^{-1}(b)=\{(x,y)\in G^2: \beta_{\calQ}(x,y)=b\}$, will play an important role in Section \ref{sec:inversenorms}.   The atoms of $\calQ$ can also be written in terms of $\beta_{\calQ}$. Specifically, for every $b\in \F_p^q$, $Q(b)=\{x\in G: \beta_{\calQ}(x,x)=b\}$.

Purely quadratic factors come with the following notion of rank.

\begin{definition}[Rank of a purely quadratic factor]\label{def:rank}
Suppose $q\geq 0$ and $\calQ$ is a purely quadratic factor on $\F_p^n$ of complexity $q$. We define the \emph{rank of $\calQ$} in cases below.
\begin{enumerate}
\item If $q=0$, then $\calQ$ has rank $n$.
\item If $q>0$ and $\calQ=\{M_1,\ldots, M_q\}$, then the rank of $\calQ$ is the minimal rank of a non-trivial linear combination of the matrices in $\calQ$. In other words,
$$
\min\{\rk(\lambda_1M_1+\ldots +\lambda_qM_q): \lambda_1,\ldots, \lambda_q\in \F_p\text{ not all zero}\}.
$$
\end{enumerate}
\end{definition}

We now define quadratic factors, which are obtained by combining linear and purely quadratic factors.

\begin{definition}[Quadratic factors]\label{def:quadfactor}
Given integers $\ell,q \geq 0$, a \emph{quadratic factor on $\F_p^n$ of complexity $(\ell,q)$} is a pair $\calB=(\calL,\calQ)$ where $\calL$ is a linear factor on $\F_p^n$ of complexity $\ell$ and $\calQ$ is a purely quadratic factor on $\F_p^n$ of complexity $q$.  An \emph{atom of $\calB$} is a set of the form 
$$
B(a,b)=L(a)\cap Q(b),
$$
for some $(a,b)\in \F_p^{\ell}\times \F_p^q$. We call the pair $(a,b)\in \F_p^{\ell}\times \F_p^q$ the \emph{label} of the atom $B(a,b)$.  We let $\At(\calB)$ denote the set of atoms of $\calB$.
\end{definition}

Given a quadratic factor $\calB$, it will be useful to have a function taking a group element $x$ to the label of the atom of $\calB$ containing $x$.  This function is naturally built from the the linear and bilinear forms $\beta_{\calL}$ and $\beta_{\calQ}$.

\begin{definition}\label{def:betab}
Suppose $\ell,q \geq 0$ are integers and $\calB$ is a quadratic factor on $G=\F_p^n$ of complexity $(\ell,q)$.  Define $\beta_{\calB}:G\rightarrow \F_p^{\ell}\times \F_p^q$ by setting 
$$
\beta_{\calB}(x)=(\beta_{\calL}(x), \beta_{\calQ}(x,x)),
$$
for all $x\in G$. 
\end{definition}

In the notation of Definition \ref{def:betab}, it is easy to check we have $x\in B(\beta_{\calB}(x))$ for all $x\in G$. We next extend the notion of rank to quadratic factors in the obvious way.

\begin{definition}[Rank of a quadratic factor]
Suppose $\ell,q\geq 0$ are integers and $\calB=(\calL,\calQ)$ is a quadratic factor on $\F_p^n$ of complexity $(\ell,q)$. The \emph{rank of $\calB$} is the rank of $\calQ$, as defined in Definition \ref{def:rank}.
\end{definition}

When a factor $\calB=(\calL,\calQ)$ has high rank, sets involving the fibres of the maps $\beta_{\calB}$ and $\beta_{\calQ}$ have predicable sizes.  We will use several statements to this effect throughout the paper.  For the convenience of the reader, we collect these and related lemmas in their own subsection at the end of this section (see Subsection \ref{ss:sizes}).

Given two quadratic factors $\calB$ and $\calB'$ on $\F_p^n$, we say $\calB'$ \emph{refines} $\calB$, denoted $\calB'\preceq \calB$, if the partition $\At(\calB')$ of $\F_p^n$ refines the partition $\At(\calB)$ of $\F_p^n$.  The following lemma says that given a quadratic factor, one can find a high rank refinement of bounded complexity (see  Lemma 3.11 in \cite{Green.2007}).

 \begin{lemma}[Rank Lemma]\label{lem:rank}
 Let $\rho$ be a growth function.  Suppose $\ell,q\geq 0$ are integers, and $\calB=(\calL,\calQ)$ is a quadratic factor on $\F_p^n$ of complexity $(\ell,q)$.  There there exist $0\leq q'\leq q$ and 
 $$
 \ell \leq \ell'\leq C=C(\ell,q,\rho),
 $$
 and a quadratic factor $\calB'\preceq \calB$ of complexity $(\ell',q')$ and rank at least $\rho(\ell'+q')$.
 \end{lemma}

The shape of the bound $C(\ell,q,\rho)$ in Lemma \ref{lem:rank} is important for this paper's main results, and will be considered  in detail in Section \ref{sec:rank}.

The setup outlined above differs slightly from that used in Prendiville's paper \cite{P}, the main result of which is crucial for us. We next set out enough definitions from \cite{P} to state the main result from \cite{P}, and then translate it into our setup. 

\begin{definition}
A \emph{quadratic polynomial on $G=\F_p^n$} is a map from $\F_p^n$ to $\F_p$ of the form $\psi(x)=\beta(x,x)+\phi(x)+c$, where $\beta:G^2\rightarrow \F_p$ is a symmetric bilinear from, $\phi:G\rightarrow \F_p$ is a linear form, and $c\in \F_p$ is a constant.
\end{definition}

As observed in \cite{P}, it is not difficult to show that any quadratic polynomial can be  built in the following way. 

\begin{fact}\label{fact:forms}
 Any quadratic polynomial $\psi:G=\F_p^n\rightarrow \F_p$ has the form $x^TMx+r\cdot x + c$ for some symmetric $n\times n$ matrix $M$ with entries from $\F_p$, some $r\in \F_p^n$, and some $c\in \F_p$.
\end{fact}
\begin{proof}
Say $\psi(x)=\beta(x,x)+\phi(x)+c$.  Let $e_1,\ldots, e_n$ denote the standard basis vectors in $\F_p^n$. Then let $M$ be the matrix with $ij$-th entry $\beta(e_i,e_j)$, and let $r$ be the vector $(\phi(e_1),\ldots, \phi(e_n))$.  It is easy to check $\psi(x)=x^TMx+r\cdot x+c$, and that $M$ is symmetric since $\beta$ is.
\end{proof}

In light of the fact above, we will from here on assume all quadratic polynomials have the form appearing in Fact \ref{fact:forms}.  The results in \cite{P} are stated in terms of preimages of maps built from tuples of quadratic polynomials.  We now give some definitions related to such tuples.

\begin{definition}
Let $d\geq 1$ be an integer, and for each $1\leq i\leq d$, let $\psi_i$ be a quadratic polynomial of the form $\psi_i(x)=x^TM_ix+r_i\cdot x+c_i$. The \emph{rank} of the tuple $\Psi=(\psi_1,\ldots, \psi_d)$ is the rank of the purely quadratic factor $\{M_1,\ldots, M_d\}$ (in the sense of Definition \ref{def:rank}).

Given $b=(b_1,\ldots, b_d)\in \F_p^d$, let 
$$
\Psi^{-1}(b)=\{x\in G: \psi_i(x)=b_i\text{ for each }i\in [d]\}.
$$
\end{definition}

We can now translate between this notation and ours.

\begin{fact}\label{fact:translatefactor}
Suppose $\ell,q\geq 0$ are integers satisfying $\ell+q\geq 1$, $\calB=(\calL,\calQ)$ is a quadratic factor of complexity $(\ell,q)$, and $(a,b)\in \F_p^{\ell}\times \F_p^q$.  Then there is a tuple $\Psi=(\psi_1,\ldots, \psi_{\ell+q})$ of quadratic polynomials, whose rank is equal to the rank of $\calB$, such that $\Psi^{-1}(0)=B(a,b)$.  
\end{fact}
\begin{proof} 
Suppose first $\ell$ and $q$ are both at least $1$. Let $\calL=\{r_1,\ldots, r_{\ell}\}$, $\calQ=\{M_1,\ldots, M_q\}$, and fix $a=(a_1,\ldots, a_{\ell})\in \F_p^{\ell}$ and $b=(b_1,\ldots, b_{q})\in \F_p^q$.  

For each $1\leq i\leq \ell$, set $\phi_i(x)=r_i\cdot x-a_i$, and for each $1\leq j\leq q$,  set $\psi_{\ell+j}(x)=x^TM_ix-b_i$.  By construction, $\Psi^{-1}(0)=B(a,b)$, and the rank of $Q$ is the rank of $\calQ$, which is by definition the rank of $\calB$. Thus we are done in the case where both $\ell$ and $q$ are positive.

If one of $\ell$ or $q$ is $0$, simply repeat the same argument omitting either the linear or quadratic maps.
\end{proof}

\subsection{Sizes of atoms and related sets} \label{ss:sizes}

In this subsection we collect several lemmas about sizes of sets related to quadratic factors.  The first such lemma tells us that atoms of a high rank factors all have about the same size (see Lemma 4.2 in \cite{Green.2007}). 

\begin{lemma}\label{lem:sizeofatoms}
Suppose $\ell,q\geq 0$ are integers, $\tau\in \mathbb{R}^{\geq 0}$, and $\calB=(\calL,\calQ)$ is a quadratic factor on $\F_p^n$ of complexity $(\ell,q)$ and rank at least $\tau$. Given any $e\in \F_p^{\ell}\times \F_p^q$,
$$
|B(e)|=(1\pm O(p^{\ell+q-\tau/2}))p^{n-\ell-q}.
$$
\end{lemma}
The next lemma tells us the approximate size of fibres of the map $\beta_{\calQ}$ from Definition \ref{def:betaq}. 

\begin{lemma}\label{lem:sizeofbeta}
Suppose $\ell,q\geq 0$ are integers, $\tau\in \mathbb{R}^{\geq 0}$, and $\calB=(\calL,\calQ)$ is a quadratic factor on $\F_p^n$ of complexity $(\ell,q)$ and rank at least $\tau$. Given any $d\in \F_p^q$,
$$
|\beta_{\calQ}^{-1}(d)|=(1\pm O(p^{q-\tau}))p^{2n-q}.
$$
\end{lemma}

We will repeatedly use the above size estimates in conjunction with a very general statement from \cite{Terry.2021d} (Lemma \ref{lem:A2} below).  It uses the following normalized indicator functions.

\begin{notation}\label{not:mu}
Given a group $G$ and $\Gamma\subseteq G\times G$, define $\mu_{\Gamma}:G\times G\rightarrow \mathbb{R}$  by setting $\mu_{\Gamma}(x,y)=\frac{|G|^2}{|\Gamma|}1_{\Gamma}(x,y)$ for all $(x,y)\in G^2$.
\end{notation}

The following is a slightly more general version of Lemma A.2 in \cite{Terry.2021d}.  It is straightforward to see this more general statement follows from exactly the same proof as that of Lemma A.2 in \cite{Terry.2021d} by simply adding the appropriate superscripts throughout.  

\begin{lemma}[Lemma A.2 of \cite{Terry.2021d}]\label{lem:A2}
Let $\tau\geq 0$ be a real, let $m\geq 1$ and $\ell,q\geq 0$ be integers, and suppose  $\calB=(\calL,\calQ)$ is a quadratic factor on $\F_p^n$ of complexity $(\ell,q)$ and rank at least $\tau$. For each $i\in [m]$, let $a_1^i, a^i_2, a^i_3\in \F_p^{\ell}\times \F_p^q$ and for each $1\leq i,j\leq m$, let $ b^{ij}_{12}, b^{ij}_{13}, b^{ij}_{23}\in \F_p^q$.   Then for all sets $I, J, K \subseteq [m]$,
\begin{align*}
\E_{\begin{subarray}{l}  x_i\in B(a^i_1)\\i\in I\end{subarray}}\E_{\begin{subarray}{l}y_j\in B(a^i_2)\\j\in J\end{subarray}}\E_{\begin{subarray}{l} z_k\in B(a^i_3)\\k\in K \end{subarray}}\prod_{i\in I, j\in J} &\mu_{\beta_{\calQ}^{-1}(b_{12}^{ij})}(x_i,y_j)\prod_{i\in I, k\in K}\mu_{\beta_{\calQ}^{-1}(b_{13}^{ij})}(x_i,z_k)\prod_{ j\in J, k\in K}\mu_{\beta_{\calQ}^{-1}(b_{23}^{ij})}(y_j,z_k)\\
&=1\pm O_m(p^{(\ell+q)(|I|+|J|+|K|)+q(|I||J|+|I||K|+|J||K|)-\tau/2}).
\end{align*}
The same holds whenever any instance of $\mu_{\beta_{\calQ}^{-1}(b^{ij}_{uv})}$ is replaced by the constant function 1.
\end{lemma}

Finally, at a certain juncture in the paper, we will be working with a quadratic factor $\calB=(\calL,\calQ)$, and need to consider an auxiliary factor of the form $(\calL',\calQ)$, where 
$$
\calL'=\calL\cup \{Mw_1, Mw_2,Mw_3,Mw_4: M\in \calQ\},
$$
for some group elements $w_1,w_2,w_3,w_4$.  When $\calL'$ is a linearly independent set, the pair $(\calL',\calQ)$ is indeed a quadratic factor in the sense of Definition \ref{def:quadfactor}.  In this case, we can estimate sizes of its atoms and related sets using the above lemmas.   On the other hand, for some choices of group elements $w_1,\ldots, w_4$, the resulting $\calL'$ will not be linearly independent. We will need to know the overall number of such ``bad" choices for $w_1,\ldots, w_4$ is small.  This is the purpose of the following lemma.

\begin{lemma}\label{lem:omegagood}
Let $r\geq 0$ be a real, let $\ell,q\geq 0$ be integers, and let $\calB=(\calL,\calQ)$ be a quadratic factor on $\F_p^n$ of complexity $(\ell,q)$ and rank at least $r$.  Then 
$$
|\{(w_1,w_2,w_3,w_4)\in G^4: \calL\cup \{Mw_1, Mw_2,Mw_3,Mw_4: M\in \calQ\}\text{ is not linearly independent}\}|
$$
is at most $14p^{4n+\ell+4q-r}$.
\end{lemma}

The proof of Lemma \ref{lem:omegagood} is straightforward and appears in the Appendix.

\section{Local $U^3$ norms and the local inverse theorem}\label{sec:inversenorms}

This section accomplishes three objectives.  First, in Subsection \ref{ss:P}, we state Prendiville's localized version of the $U^3$ inverse theorem in terms of a local $U^3$ norm implicit in that paper \cite{P}.  Second, Subsection \ref{ss:TW} reviews the definition of a different local $U^3$ norm from \cite{Terry.2021d}, and states a crucial fact, also proved in \cite{Terry.2021d}, that a set $A$ of bounded $\VC_2$-dimension must have density near $0$ or $1$ on any atom which is uniform  with respect to $A$ in this sense.  Finally, in Subsection \ref{ss:norms}, we prove the local norms from \cite{P} and \cite{Terry.2021d} are approximately equal. This allows us to deduce the necessary corollary about sets of bounded $\VC_2$-dimension in terms of the local norm from \cite{P}, as is required for the proof of our main theorem.

\subsection{Prendiville's inverse theorem for the $U^3$-norm localized to a quadratic atom}\label{ss:P}

In this section we state a localized version of the inverse theorem for the $U^3$-norm\footnote{Throughout this paper, we use the term ``norm" in a somewhat colloquial manner, as several of the objects we define yield only semi-norms.} due to Prendiville \cite{P}. Throughout the section we deal with functions $f:\F_p^n\rightarrow [-1,1]$, as opposed to  the more general functions $f:\F_p^n\rightarrow \mathbb{C}$  considered in Prendiville's paper.  We do this merely to ease notation and because this suffices for our purposes.   We do not attempt to fully contextualize the definition of the Gowers norms or the statement  of the inverse theorem. We refer the reader to the literature for more extensive background \cite{GT, P, Gowers.2001, Gowers.1998, Gowers.2024, Tao.2010}.

We begin by defining the usual (global)  versions of the Gowers $U^2$ and $U^3$ norms.  

\begin{definition}[$U^2$ and $U^3$-norms]\label{def:u3}
Suppose $G=\F_p^n$.  
\begin{enumerate}
\item Given a function $f:G\rightarrow [-1,1]$, the \emph{$U^2$-norm of $f$}, denoted $\|f\|_{U^2}$, is defined by the equation
$$
\|f\|_{U^2}^4=\sum_{x,h_2,h_2\in G}f(x)f(x+h_1)f(x+h_2)f(x+h_1+h_2).
$$
\item  Given a function $f:G\rightarrow [-1,1]$, the \emph{$U^3$-norm of $f$}, denoted $\|f\|_{U^3}$, is defined by the equation 
$$
\|f\|^8_{U^3}=\sum_{h\in G}\|\Delta_hf\|_{U^2}^4,
$$
where $\Delta_hf$ is the function defined by setting $\Delta_hf(x)=f(x+h)$. 
\end{enumerate}
\end{definition} 

It is well known the above define norms on the space of functions from $G$ into $[-1,1]$ (see e.g. Section 11 of \cite{Tao.2010}).  The following notation will be convenient for dealing with the $U^3$ norm.  

\begin{notation}\label{not:pif}
Given $f:G\rightarrow [-1,1]$ and $(x,h_1,h_2,h_3)\in G^4$, let 
\begin{align*}
\pi_f(x,h_1,h_2,h_3)=&f(x)f(x+h_1)f(x+h_2)f(x+h_3)\cdot\\
&f(x+h_1+h_2)f(x+h_1+h_3)f(x+h_2+h_3)f(x+h_1+h_2+h_3).
\end{align*}
\end{notation}

This notation allows us to rewrite the $U^3$ norm as follows.

\begin{observation}
Let $G=\F_p^n$. Given a function $f:G\rightarrow [-1,1]$, 
$$
\|f\|^8_{U^3}=\sum_{(x,h_1,h_2,h_3)\in G^4}\pi_f(x,h_1,h_2,h_3).
$$
\end{observation}

The inverse theorem for the $U^3$ norm says a function with large $U^3$ has some correlation with a quadratic phase function.  For the version stated below, see Corollary 1.2 of \cite{Gowers.2024}.  

\begin{theorem}
\label{thm:globalinverse}[Global $U^3$ inverse theorem in $\F_p^n$ with polynomial bounds \cite{Gowers.2024}]
    Let  $\delta\in (0,1)$, and suppose $f:G=\F_p^n\to [-1,1]$  satisfies $\|f\|_{U^3}\geq \d p^{4n}$. Then there exists a quadratic polynomial $q:G\to \F_p$ such that
    $$\left|\sum_{x\in G}f(x)e^{2\pi iq(x)/p}\right|\gg_p \delta^{O_p(1)}p^{n}.
    $$
\end{theorem}

In \cite{P}, Prendiville proved a version of this inverse theorem localized to a single quadratic atom.  This theorem will play a crucial role in our main result.  

\begin{theorem}
\label{thm:inverse}[Local $U^3$ inverse theorem with polynomial bounds \cite{P}]
    Let $\delta\in (0,1)$, let $d\geq 1$ be an integer, and let $\Psi = (\psi_1, \ldots, \psi_d)$ be a $d$-tuple of quadratic polynomials on $G$.  Suppose $f:G\to [-1,1]$ is a function with support contained in $B = \Psi^{-1}(0)$ satisfying $\|f\|_{U^3}\geq \d\|1_B\|_{U^3}$. Then one of the following hold.
    \begin{itemize}
    \item there exists a quadratic polynomial $\psi:G\to \F_p$ such that 
    $$
    \left|\sum_{x\in B}f(x)e^{2\pi i \psi(x)/p}\right|\gg_p \delta^{O_p(1)}|B|;
    $$
    \item $\Psi$ has rank at most $O_p(d+\log (2/\delta))$.
    \end{itemize}
\end{theorem}

Our next goal is to rephrase Theorem \ref{thm:inverse}  in terms of the quadratic factors defined in Subsection \ref{ss:factors}, and in terms of a local $U^3$ norm.  

First, we define a version of the $U^3$ norm localized to a quadratic atom.  We decorate this norm with a $P$ to denote it is implicit in Prendiville's paper, and to distinguish it from a different local $U^3$-norm (due to the second author and Wolf) which we will also use in this paper (see Definition \ref{def:TW}). 

\begin{definition}[Local $U^3$-norm: Prendiville version]\label{def:P}
Suppose $\ell,q\geq 0$ are integers, and $\calB=(\calL,\calQ)$ is a quadratic factor of complexity $(\ell,q)$ on $G=\F_p^n$. Given a function $f:G\rightarrow [-1,1]$ and a tuple  $b\in \F_p^{\ell}\times \F_p^q$, define 
$$
\|f\|^P_{U^3(b)}=\begin{cases} \frac{\|f\cdot 1_{B(b)}\|_{U^3}}{\|1_{B(b)}\|_{U^3}}&\text{ if }\|f\cdot 1_{B(b)}\|_{U^3}\neq 0,\text{ and }\\
0&\text{ if }\|f\cdot 1_{B(b)}\|_{U^3}=0.\end{cases}
$$
\end{definition}
We note that in Definition \ref{def:P}, we have suppressed the dependence on the factor $\calB$ in the notation $\|\cdot \|_{U^3(b)}^P$.  When we use this notation, the factor in question will be clear from context. We now  restate Theorem \ref{thm:inverse}.  

\begin{corollary}
\label{cor:inverse}
    There exist a constant $C=C(p)>0$ so that the following holds.  Let $\delta\in (0,1)$ and let $n\geq 1$ and $\ell,q\geq 0$ be integers.
    
    Suppose $\calB=(\calL,\calQ)$ is a quadratic factor on $G=\F_p^n$ of complexity $(\ell,q)$, $f:\F_p^n\to [-1,1]$ is a function, and $b\in \F_p^{\ell}\times \F_p^q$.  If   $\|f\|^P_{U^3(b)}\geq \d$, then one of the following hold.
    \begin{itemize}
    \item there exists a quadratic polynomial $\psi:G\to \F_p$ such that 
    $$\left|\sum_{x\in B(b)}f(x)e^{2\pi i\psi(x)/p}\right|\geq C^{-1} \delta^{C}|B(b)|;$$
    \item $\calB$ has rank at most $C(\ell+q+\log (2/\delta))$.
    \end{itemize}
\end{corollary}
\begin{proof}
Let $C=C(p)$ be sufficiently large based on the $\gg_p$ and $O_p(1)$ appearing in Theorems \ref{thm:globalinverse} and \ref{thm:inverse}.  Fix $\delta$, $n$, $\ell$, $q$, and $\calB$ as in the hypotheses. Suppose $f:G=\F_p^n\rightarrow [-1,1]$ and $b\in \F_p^{\ell}\times\F_p^{q}$ are such that $\|f\|^P_{U^3(b)}\geq \d$.  

Suppose first $\ell=q=0$. In this case, $b=(0,0)\in \F_p^0\times \F_p^0$, $B(b)=\F_p^n$, and the hypothesis $\|f\|^P_{U^3(b)}\geq \d$ translates into saying the global $U^3$ norm is large, specifically, $\|f\|_{U^3}\geq \delta p^{4n}$.    The desired conclusion then follows immediately from the global $U^3$ inverse theorem, Theorem \ref{thm:globalinverse}.   

Suppose now $\min\{\ell,q\}\geq 1$.  By Fact \ref{fact:translatefactor}, there is a tuple of quadratic polynomials $\Psi=(\psi_1,\ldots, \psi_{\ell+q})$ with the same rank as $\calB$, such that $\Psi^{-1}(0)=B(b)$.  By Definition \ref{def:P}, the hypothesis $\|f\|^P_{U^3(b)}\geq \delta$ means $\|f\cdot 1_{B(b)}\|_{U^3}\geq \d\|1_{B(b)}\|_{U^3}$. Consequently, applying  Theorem \ref{thm:inverse} to $\Psi$ and the function $f\cdot 1_{B(b)}$, we have that either there exists a quadratic polynomial $\psi:G\to \F_p$ such that 
    $$\left|\sum_{x\in B(b)}f(x)e^{2\pi i\psi(x)/p}\right|\geq C^{-1} \delta^{C}|B(b)|,$$
    or $\Psi$ has rank at most $C(\ell+q+\log (2/\delta))$.  Since the rank of $\Psi$ is the rank of $\calB$, we are done.
\end{proof}

We will end this subsection with some observations about Definition \ref{def:P} that we will need in Subsection \ref{ss:norms}. For these remarks, it will be useful to have the following definition for the set of $4$-tuples which yield a non-zero term in the sum appearing in $\|1_{B(b)}\|^8_{U^3}$.  

 \begin{definition}\label{def:omegab}
Suppose $\calB$ is a quadratic factor on $G=\F_p^n$ and $B\in \At(\calB)$.  Define 
\begin{align*}
\Omega_B=\Big\{&(x,h_1,h_2,h_3)\in G^4:\\
&1_B(x)\Big(\prod_{i=1}^31_B(x+h_i)\Big)\Big(\prod_{i\neq j\in [3]} 1_B(x+h_i+h_j)\Big)1_B(x+h_1+h_2+h_3)\neq 0\Big\}.
\end{align*}
 \end{definition}

Observe that in the notation of Definitions \ref{def:omegab} and Definition \ref{def:u3},   $|\Omega_B|=\|1_B\|_{U^3 }^8$.  Using this observation, it is easy to show $\|f\|_{U^3(b)}^P$ from Definition \ref{def:P} is always bounded by $1$.  

\begin{observation}\label{ob:1bound}
Suppose $\ell,q\geq 0$ are integers, and $\calB=(\calL,\calQ)$ is a quadratic factor of complexity $(\ell,q)$ on $G=\F_p^n$. For any $f:G\rightarrow [-1,1]$ and $b\in \F_p^{\ell}\times \F_p^q$, $\|f\|^P_{U^3(b)}\leq 1$.
\end{observation}
\begin{proof}
Since $f$ is $1$-bounded, it is easy to see 
$$
\|f\cdot 1_{B(b)}\|_{U^3}^8=\sum_{(x,h_1,h_2,h_3)\in \Omega_{B(b)}}\pi_f(x,h_1,h_2,h_3)\leq |\Omega_{B(b)}|=\|1_{B(b)}\|_{U^3}^8.
$$
Thus, $\|f\cdot 1_{B(b)}\|_{U^3}\leq \|  1_{B(b)}\|_{U^3}$, so by definition,  $\|f\|_{U^3(b)}^P\leq 1$.
\end{proof}

The fact that $|\Omega_B|=\|1_B\|_{U^3 }^8$ also allows us to  rewrite Definition \ref{def:P} as follows (when the factor in question has sufficiently high rank).

\begin{observation}\label{ob:P}
There exists a constant $C=C(p)>0$ so that the following holds. Suppose $\calB=(\calL,\calQ)$ is a quadratic factor on $G=\F_p^n$ of complexity $(\ell,q)$ and rank at least $C(\ell+q)$.  For any function $f:\F_p^n\to [-1,1]$ and $b\in \F_p^{\ell}\times \F_p^q$, we have 
$$
\left(\|f\|_{U^3(b)}^P\right)^8=\frac{\|f\cdot 1_{B(b)}\|_{U^3}^8}{|\Omega_{B(b)}|}.
$$
\end{observation}

Observation \ref{ob:P} follows immediately from the definitions of $\Omega_{B(b)}$ and $\|\cdot \|_{U^3(b)}^P$, and the fact that when $\calB$ has sufficiently large rank, every atom of $\calB$ will be nonempty (by Lemma \ref{lem:sizeofatoms}).  

In light of Observation \ref{ob:P}, we will need to estimate the size of   $\Omega_{B(b)}$ in order to fully understand $\| f\|_{U^3(b)}^P$. For this we will use the following lemma.  

\begin{lemma}\label{lem:constraints}
Suppose $\ell,q\geq 0$ are integers,  $\calB=(\calL,\calQ)$ is a quadratic factor on $G=\F_p^n$ of complexity $(\ell,q)$, and $B\in \At(\calB)$.  Then for any $(x,h_1,h_2,h_3)\in G^4$, the following are equivalent.
\begin{enumerate}
\item  $(x,h_1,h_2,h_3)\in \Omega_B$,
\item the following hold: $x\in B$, $h_i\in L(0)$ for all $i\in [3]$,  $2\beta_{\calQ}(x,h_i)=-\beta_{\calQ}(h_i,h_i)$ for all $i\in [3]$,  and $\beta_{\calQ}(h_i,h_j)=0$ for all $i\neq j\in [3]$.
\end{enumerate}
\end{lemma}

The proof of Lemma \ref{lem:constraints} is straightforward and appears in the appendix. Using Lemma \ref{lem:constraints}, we can now estimate the size of the sets of the form $\Omega_B$.

\begin{lemma}\label{lem:omegasize}
There is a constant $C>0$ so that the following holds. Suppose $\ell,q\geq 0$ are integers and $\calB=(\calL,\calQ)$ is a quadratic factor on $G=\F_p^n$ of complexity $(\ell,q)$ and rank at least $C(\ell+q+\log_p(\e))$. Then for all   $B\in \At(\calB)$,
$$
|\Omega_B|=(1\pm \e) p^{4n-4\ell-7q}.
$$ 
\end{lemma}
\begin{proof}
 Let $C$ be sufficiently large. By Lemma \ref{lem:constraints}, 
\begin{align*}
|\Omega_B|&=\sum_{b_1,b_2,b_3\in \F_p^q}\sum_{h_1\in B(0,b_1)}\sum_{h_2\in B(0,b_2)}\sum_{h_3\in B(0,b_3)}\sum_{x \in B}\\
&1_{\beta_{\calQ}^{-1}(-b_1)}(x,h_1)1_{\beta_{\calQ}^{-1}(-b_2)}(x,h_2)1_{\beta_{\calQ}^{-1}(-b_3)}(x,h_3)1_{\beta_{\calQ}^{-1}(0)}(h_1,h_2)1_{\beta_{\calQ}^{-1}(0)}(h_1,h_3)1_{\beta_{\calQ}^{-1}(0)}(h_2,h_3).
\end{align*}
For any $b_1,b_2,b_3\in \F_p^q$, we have by Lemma \ref{lem:A2},
\begin{align*}
&\sum_{h_1\in B(0,b_1)}\sum_{h_2\in B(0,b_2)}\sum_{h_3\in B(0,b_3)}\sum_{x \in B}\\
&1_{\beta_{\calQ}^{-1}(-b_1)}(x,h_1)1_{\beta_{\calQ}^{-1}(-b_2)}(x,h_2)1_{\beta_{\calQ}^{-1}(-b_3)}(x,h_3)1_{\beta_{\calQ}^{-1}(0)}(h_1,h_2)1_{\beta_{\calQ}^{-1}(0)}(h_1,h_3)1_{\beta_{\calQ}^{-1}(0)}(h_2,h_3)\\
&=(1\pm \e)p^{4n-4\ell-10q}.
\end{align*}
Combining these equations yields that $|\Omega_B|=\sum_{b_1,b_2,b_3\in \F_p^q}(1\pm \e)p^{4n-4\ell-10q}=(1\pm \e)p^{4n-4\ell-7q}$.
\end{proof}

\subsection{The local $U^3$-norms of \cite{Terry.2024d}}\label{ss:TW}

In this subsection we introduce local $U^3$ norms defined by the second author and Wolf in \cite{Terry.2024d}, after which we state a crucial result, also from \cite{Terry.2024d}, which gives a sufficient condition for a set of bounded $\VC_2$-dimension to have density near $0$ or $1$ on an atom of a quadratic factor  in terms of this local norm (see Theorem \ref{thm:main}). Before we define these local norms, we require several definitions.  

First, we introduce notation aimed at computing, given a triple $(x,y,z)\in G^3$, which atom of a quadratic factor $\calB=(\calL,\calQ)$ contains the sum $x+y+z$.  It turns out that where $x+y+z$ lands depends on the labels associated to $x$, $y$, and $z$ (i.e. $\beta_{\calB}(x)$, $\beta_{\calB}(y)$, and $\beta_{\calB}(z)$ from Definition \ref{def:betab}), as well as the values of the bilinear form on the pairs (i.e. $\beta_{\calQ}(x,y)$, $\beta_{\calQ}(x,y)$ and $\beta_{\calQ}(y,z)$ from Definition \ref{def:betaq}). This is the motivation for the following notation, which considers the set of triples carrying a fixed tuple of these labels.  

\begin{definition}\label{def:sigma}
Suppose $\ell,q\geq 0$ are integers and $\calB=(\calL,\calQ)$ is a quadratic factor on $G=\F_p^n$ of complexity $(\ell,q)$. Given  tuples
$$
d_1=(d_a,d_b,d_c)\in (\F_p^{\ell}\times \F_p^q)^3\text{ and } d_2=(d_{ab},d_{ac},d_{bc})\in (\F_p^q)^3,
$$
define
\begin{align*}
K_{1,1,1}(d_1,d_2)=\{(x,y,z)\in G^3&:  \beta_{\calB}(x)=d_a, \beta_{\calB}(y)=d_b, \beta_{\calB}(z)=d_c,\text{ and } \\
&\beta_{\calQ}(x,y)=d_{ab}, \beta_{\calQ}(x,z)=d_{ac},\beta_{\calQ}(y,z)=d_{bc}\}.
\end{align*}
\end{definition}
The notation $K_{1,1,1}(d_1,d_2)$ is meant to remind the reader that $K_{1,1,1}(d_1,d_2)$  consists of copies of $K_{1,1,1}$ (the complete tripartite graphs with parts of size $1$) in an auxiliary graph which depends on the tuple $(d_1,d_2)$.   It is clear that the sets of the form $K_{1,1,1}(d_1,d_2)$ partition $G^3$, and it turns out that triples in the same piece of this partition always sum into the same atom. To make this precise, we will use the following notation.

\begin{definition}\label{def:sigma1}
Suppose $\ell,q\geq 0$ are integers and $\calB=(\calL,\calQ)$ is a quadratic factor of complexity $(\ell,q)$. Given  tuples
$$
d_1=(d_a,d_b,d_c)\in (\F_p^{\ell}\times \F_p^q)^3\text{ and } d_2=(d_{ab},d_{ac},d_{bc})\in (\F_p^q)^3,
$$
define
\begin{align*} 
\Sigma(d_1,d_2)=d_a+d_b+d_c+2(0d_{ab})+2(0d_{ac})+2(0d_{bc}),
\end{align*}
where for each $xy\in \{ab,ac,bc\}$, $0d_{xy}$ denotes the tuple obtained by adjoining $\ell$ many $0$'s to the left of $d_{xy}$.  
\end{definition}

\begin{observation}\label{ob:sigma1}
In the notation of Definition \ref{def:sigma1}, we have that $(x,y,z)\in K_{1,1,1}(d_1,d_2)$ implies $x+y+z\in B(\Sigma(d_1,d_2))$. 
\end{observation}

We now define the local norm from \cite{Terry.2024d}, which we decorate with a  TW to distinguish it from Definition \ref{def:P}. We note that unlike Definition \ref{def:P}, this norm is not localized using the label of a single specific atom, but instead a tuple of labels $d=(d_1,d_2)$ as in Definition \ref{def:sigma} (see Notation  \ref{not:mu} for the definition of the normalized indicator functions used below).

\begin{definition}[Local $U^3$ norm: Terry-Wolf version]\label{def:TW}
Suppose $\ell,q\geq 0$ are integers, $\calB=(\calL,\calQ)$ is a quadratic factor on $G=\F_p^n$ of complexity $(\ell,q)$,  and
$$
d_1=(d_a,d_b,d_c)\in (\F_p^{\ell}\times \F_p^q)^3\text{ and } d_2=(d_{ab},d_{ac},d_{bc})\in (\F_p^q)^3.
$$
Then for $d=(d_1,d_2)$ and $f:G\rightarrow [-1,1]$, define $\|f\|_{U^3(d)}^{TW}$ by the following equation.
\begin{align*}
\Big(\|f\|_{U^3(d)}^{TW}\Big)^8&= \E_{x_0,x_1\in B(d_a)}\E_{y_0,y_1\in B(d_b)}\E_{z_0,z_1\in B(d_c)}\\
&\prod_{(i,j)\in [2]^2}\mu_{\beta_{\calQ}^{-1}(d_{ab})}(x_i,y_j)\mu_{\beta_{\calQ}^{-1}(d_{ac})}(x_i,z_j)\mu_{\beta_{\calQ}^{-1}(d_{bc})}(y_i,z_j) \prod_{(i,j,k)\in [2]^3}f(x_i+y_j+z_k).
\end{align*}
\end{definition}

The second author and Wolf proved a $U^3$ regularity lemma in terms of these local  norms (Proposition 4.2 of \cite{Terry.2021a}), in addition to a corresponding counting lemma (Propositions 3.19 and D.11 in \cite{Terry.2021d}).\footnote{These results were originally proved in \cite{Terry.2021d}, but were recently separated into their own paper \cite{Terry.2021d} with reworked proofs and a new emphasis on the local norms.}  For the purposes of this paper, the following corollary of said counting lemma is key. Specifically, Theorem \ref{thm:key} below gives a sufficient condition for when a set of bounded $\VC_2$-dimension has density near $0$ or $1$ on a quadratic atom, in terms of this type of local $U^3$ norm.

\begin{theorem}[Proposition 4.3 of \cite{Terry.2021d}]\label{thm:key}
For all integers $k\geq 1$ there is a constant $C>0$ so that the following holds. Suppose $\e\in (0,1)$, $\ell,q\geq 0$ are integers\footnote{The $\ell$ and $q$ in \cite{Terry.2021d} are implicitly assumed to be at least $1$, but it is not difficult to see the proofs extend to the cases where one or both are $0$.}, and $\calB=(\calL,\calQ)$ is a quadratic factor on $G=\F_p^n$ of complexity $(\ell,q)$ and rank\footnote{An inspection of the proofs in \cite{Terry.2021d} shows such a rank assumption suffices.} at least $C(\ell+q+\log_p(\e^{-1}))$.  Assume $A\subseteq \F_p^n$ satisfies $\VC_2(A)\leq k$.  

Let $b\in \F_p^{\ell}\times \F_p^{q}$, and let $\alpha_{B(b)}$ denote the density of $A$ on the atom $B(b)$.  If there exists some $d\in (\F_p^{\ell}\times \F_p^q)^3\times (\F_p^q)^6$ such that $\Sigma(d)=b$ and $\|1_A-\alpha_{B(b)}\|_{U^3(d)}^{TW}<(\e/4)^{m^22^{m^2}}$, then  $\alpha_{B(b)}\in [0,\e)\cup (1-\e,1]$.
\end{theorem}

To use Theorem \ref{thm:key}, we will need to relate the norm of Definition \ref{def:TW} with that from Definition \ref{def:P}. We do this in the next subsection.  Preparing for that, it will  be convenient  to have an expression for $\|f\|_{U^3(d)}^{TW}$ which hides the normalization. For this we will us the following notation.

\begin{definition}\label{def:k222labels}
Let $\ell,q\geq 0$ be integers, and let $\calB=(\calL,\calQ)$ be a quadratic factor on $G=\F_p^n$ of complexity $(\ell,q)$. Given tuples 
$$
d_1=(d_a,d_b,d_c)\in (\F_p^{\ell}\times \F_p^q)^3,\text{ and }(d_{ab},d_{ac},d_{bc})\in   (\F_p^q)^3,
$$
define 
\begin{align*}
K_{2,2,2}(d_1,d_2)=&\{(x_1,x_2,y_1,y_2,z_1,z_2)\in G^6:\\
&\text{ for each $i\in [2]$, $\beta_{\calB}(x_i)=d_a$, $\beta_{\calB}(y_i)=d_b$, $\beta_{\calB}(z_i)=d_c$, and }\\
&\text{ for each $(i,j)\in [2]^2$, $\beta_{\calQ}(x_i,y_j)=d_{ab}$, $\beta_{\calQ}(x_i,z_j)=d_{ac}$, $\beta_{\calQ}(y_i,z_j)=d_{bc}$}\}.
\end{align*}
\end{definition}

We use the notation $K_{2,2,2}(d_1,d_2)$ to indicate this  consists of copies of $K_{2,2,2}$ (the complete tripartite graph with parts of size $2$) in a certain auxiliary graph depending on the tuple $(d_1,d_2)$.  In the case  where the factor $\calB$ has high rank, we can express  $\|\cdot \|_{U^3(d)}^{TW}$ as approximately a sum over tuples from $K_{2,2,2}(d_1,d_2)$.

\begin{proposition}\label{prop:rewritetw}
There is a constant $C=C(p)$ so that the following holds.  Suppose $\e \in (0,1)$, $\ell,q\geq 0$ are integers,  $\calB=(\calL,\calQ)$ is a quadratic factor on $G=\F_p^n$ of complexity $(\ell,q)$ of rank at least $C(\ell+q+\log_p(\e^{-1}))$, and
$$
d_1=(d_a,d_b,d_c)\in (\F_p^{\ell}\times \F_p^q)^3\text{ and } d_2=(d_{ab},d_{ac},d_{bc})\in (\F_p^q)^3.
$$
Then for $d=(d_1,d_2)$ and $f:G\rightarrow [-1,1]$,
\begin{align*}
\Big(\|f\|_{U^3(d)}^{TW}\Big)^8&=(1\pm \e)p^{-6n+6\ell+18q}\sum_{(x_1,x_2,y_1,y_2,z_1,z_2)\in K_{2,2,2}(d)}\prod_{(i,j,k)\in [2]^3}f(x_i+y_j+z_k).
\end{align*}
\end{proposition}
\begin{proof}
By definition of $\|f\|_{U^3(d)}^{TW}$ and $K_{2,2,2}(d)$, 
\begin{align*}
\Big(\|f\|_{U^3(d)}^{TW}\Big)^8=\left(\frac{1}{|B(d_a)||B(d_b)||B(d_c)|}\right)^2\left(\frac{1}{|\beta_{\calQ}^{-1}(d_{ab})|\beta_{\calQ}^{-1}(d_{ac})||\beta_{\calQ}^{-1}(d_{bc})|}\right)^4&\\
\sum_{(x_1,x_2,y_1,y_2,z_1,z_2)\in K_{2,2,2}(d)}\prod_{(i,j,k)\in [2]^3}f(x_i+y_j+z_k)&.
\end{align*}
Combining this with Lemmas \ref{lem:sizeofatoms} and \ref{lem:sizeofbeta}, it is not difficult to see that if $\calB$ has rank at least $C(\ell+q+\log_p(\e^{-1}))$ for a sufficiently large $C$, the above will imply 
\begin{align*}
\Big(\|f\|_{U^3(d)}^{TW}\Big)^8&=(1\pm \e)p^{-6n+6\ell+18q}\sum_{(x_1,x_2,y_1,y_2,z_1,z_2)\in K_{2,2,2}(d)}\prod_{(i,j,k)\in [2]^3}f(x_i+y_j+z_k).
\end{align*}
\end{proof}

\subsection{Relating the local $U^3$-norms}\label{ss:norms}
The goal of this subsection is to show that when the factor in question has high rank, the local norm $\|\cdot \|^P_{U^3(b)}$ of Definition \ref{def:P} is approximately equal to the local norm $\|\cdot \|_{U^3(d)}^{TW}$ of Definition \ref{def:TW}, for any tuple $d$ satisfying $\Sigma(d)=b$ (see Definition \ref{def:sigma1}).  Specifically, we will prove the following.

\begin{theorem}\label{thm:relatenorms}
There is a constant $C>0$ so that the following holds. Let $\e\in (0,1)$, let $\ell,q\geq 0$ be integers, and let $\calB=(\calL,\calQ)$ be a quadratic factor on $G=\F_p^n$ of complexity $(\ell,q)$ and rank at least $C(\ell+q+\log_p(\e^{-1}))$. 

For any $b\in \F_p^{\ell}\times \F_p^q$ and $d\in (\F_p^{\ell+q})^3\times (\F_p^q)^{3}$ such that $\Sigma(d)=b$, and any function $f:G\rightarrow [-1,1]$, 
\begin{align*}
&\Big(\|f\|_{U^3(d)}^{TW}\Big)^8= \Big(\|f\|_{U^3(b)}^P\Big)^{8}\pm \e.
\end{align*}
\end{theorem}

We start by reviewing a well known reformulation  of the global $U^3$ norm, which uses a change of variables to rewrite the  norm as  a sum over $6$-tuples, rather than $4$-tuples  (Fact \ref{fact:rewritenorm} below).   We first set notation for a change of variables function, which takes a $6$-tuple of group elements to a $4$-tuple.  

\begin{definition}\label{def:psi}
Let $G=\F_p^n$. Define $\Psi:G^6\rightarrow G^4$ by setting
$$
\Psi(x_1,x_2,y_1,y_2,z_1,z_2)=(x_1+y_1+z_1,x_2-x_1,y_2-y_1,z_2-z_1).
$$
\end{definition}

\begin{observation}\label{ob:psi}
The function $\Psi$ in Definition \ref{def:psi} is a surjection, all of whose fibres have size $|G|^2$.  Moreover, for any function $f:G\rightarrow [-1,1]$ and any $(x_1,x_2,y_1,y_2,z_1,z_2)\in G^6$, if $\Psi(x_1,x_2,y_1,y_2,z_1,z_2)=(x,h_1,h_2,h_3)$, then the following holds (recall Notation \ref{not:pif}). 
\begin{align*}
 \prod_{(i,j,k)\in [2]^3}f(x_i+y_j+z_k)=\pi_f(x,h_1,h_2,h_3).
\end{align*}
\end{observation}
\begin{proof}
Fix $(x,h_1,h_2,h_3)\in G^4$. Choose any $(x_1,y_1,z_1)\in G^3$ satisfying $x=x_1+y_1+z_1$, and let $x_2=x_1+h_1$, $y_2=y_1+h_2$, and $z_2=z_1+h_3$ (there are $|G|^2$ ways to do this).  It is an easy exercise to check that $(x_1,x_2,y_1,y_2,z_1,z_2)\in \Psi^{-1}(x,h_1,h_2,h_3)$, and moreover, different choices of such $(x_1,y_1,z_1)$ give rise to distinct elements of $\Psi^{-1}(x,h_1,h_2,h_3)$.  On the other hand, by definition, any element of $\Psi^{-1}(x,h_1,h_2,h_3)$  arises in this way.  This shows $|\Psi^{-1}(x,h_1,h_2,h_3)|=|G|^2$.  The moreover statement regarding products is obvious from the definition of $\Psi$.
\end{proof}

Observation \ref{ob:psi} immediately implies the following well known reformulation of the  global $U^3$ norm from Definition \ref{def:u3}.

\begin{fact}\label{fact:rewritenorm}
For any function $f:G=\F_p^n\rightarrow [-1,1]$, 
\begin{align*}
\|f\|_{U^3}^8=p^{-2n} \sum_{x_0,x_1\in G}\sum_{y_0,y_1\in G}\sum_{z_0,z_1\in G} \prod_{(i,j,k)\in [2]^3}f(x_i+y_j+z_k).
\end{align*}
\end{fact}

The change of variable formula $\Psi$ will play a crucial role in relating the local norms of Definitions \ref{def:P} and \ref{def:TW}.  Specifically, we will need to analyze the behavior of $\Psi$ on subsets of $G^6$ of the form $K_{2,2,2}(d)$ (see Definition \ref{def:k222labels}). Our first observation in this direction is basically immediate from the relevant definitions.    

\begin{observation}\label{ob:k222}
Let $\ell,q\geq 0$ be integers, and let $\calB=(\calL,\calQ)$ be a quadratic factor on $G=\F_p^n$ of complexity $(\ell,q)$.  Given $d\in (\F_p^{\ell}\times \F_p^{q})^3\times  (\F_p^q)^3$, $K_{2,2,2}(d)\subseteq \Psi^{-1}(\Omega_{B(\Sigma(d))})$.
\end{observation}
\begin{proof}
Fix $(x_1,x_2,y_1,y_2,z_1,z_2)\in K_{2,2,2}(d)$.  By Observation \ref{ob:sigma1},  we have that for each $(i,j,k)\in [2]^3$, $x_i+y_j+z_k\in B(\Sigma(d))$.  Combining this with Observation \ref{ob:psi}, we have
$$
1= \prod_{(i,j,k)\in [2]^3}1_{B(\Sigma(d))}(x_i+y_j+z_k)=\pi_{1_{B(\Sigma(d))}}(\Psi(x_1,x_2,y_1,y_2,z_1,z_2)).
$$
Thus by definition, $\Psi(x_1,x_2,y_1,y_2,z_1,z_2)\in \Omega_{B(\Sigma(d))}$, i.e. $(x_1,x_2,y_1,y_2,z_1,z_2)\in \Psi^{-1}(\Omega_{B(\Sigma(d))})$.
\end{proof}

The above observation shows $K_{2,2,2}(d)$ is always covered by $\Psi^{-1}(\Omega_{B(\Sigma(d))})$, which   in turn can be written as a union of fibres
$$
\Psi^{-1}(\Omega_{B(\Sigma(d))})=\bigcup_{z\in \Omega_{B(\Sigma(d))}}\Psi^{-1}(z).
$$
Our next task involves understanding how the set $K_{2,2,2}(d)$ is distributed across the individual fibres in this union. Our first step in this direction is to give an algebraic description of the intersection of $K_{2,2,2}(d)$ with  such a fibre.  We note the proof is straightforward, although a bit tedious.

\begin{proposition}\label{prop:preimage}
There is a constant $C>0$ so that the following holds.  Let $\ell,q\geq 0$ be integers, and let $\calB=(\calL,\calQ)$ be a factor on $G=\F_p^n$  complexity $(\ell,q)$ and rank at least $C(\ell+q+\log_p(\e))$.  

Fix $e\in \F_p^{\ell}\times \F_p^q$ and $d=(d_a,d_b,d_c,d_{ab},d_{ac},d_{bc})\in (\F_p^{\ell}\times \F_p^q)^3\times (\F_p^q)^{3}$ such that $\Sigma(d)=e$, and let $B=B(e)$.  Given   $(w,h_a,h_b,h_c)\in \Omega_{B}$,  we have
\begin{align*}
&\Psi^{-1}(w,h_a,h_b,h_c)\cap K_{2,2,2}(d)\\
&=\{(x,x+h_a, y, y+h_b, w-x-y, w-x-y+h_c):x\in X, y\in Y,\text{ and } \beta_{\calQ}(x,y)=d_{ab}\},
\end{align*}
where 
\begin{align*}
X=\{x\in B(d_a): \beta_{\calQ}(x,h_b)=\beta_{\calQ}(x,h_c)=0, 2\beta_{\calQ}(x,h_a)=-\beta_{\calQ}(h_a,h_a),\beta_{\calQ}(x,w)=d_a+d_{ac}+ d_{ab}\},
\end{align*}
and 
\begin{align*}
Y=\{y\in B(d_b): \beta_{\calQ}(y,h_a)=\beta_{\calQ}(y,h_c)=0, 2\beta_{\calQ}(y,h_b)=-\beta_{\calQ}(h_b,h_b),\beta_{\calQ}(y,w)=d_b+d_{bc}+d_{ab}\}.
\end{align*}
\end{proposition}
\begin{proof}
We first show the $\subseteq$ direction of the desired equality. Fix an element $(x_1,x_2,y_1,y_2,z_1,z_2)$ of $\Psi^{-1}(w,h_a,h_b,h_c)\cap K_{2,2,2}(d)$. Since $\Psi(x_1,x_2,y_1,y_2,z_1,z_2)=(w,h_a,h_b,h_c)$, we know $w=x_1+y_1+z_1$, $x_2=x_1+h_a$, $y_2=y_1+h_b$, and $z_2=z_1+h_c$.  These equalities imply
$$
(x_1,x_2,y_1,y_2,z_1,z_2)=(x_1,x_1+h_a, y_1, y_1+h_b, w-x_1-y_1, w-x_1-y_1+h_c).
$$
Since $(x_1,x_2,y_1,y_2,z_1,z_2) \in K_{2,2,2}(d)$, we already know $\beta_{\calQ}(x_1,y_1)=d_{ab}$. Thus, it suffices to show $x\in X$ and $y\in Y$.  We first show $x\in X$. Since $(x_1,x_2,y_1,y_2,z_1,z_2)\in K_{2,2,2}(d)$, we know $x_1\in B(d_a)$.  We also know from this that $x_2\in B(d_a)$. Since $x_2=x_1+h_2$, this implies
$$
d_a=\beta_{\calQ}(x_1+h_a,x_1+h_a)=d_a+2\beta_{\calQ}(x_1,h_a)+ \beta_{\calQ}(h_a,h_a).
$$
This implies $2\beta_{\calQ}(x_1,h_a)=- \beta_{\calQ}(h_a,h_a)$. On the other hand, we know $d_{ab}=\beta_{\calQ}(x_1,y_2)=\beta_{\calQ}(x_1,y_1)$, and $y_2=y_1+h_b$. These imply 
$$
d_{ab}=\beta_{\calQ}(x_1,y_2)=\beta_{\calQ}(x_1,y_1+h_b)=d_{ab}+\beta_{\calQ}(x_1,h_b),
$$
which implies $\beta_{\calQ}(x_1,h_b)=0$.  Similarly, we have
$$
d_{ac}=\beta_{\calQ}(x_1,z_2)=\beta_{\calQ}(x_1,z_1+h_c)=d_{ab}+\beta_{\calQ}(x_1,h_c),
$$
which implies $\beta_{\calQ}(x_1,h_c)=0$. Since $w=x_1+y_1+z_1$, we have
\begin{align*}
\beta_{\calQ}(x_1,w)&=\beta_{\calQ}(x_1,x_1+y_1+z_1)=d_a+d_{ab}+d_{ac},
\end{align*}
where the last equality uses the assumption $(x_1,x_2,y_1,y_2,z_1,z_2) \in K_{2,2,2}(d)$.  This finishes our verification that $x\in X$.  A symmetric argument shows $y\in Y$.

We now prove the $\supseteq$ direction of the desired inequality.  To this end, fix a tuple $(x_1,x_2,y_1,y_2,z_1,z_2)\in G^6$ such that
\begin{align}\label{al:eq1}
(x_1,x_2,y_1,y_2,z_1,z_2)=(x_1,x_1+h_a, y_1, y_1+h_b, w-x_1-y_1, w-x_1-y_1+h_c),
\end{align}
and such that $x_1\in X$, $y_1\in Y$, and $\beta_{\calQ}(x_1,y_1)=d_{ab}$. Note (\ref{al:eq1}) implies $x_1+y_1+z_1=w$, $h_a=x_2-x_1$, $h_b=y_2-y_1$, and $h_c=z_2-z_1$. Combining with the definition of $\Psi$, we have
\begin{align*}
\Psi(x_1,x_2,y_1,y_2,z_1,z_2)=(x_1+y_1+z_1,x_2-x_1,y_2-y_1,z_2-z_1)=(w, h_a,h_b,h_c).
\end{align*}
Thus, $(x_1,x_2,y_1,y_2,z_1,z_2)\in \Psi^{-1}(w, h_a,h_b,h_c)$. It now suffices to show $(x_1,x_2,y_1,y_2,z_1,z_2)\in K_{2,2,2}(d)$.

Since $x_1\in X$ and $y_1\in Y$, we know $x_1\in B(d_a)$ and $y_1\in B(d_b)$.  We also know by assumption $\beta_{\calQ}(x_1,y_1)=d_{ab}$.  Note
\begin{align*}
\beta_{\calQ}(x_2,x_2)=\beta_{\calQ}(x_1+h_a,x_1+h_a)=d_a+2\beta_{\calQ}(x_1,h_a)+\beta_{\calQ}(h_a,h_a)=d_a,
\end{align*}
where the last equality uses that $x_1\in X$.  An identical argument shows $\beta_{\calQ}(y_2,y_2)=d_b$.  We also have
$$
\beta_{\calQ}(x_1,y_2)=\beta_{\calQ}(x_1,y_1+h_b)=d_{ab}+\beta_{\calQ}(x_1,h_b)=d_{ab},
$$
where the second equality uses that $x_1\in X$.  A symmetric argument shows $\beta_{\calQ}(y_1,x_2)=d_{ab}$.  We have now checked the constraints from the definition of $K_{2,2,2}(d)$ involving only the $x_i$'s and $y_j$'s.

We next deal with the constraints involving $z_1$ and $z_2$. Note $w=x_1+y_1+z_1$ and $w\in B(e)$ implies
\begin{align*}
\beta_{\calQ}(z_1,z_1)&=\beta_{\calQ}(w-x_1-y_1,w-x_1-y_1)\\
&=e+d_a+d_c-2\beta_{\calQ}(x_1,y_1)-2\beta_{\calQ}(x_1,w)-2\beta_{\calQ}(y_1,w)\\
&=e+d_a+d_c+2d_{ab}-2(d_a+d_{ac}+d_{ab})-2(d_c+d_{bc}+d_{ab})\\
&=e-d_a-d_c-2d_{ab}-2d_{ac}-2d_{bc}\\
&=d_b.
\end{align*}
where the second equality uses that $x_1\in X$ and $y_1\in Y$, and the last equality uses $\Sigma(d)=e$.   We also have, from (\ref{al:eq1}) and $x_1\in X$ and $y_1\in Y$ that 
\begin{align}\label{al:eq2}
\beta_{\calQ}(x_1,z_1)=\beta_{\calQ}(x_1,w-x_1-y_1)=\beta_{\calQ}(x_1,w)-d_a-d_{ab}=d_{ac}.
\end{align}
A symmetric argument shows $\beta_{\calQ}(y_1,z_1)=d_{ac}$.  Observe (\ref{al:eq1}) and $x_1\in X$ imply
$$
\beta_{\calQ}(x_1,z_2)=\beta_{\calQ}(x_1,z_1+h_c)=\beta_{\calQ}(x_1,z_1)+\beta_{\calQ}(x_1,h_c)=d_{ac}.
$$
A symmetric argument shows $\beta_{\calQ}(y_1,z_2)=d_{bc}$.  Similarly (\ref{al:eq1}) and (\ref{al:eq2}) imply
\begin{align*}
\beta_{\calQ}(x_2,z_1)=\beta_{\calQ}(x_1+h_a,z_1)=d_{ac}+\beta_{\calQ}(h_a,z_1)&=d_{ac}+\beta_{\calQ}(h_a,w-x_1-y_1)\\
&=d_{ac}+\beta_{\calQ}(h_a,w)-\beta_{\calQ}(x_1,h_a)-\beta_{\calQ}(y_1,h_a)\\
&=d_{ac},
\end{align*}
where the last equality uses that $\beta_{\calQ}(y_1,h_a)=0$ (since $y_1\in Y$), that $\beta_{\calQ}(x_1,h_a)=-2\beta_{\calQ}(h_a,h_a)$ (since $x_1\in X$), and that $\beta_{\calQ}(h_a,w)=-2\beta_{\calQ}(h_a,h_a)$ (by Lemma \ref{lem:constraints}).  A symmetric argument shows $\beta_{\calQ}(y_2,z_1)=d_{bc}$.  Finally, (\ref{al:eq1}), (\ref{al:eq2}), and $x_1\in X$ imply
\begin{align*}
\beta_{\calQ}(x_1,z_2)=\beta_{\calQ}(x_1,z_1+h_c)=\beta_{\calQ}(x_1,z_1)+\beta_{\calQ}(x_1,h_c)=d_{ac}.
\end{align*}
A symmetric argument shows $\beta_{\calQ}(y_1,z_2)=d_{bc}$. This finishes the proof.
\end{proof}

We can now give sufficiently good size estimate on the intersection of $K_{2,2,2}(d)$ with the fibres of the form $\Psi^{-1}(z)$ for $z$ in $\Omega_{B(\Sigma(d))}$ (in the case where $\calB$ is a high rank factor).  In particular, we give a very rough upper bound on the size of any such intersection, and then show that when $z$ satisfies some weak independence conditions, we can compute the size of the intersection  almost exactly.

\begin{proposition}\label{prop:preimagesize}
There is a constant $C>0$ so that the following holds.  Let $\calB=(\calL,\calQ)$ be a factor on $G=\F_p^n$  complexity $(\ell,q)$ and rank at least $C(\ell+q+\log_p(\e))$.  

Fix $e\in \F_p^{\ell}\times \F_p^q$ and $d=(d_a,d_b,d_c,d_{ab},d_{ac},d_{bc})\in (\F_p^{\ell}\times \F_p^q)^3\times (\F_p^q)^{3}$ such that $\Sigma(d)=e$, and set $B=B(e)$. Then for all $(w,h_a,h_b,h_c)\in \Omega_{B}$,
$$
|\Psi^{-1}(w,h_a,h_b,h_c)\cap K_{2,2,2}(d)|\leq (1+\e)p^{2n-2\ell-3q}.
$$
Moreover, if $\calL\cup \{Mw, Mh_a,Mh_b,Mh_c: M\in \calQ\}$ is a linearly independent set, then
$$
|\Psi^{-1}(w,h_a,h_b,h_c)\cap K_{2,2,2}(d)|=(1\pm \e)p^{2n-2\ell -11q}.
$$
\end{proposition}
\begin{proof}
Let $C>0$ be a sufficiently large constant.  Fix $\ell$, $q$, $\calB$, $e$, $d$ and $B=B(e)$ as in the hypotheses.  Fix $(w,h_a,h_b,h_c)\in \Omega_{B}$.  By Proposition \ref{prop:preimage},  
\begin{align*}
&\Psi^{-1}(w,h_a,h_b,h_c)\cap K_{2,2,2}(d)\\
&=\{(x,x+h_a, y, y+h_b, w-x-y, w-x-y+h_c):x\in X, y\in Y,\text{ and } \beta_{\calQ}(x,y)=d_{ab}\},
\end{align*}
where
\begin{align*}
X&=\{x\in B(d_a): \beta_{\calQ}(x,h_b)=\beta_{\calQ}(x,h_c)=0, 2\beta_{\calQ}(x,h_a)=-\beta_{\calQ}(h_a,h_a),\beta_{\calQ}(x,w)=d_a+d_{ac}+ d_{ab}\},
\end{align*}
and 
\begin{align*}
Y&=\{y\in B(d_b): \beta_{\calQ}(y,h_a)=\beta_{\calQ}(y,h_c)=0, 2\beta_{\calQ}(y,h_b)=-\beta_{\calQ}(h_b,h_b),\beta_{\calQ}(y,w)=d_b+d_{bc}+d_{ab}\}.
\end{align*}
Thus,
$$
|\Psi^{-1}(w,h_a,h_b,h_c)\cap K_{2,2,2}(d)|=|\{(x,y)\in X\times Y: \beta_{\calQ}(x,y)=d_{ab}\}|.
$$
In particular, 
$$
|\Psi^{-1}(w,h_a,h_b,h_c)\cap K_{2,2,2}(d)|\leq |\{(x,y)\in B(d_a)\times B(d_b): \beta_{\calQ}(x,y)=d_{ab}\}|\leq (1+\e)p^{2n-2\ell-3q},
$$
where the last inequality is by Lemmas \ref{lem:sizeofatoms}, \ref{lem:sizeofbeta}, and \ref{lem:A2} (and because $\calB$ has high rank). 

Suppose now that, moreover, $\calL\cup \{Mw, Mh_a,Mh_b,Mh_c: M\in \calQ\}$ is a linearly independent set. Let $\calB'=(\calL',\calQ')$  where
$$
\calL'=\calL\cup \{Mw, Mh_a,Mh_b,Mh_c: M\in \calQ\}\text{ and }\calQ'=\calQ.
$$
As $\calL'$ is linearly independent, this is a factor on $\F_p^n$ of complexity  $(\ell',q')=(\ell+4q, q)$   with the same rank as $\calB$.  It is not hard to see that both $X$ and $Y$ are atoms of this factor $\calB'$, and thus we are trying to compute the size of a set of the form 
$$
|\beta_{\calQ}^{-1}(d_{ac})\cap (B'(u)\times B'(v))|.
$$
Using Lemmas \ref{lem:sizeofatoms}, \ref{lem:sizeofbeta}, and \ref{lem:A2}, and our rank assumption, we see this set has size 
$$
p^{-2n-2(\ell'-q')-q'}(1\pm \e)=(1\pm \e)p^{2n-2(\ell+4q+q)-q}=(1\pm \e)p^{2n-2\ell -11q}.
$$
\end{proof}

We now have the necessary ingredients to prove the main result of this section, Theorem \ref{thm:relatenorms}, which says the local $U^3$ norms of Definitions \ref{def:P} and \ref{def:TW} are approximately the same.

\vspace{3mm}

\begin{proofof}{Theorem \ref{thm:relatenorms}}
The reader may wish to review Definitions \ref{def:u3}, \ref{def:P}, and \ref{def:TW} as all three ``$U^3$ norms" make appearances in this proof. Let $C>0$ be a sufficiently large constant. Fix $\e\in (0,1)$, integers $\ell,q\geq 0$, and a quadratic factor $\calB=(\calL,\calQ)$ on $G=\F_p^n$ of complexity $(\ell,q)$ and rank at least $C(\ell+q+\log_p(\e))$. Without loss of generality (after possibly replacing $\e$), assume $0<\e<1/11$.  Suppose $b\in \F_p^{\ell}\times \F_p^{q}$ and  $d\in (\F_p^{\ell}\times \F_p^q)^3\times (\F_p^q)^3$ are such that $\Sigma(d)=b$. Let $f:G\rightarrow [-1,1]$ be a function.  Recall that by Lemma \ref{lem:omegasize} (and because $\calB$ has sufficiently high rank), we have
\begin{align}\label{omegab}
|\Omega_{\calB}|=(1\pm \e^2)p^{4n-4\ell-7q}.
\end{align}
We will work with a partition of $\Omega_{\calB}$ consisting of the following two sets.   
\begin{align*}
\Omega_{good}&=\{(x,h_1,h_2,h_3)\in \Omega_B: \calL\cup \{xM,h_1M,h_2M,h_3M\}\text{ is a linearly independent set}\} \text{ and }\\
\Omega_{bad}&=\Omega_B\setminus \Omega_{good}. 
\end{align*}
In light of Observations \ref{ob:psi} and \ref{ob:k222}, we can write  the following (see Notation \ref{not:pif} as well).
\begin{align*}
&\sum_{(x_1,x_2,y_1,y_2,z_1,z_2)\in  K_{2,2,2}(d)}\prod_{(i,j,k)\in [2]^3}f(x_i+y_j+z_k)\\
&=\sum_{(x,h_1,h_2,h_3)\in \Omega_B}\Big(\sum_{(x_1,x_2,y_1,y_2,z_1,z_2)\in \Psi^{-1}(x,h_1,h_2,h_3)\cap  K_{2,2,2}(d)}\pi_f(x,h_1,h_2,h_3)\Big)\\
&=\sum_{(x,h_1,h_2,h_3)\in \Omega_B}|\Psi^{-1}(x,h_1,h_2,h_3)\cap  K_{2,2,2}(d)|\pi_f(x,h_1,h_2,h_3).
\end{align*}

Combining with Definition \ref{def:u3}, we have  
\begin{align*}
& \Big| \Big(\sum_{(x_1,x_2,y_1,y_2,z_1,z_2)\in  K_{2,2,2}(d)}\prod_{(i,j,k)\in [2]^3}f(x_i+y_j+z_k)\Big)-\Big(p^{2n-2\ell-11q}\|f\cdot 1_B\|_{U^3}^8\Big)\Big|\\
&= \Big| \Big(\sum_{(x_1,x_2,y_1,y_2,z_1,z_2)\in  K_{2,2,2}(d)}\prod_{(i,j,k)\in [2]^3}f(x_i+y_j+z_k)\Big)\\
&-\Big(p^{2n-2\ell-11q}\sum_{(x,h_1,h_2,h_3)\in \Omega_B}\pi_f(x,h_1,h_2,h_3)\Big)\Big|\\
&=\Big|\sum_{(x,h_1,h_2,h_3)\in \Omega_B}\Big(|\Psi^{-1}(x,h_1,h_2,h_3)\cap  K_{2,2,2}(d)|-p^{2n-2\ell-11q}\Big) \pi_f(x,h_1,h_2,h_3)\Big|\\
&\leq \Big|\sum_{(x,h_1,h_2,h_3)\in \Omega_{good}}\Big(|\Psi^{-1}(x,h_1,h_2,h_3)\cap  K_{2,2,2}(d)|-p^{2n-2\ell-11q}\Big) \pi_f(x,h_1,h_2,h_3)\Big|\\
&+\Big|\sum_{(x,h_1,h_2,h_3)\in \Omega_{bad}}\Big(|\Psi^{-1}(x,h_1,h_2,h_3)\cap  K_{2,2,2}(d)|-p^{2n-2\ell-11q}\Big) \pi_f(x,h_1,h_2,h_3)\Big|,
\end{align*}
where the last inequality is by the triangle inequality and because $\Omega_{\calB}=\Omega_{good}\sqcup\Omega_{bad}$ by definition. We consider the two terms in the final sum above separately. First,  by applying the triangle inequality, using the fact   $f$ is $1$-bounded, and applying Proposition \ref{prop:preimagesize} (with error parameter $\e^2$), we have
\begin{align*}
&\Big|\sum_{(x,h_1,h_2,h_3)\in \Omega_{good}}\Big(|\Psi^{-1}(x,h_1,h_2,h_3)\cap  K_{2,2,2}(d)|-p^{2n-2\ell-11q}\Big)\pi_f(x,h_1,h_2,h_3)\Big| \\
&\leq \sum_{(x,h_1,h_2,h_3)\in \Omega_{good}}\Big||\Psi^{-1}(x,h_1,h_2,h_3)\cap  K_{2,2,2}(d)|-p^{2n-2\ell-11q}\Big|\\
&\leq \e^2 p^{2n-2\ell-11q}|\Omega_{good}|\\
&\leq \e^2 p^{2n-2\ell-11q}|\Omega_B|,
\end{align*}
where the last inequality is because $\Omega_{good}\subseteq \Omega_B$.  Turning to the second sum, we have by the triangle inequality, because $f$ is $1$-bounded, and by Proposition \ref{prop:preimagesize} (again with error parameter $\e^2$), that
\begin{align*}
&\Big|\sum_{(x,h_1,h_2,h_3)\in \Omega_{bad}}\Big(|\Psi^{-1}(x,h_1,h_2,h_3)\cap  K_{2,2,2}(d)|-p^{2n-2\ell-11q}\Big) \pi_f(x,h_1,h_2,h_3)\Big|\\
&\leq \sum_{(x,h_1,h_2,h_3)\in \Omega_{bad}}\Big(|\Psi^{-1}(x,h_1,h_2,h_3)\cap  K_{2,2,2}(d)|+p^{2n-2\ell-11q}\Big)\\
&\leq |\Omega_{bad}|((1+\e^2)p^{2n-2\ell-3q}+p^{2n-2\ell-11q})\\
&\leq 2|\Omega_{bad}|(1+\e^2)p^{2n-2\ell-3q}\\
&\leq 28 (1+\e^2)p^{6n-\ell+q-r}\\
&\leq \e^2 p^{2n-2\ell-11q}|\Omega_B|,
\end{align*}
where the second to last inequality uses Lemma \ref{lem:omegagood}, and the last uses (\ref{omegab}) and the fact $\rk(\calB)$ is sufficiently large.  We can now conclude
\begin{align}\label{1}
 \nonumber \Big| \Big(\sum_{(x_1,x_2,y_1,y_2,z_1,z_2)\in  K_{2,2,2}(d)}\prod_{(i,j,k)\in [2]^3}f(x_i+y_j+z_k)\Big)-\Big(p^{2n-2\ell-11q}\Big(\|f\cdot 1_B\|_{U^3}\Big)^8\Big)\Big| &\\
\leq 2\e^2 p^{2n-2\ell-11q} |\Omega_B| &.
\end{align}
 By Proposition \ref{prop:rewritetw} (and since $\rk(\calB)$ is sufficiently large), we have
\begin{align}\label{2}
\Big(\|f\|_{U^3(d)}^{TW}\Big)^8=(1\pm \e^2)p^{-6n+6\ell+18q}\sum_{(x_1,x_2,y_1,y_2,z_1,z_2)\in  K_{2,2,2}(d)}\prod_{(i,j,k)\in [2]^3}f(x_i+y_j+z_k).
\end{align}
Combining (\ref{1}) and (\ref{2}), we have
\begin{align*}
&\Big(\|f\|_{U^3(d)}^{TW}\Big)^8\\
&= (1\pm \e^2)p^{-6n+6\ell+18q}\sum_{(x_1,x_2,y_1,y_2,z_1,z_2)\in  K_{2,2,2}(d)}\prod_{(i,j,k)\in [2]^3}f(x_i+y_j+z_k)\\
&= (1\pm \e^2)p^{-6n+6\ell+18q}\Big(p^{2n-2\ell-11q}\Big(\|f\cdot 1_B\|_{U^3}\Big)^8\pm 2\e^2 p^{2n-2\ell-11q} |\Omega_B|\Big)\\
&=(1\pm \e^2)p^{-4n+4\ell+7q}\Big(\Big(\|f\cdot 1_B\|_{U^3}\Big)^8 \pm 2\e^2|\Omega_B|\Big)\\
&= (1\pm \e^2)p^{-4n+4\ell+7q}|\Omega_B|\Big(\Big(\|f\|_{U^3(e)}^P\Big)^8+2\e^2\Big)\\
&=(1\pm  \e^2)^2\Big(\Big(\|f\|_{U^3(e)}^P\Big)^8+2\e^2\Big),
\end{align*}
where the last inequality is by (\ref{omegab}). Since $\e<1/11$, and since $(\|f\|_{U^3(e)}^P)^8\leq 1$ (see Observation \ref{ob:1bound}), this yields the desired statement  $(\|f\|_{U^3(d)}^{TW})^8=(\|f\|_{U^3(e)}^P)^8\pm \e$.
\end{proofof}

\vspace{2mm}

We can now deduce the key corollary of Theorems \ref{thm:key} and \ref{thm:relatenorms} needed for the proof of our main theorem.  

 \begin{corollary}\label{cor:key}
  For all integers $k\geq 1$, there exists a constant $K>0$ such that the following holds.  Suppose $0<\e<1/11$, $\ell,q\geq 0$, are integers, and $\calB=(\calL,\calQ)$ is a quadratic factor on $\F_p^n$ of complexity $(\ell,q)$ and rank at least $C(\ell+q+\log_p(\e))$.
  
  Suppose  $A\subseteq \F_p^n$ satisfies $\VC_2(A)\leq k$,  $b\in \F_p^{\ell}\times \F_p^q$, and $\|1_A-\alpha_{B(b)}\|_{U^3(b)}^P<(\e/4)^{m^22^{m^2}}2^{-1}$, where $\alpha_{B(b)}$ denotes the density of $A$ on the atom $B(b)$. Then  $\alpha_{B(b)}\in [0,\e)\cup (1-\e,1]$.
 \end{corollary}
\begin{proof}
Let $C>0$ be as in Theorem \ref{thm:relatenorms} and let $K=2k^2 2^{k^2}C$.  Fix $0<\e<1/11$,  integers  $\ell,q\geq 0$, a quadratic factor $\calB=(\calL,\calQ)$ on $\F_p^n$ of complexity $(\ell,q)$ and rank at least $K(\ell+q+\log_p(\e))$, and a label $b\in \F_p^{\ell}\times \F_p^q$. Assume $A\subseteq \F_p^n$ has $\VC_2$-dimension at most $k$ and satisfies 
\begin{align}\label{norm}
\|1_A-\alpha_{B(b)}\|_{U^3(b)}^P<(\e/4)^{m^22^{m^2}}2^{-1},
\end{align}
where $\alpha_{B(b)}$ is the density of $A$ on the atom $B(b)$.   Fix any tuple $d\in (\F_p^{\ell}\times \F_p^q)^3\times (\F_p^q)^3$ satifying $\Sigma(d)=b$. Our choice of $K$ implies $\calB$ has rank at least
$$
C(\ell+q+\log_p((\e/4)^{m^22^{m^2}}2^{-1})),
$$
and thus, Theorem \ref{thm:relatenorms} and (\ref{norm}) imply 
$$
\|1_A-\alpha_{B(b)}\|_{U^3(d)}^{TW}\leq \|1_A-\alpha_{B(b)}\|_{U^3(b)}^P+(\e/4)^{m^22^{m^2}}2^{-1}< (\e/4)^{m^22^{m^2}}.
$$
By Theorem \ref{thm:key}, we have $\alpha_{B(b)}\in [0,\e)\cup (1-\e,1]$.
\end{proof}

\section{Matrix deletion and addition}\label{sec:rank}

In this section we examine the bounds generated by performing sequences of operations on factors.  We begin by considering the rank lemma (Lemma \ref{lem:rank}) in detail, which is proved by iteratively applying a matrix deletion operation. We then discuss more complicated procedures, which interweave matrix deletion and addition steps.

 Those familiar with the proof of Lemma \ref{lem:rank} will know it is proved by iteratively deleting a matrix involved in a low rank linear combination, then adding enough vectors to the linear component to recover the deleted information.  It will be useful to have the following definition for when one factor arises from another by a single such deletion.   

\begin{definition}\label{def:rhodeletion}
Suppose $\rho$ is a growth function, $\ell\geq 0$ and $q\geq 1$ are integers, and $\calB=(\calL,\calQ)$ is a quadratic factor on $\F_p^n$ of complexity $(\ell,q)$ with $\calQ=\{M_1,\ldots, M_q\}$.  

A \emph{$\rho$-matrix deletion of $\calB$} is a quadratic factor $\calB'=(\calL',\calQ')$ on $\F_p^n$ such that for some non-trivial linear combination $U=\lambda_1M_1+\ldots +\lambda_qM_q$, the following hold.
\begin{enumerate}
\item $\rk(U)<\rho(\ell+q)$, 
\item there is some $i\in [q]$ such that $\lambda_i\neq 0$ and  $\calQ'=\calQ\setminus \{M_i\}$,
\item  $\calL'$ is a minimal set of vectors containing $\calL$ and spanning $\ker(U)^{\perp}$. 
\end{enumerate}
\end{definition}

If $\calB'$ is a $\rho$-matrix deletion of $\calB$, we obtain the following information about the complexity of $\calB'$ by definition.  

\begin{observation}\label{ob:delete}
Suppose $\rho$ is a growth function, $\ell\geq 0$ and $q\geq 1$ are integers, and $\calB$ is a quadratic factor on $\F_p^n$ of complexity $(\ell,q)$.  Suppose $\calB'$ is a quadratic factor on $\F_p^n$ of complexity $(\ell',q')$, and assume $\calB'$ is a $\rho$-matrix deletion of $\calB$. Then
\begin{align*}
q'=q-1\text{ and }\ell'<\ell+\rho(\ell+q).
\end{align*}
\end{observation}

While it is easy to understand the bounds resulting from a single matrix deletion, we also need to understand what happens to the bounds after a sequence of matrix deletions.   Towards this goal, we define functions to help us bound the complexity after  performing $i$ many $\rho$-matrix deletions.

\begin{definition}\label{def:tau}
Suppose $\rho$ is a growth function.  We define functions $\tau^{\rho}_i:\mathbb{Z}\times \mathbb{Z}\to \R$ for all $i\in \mathbb{N}$ inductively as follows.  
\begin{itemize}
\item Define $\tau^{\rho}_0$ by setting $\tau^{\rho}_0(x,y) = x$ for all $(x,y)\in \mathbb{Z}\times \mathbb{Z}$.  
\item Given $i\geq 0$, define $\tau^{\rho}_{i+1}$ in terms of $\tau^{\rho}_{i}$ by setting 
$$\tau^{\rho}_{i+1}(x,y) = \tau^{\rho}_{i}(x,y) +\rho(\tau^{\rho}_{i}(x,y) + y-i),$$
for all $(x,y)\in \mathbb{Z}\times \mathbb{Z}$.  
\end{itemize}
\end{definition}

 Definition \ref{def:tau} is designed to bound the complexity of a factor obtained at the end of a chain of $\rho$-matrix deletions.  This is the content of the following lemma.

\begin{lemma}\label{lem:rankbound}
Let $\rho$ be a growth function, and let $m\geq 0$ be an integer. Suppose that  for each $0\leq i\leq m$, $\calB_i=(\calL_i,\calQ_i)$ is a quadratic factor on $\F_p^n$ of complexity $(\ell_i,q_i)$, and assume that for each $0\leq j\leq m-1$, $\calB_{j+1}$ is $\rho$-matrix deletion of $\calB_j$.  Then  
$$
|\calL_m|\leq \tau_m^{\rho}(\ell_0,q_0)
$$
\end{lemma}
\begin{proof}
This is immediate from Definition \ref{def:tau} and Observation \ref{ob:delete}.
\end{proof}

As a corollary, we can state the following more quantitative version of Lemma \ref{lem:rank}.

\begin{lemma}\label{lem:rank2}
 Let $\rho$ be a growth function, let $\ell,q\geq 0$ be integers, and let $\calB=(\calL,\calQ)$ be a quadratic factor on $\F_p^n$ of complexity $(\ell,q)$.  Then there exists a quadratic factor $\calB'\preceq \calB$ of complexity $(\ell',q')$ and rank at least $\rho(\ell'+q')$ for some integers $\ell',q'$ satisfying  
 $$
 \ell \leq \ell'\leq \tau^{\rho}_q(\ell,q)\text{ and }0\leq q'\leq q.
 $$
\end{lemma}
\begin{proof}
This is immediate from Lemma \ref{lem:rankbound} and the proof of Lemma \ref{lem:rank} from \cite{Green.2007}, which obtains $\calB'$ at the end of a sequence of at most $q$ many $\rho$-matrix deletions.
\end{proof}

To analyze the bounds in our proofs, we will need the following lemma, which bounds $\tau_i^{\rho}(x,y)$ for polynomial $\rho$ and certain inputs $x$ and $y$.

\begin{lemma}\label{lem:tau}
Suppose $C>1$ is a real and $k\geq 1$ is an integer, and $\rho$ is a growth function satisfying $\rho(x)\geq x$ and $\rho(x)\leq Cx^k$ for all $x\geq 1$.  Then for all integers $i\geq 0$ and all $(x,y)\in \mathbb{N}\times \mathbb{N}^{\geq i}$,
$$
\tau^{\rho}_i(x,y)\leq 2^{ik^i}C^{ik^i}(x+y)^{k^i}.
$$
\end{lemma}
\begin{proof}
We prove this by induction on $i\geq 0$.  

\noindent\underline{Base Case:} Suppose first $i=0$. By definition of $\tau_0$ and since $C>1$, we have that for any $(x,y)\in \mathbb{N}\times \mathbb{N}$, 
$$
\tau_0^{\rho}(x,y)=x\leq x+y=2^{ik^i}C^{ik^i}(x+y)^{k^i}.
$$

\noindent\underline{Induction Step:} Suppose $i\geq 0$, and assume by induction we have shown that for all $(x,y)\in \mathbb{N}\times \mathbb{N}^{\geq i}$, $\tau^{\rho}_i(x,y)\leq 2^{ik^i}C^{ik^i}(x+y)^{k^i}$.  Fix $(x,y)\in \mathbb{N}\times \mathbb{N}^{\geq i+1}$. By definition of $\tau_{i+1}^{\rho}$,  our induction hypothesis, and since $\rho$ is increasing,
\begin{align*}
\tau^{\rho}_{i+1}(x,y)&=\tau^{\rho}_i(x,y)+\rho(\tau^{\rho}_i(x,y)+y-i-1)\\
&\leq 2^{ik^i}C^{ik^i}(x+y)^{k^i}+\rho(2^{ik^i}C^{k^i}(x+y)^{k^i}+y-i-1)\\
&\leq 2\rho(2^{ik^i}C^{ik^i}(x+y)^{k^i}+y-i-1)\\
&\leq 2\rho(2^{ik^i+1}C^{ik^i}(x+y)^{k^i}),
\end{align*}
where the second inequality is because $\rho(z)\geq z$ for all $z\geq 1$, and the third is because $\rho$ is increasing (this step uses our assumption that $y\geq i+1$).  Using that $\rho(z)\leq Cz^k$ for all $z\in \mathbb{R}^{\geq 0}$, this is at most 
\begin{align*}
2C(2^{ik^i+1}C^{ik^i}(x+y)^{k^i})^k= 2^{ik^{i+1}+k+1}C^{ik^{i+1}+1}(x+y)^{k^{i+1}}\leq 2^{(i+1)k^{i+1}}C^{(i+1)k^{i+1}}(x+y)^{k^{i+1}}.
\end{align*}
This finishes the proof.
\end{proof}

With these tools in hand, we can easily upper bound the linear complexity in Lemma \ref{lem:rank} when the growth function $\rho$ is bounded above by a polynomial.

\begin{lemma}\label{lem:rankpoly}
For any polynomial growth function $\rho$ of degree $k\geq 1$ satisfying $\rho(x)\geq x$ for all $x\geq 1$, there is a constant $C>0$ so that the following holds. Suppose $\ell,q\geq 0$ are integers and $\calB=(\calL,\calQ)$ is a quadratic factor on $\F_p^n$ of complexity $(\ell,q)$.  There there exist integers $\ell',q'$ satisfying $0\leq q'\leq q$ and
$$
\ell \leq \ell'\leq 2^{qk^q}C^{qk^q}(\ell+q)^{k^q},
$$
and a quadratic factor $\calB'\preceq \calB$ of complexity $(\ell',q')$ and rank at least $\rho(\ell'+q')$.
\end{lemma}
\begin{proof}
Choose $C$ sufficiently large so that for all $x\geq 1$, $p(x)\leq C x^k$.   The stated bound is then immediate from Lemmas \ref{lem:rank2} and \ref{lem:tau}.
\end{proof}

Lemma \ref{lem:rankpoly} gives us a good estimate on the bounds resulting from performing many matrix deletions in a row (as occurs in the proof of Lemma \ref{lem:rank2}). It also suffices for computing bounds in proofs which alternate  Lemma \ref{lem:rank2} with steps that add new linear and quadratic terms to a factor, as long as this alternation follows a simple pattern. An example of this occurs in our energy increment proof of the (non-cylinder) quadratic arithmetic regularity lemma given in Section \ref{sec:warmup}.  

However, analyzing the bounds produced by the proof of our cylinder quadratic arithmetic regularity lemma (Theorem \ref{thm:cylinder}) will be more subtle. The individual steps in this proof will take one of two forms: a single $\rho$-matrix deletion (we call this a \emph{deletion step}), or the addition of at most one new linear term and at most one new quadratic term to a factor (we call this an  \emph{addition step}).  We will need to  analyze the bounds resulting from a sequence of such steps, performed in an order we have very little information about. 

To make our analysis of such a process precise, we will use sequences of $-1$'s and $1$'s to keep track of the pattern of addition and deletion steps performed, with $-1$ representing a deletion step, and $1$ representing an addition step.  For this reason, we will need several definitions and notational conventions related to sequences of $-1$'s and $1$'s, which we will refer to as \emph{binary strings}.  

\begin{notation}[Binary strings]\label{not:binary}
$\text{ }$
\begin{itemize}
\item Given an integer $m\geq 1$, a \emph{binary string of length $m$} is a tuple $\sigma\in \{-1,1\}^m$.
\item By convention, $\{-1,1\}^0$ contains a unique element, called the \emph{empty string}, which we denote by $<>$.  
\item Given an integer $m\geq 0$, let 
$$
\{-1,1\}^{\leq m}=\bigcup_{i=0}^m\{-1,1\}^i.
$$
\item Given integers $m,m'\geq 1$, a sequence  $\sigma=(\sigma_1,\ldots, \sigma_m)\in \{-1,1\}^m$, and a sequence $\sigma'=(\sigma'_1,\ldots, \sigma'_{m'})\in \{-1,1\}^{m'}$, we let $\sigma\wedge \sigma'$ denote  the element of $\{-1,1\}^{m+m'}$ obtained by appending $\sigma'$ to the end of $\sigma$, i.e.
$$
\sigma\wedge \sigma':=(\sigma_1,\ldots, \sigma_m,\sigma'_1,\ldots, \sigma'_{m'}).
$$  
\item For integers $m_1\geq m_2\geq 1$ and $\sigma=(\sigma_1,\ldots, \sigma_{m_1})\in \{-1,1\}^{m_1}$, define the \emph{restriction of $\sigma$ to $[m_2]$} to be 
$$
\sigma|_{[m_2]}=(\sigma_1,\ldots, \sigma_{m_2})\in \{-1,1\}^{m_2}.
$$
\item By convention, for any $m\geq 0$ and $\sigma\in \{-1,1\}^m$, $<> \wedge \sigma=\sigma\wedge<>=\sigma$ and $\sigma|_{[0]}=<>$.
\end{itemize}
\end{notation}

We will use the following simple observation several times.

\begin{observation}\label{ob:sigma}
For any integer $m>0$ and $\sigma\in \{-1,1\}^m$, there is $\sigma'\in \{-1,1\}^{m-1}$ and $u\in \{-1,1\}$ such that $\sigma=\sigma'\wedge u$.
\end{observation}

In our proofs, it will be important to track the number of $1$'s appearing in a binary string, as well as the discrepancy between the number of $1$'s and the number of $-1$'s.  For this we will use the following notation.

\begin{definition}\label{def:disc}
Given $m\geq 1$ and $\sigma=(\sigma_1,\ldots, \sigma_{m})\in \{-1,1\}^{m}$, define
$$
|\sigma|=|\{i\in [m]: \sigma_i=1\}|,
$$
and define the \emph{discrepancy} of $\sigma$ to be 
$$
\disc(\sigma)=\sum_{i=1}^m\sigma_i.
$$
By convention, we set $|<>|=0$ and $\disc(<>)=0$.
\end{definition}

Our next definition will help us understand quadratic factors which arise from performing a sequence of matrix addition and deletion steps, starting with the trivial   factor.  Informally, a $(\rho,\sigma)$-chain will be a sequence of factors, each of which is obtained from the preceding factor by performing either a deletion or addition step, and where the pattern of steps is dictated by the binary string $\sigma$.  It will be convenient later to also have a definition for a $(\rho,\sigma)$-chain when $\sigma$ is the empty string. Such a chain will consist of the trivial factor.

\begin{definition}\label{def:chain}
Suppose $\rho$ is a growth function.
\begin{enumerate}
\item A \emph{$(\rho,<>)$-chain on $\F_p^n$} is simply the trivial factor $\calB_0=(\calL_0,\calQ_0)$ where $\calL_0=\calQ_0=\emptyset$.
\item Given an integer $m\geq 1$ and  $\sigma=(\sigma_1,\ldots, \sigma_m)\in \{-1,1\}^m$, a \emph{$(\rho,\sigma)$-chain  on $\F_p^n$} is a chain of quadratic factors on $\F_p^n$,
$$
\calB_0=(\calL_0,\calQ_0)\preceq \cdots \preceq \calB_m=(\calL_m,\calQ_m),
$$
such that $\calL_0=\calQ_0=\emptyset$, and such that for each $0\leq i\leq m-1$, one of the following hold.
\begin{itemize}
\item $\sigma_i=-1$ and $\calB_{i+1}$ is a $\rho$-matrix deletion of $\calB_i$,
\item  $\sigma_i=1$ and $\calL_i\subseteq \calL_{i+1}$, $\calQ_i\subseteq \calQ_{i+1}$,  $|\calL_{i+1}|\leq |\calL_i|+1$, and $|\calQ_{i+1}|\leq |\calQ_i|+1$.
\end{itemize}
\end{enumerate}
\end{definition}

We now state Definition \ref{def:fsigma}, whose goal is to define a pair of real numbers, denoted  $f_{\sigma}^{\rho}$, to serve as a bound on the complexity of a factor obtained at the end of a $(\rho,\sigma)$-chain.  

\begin{definition}\label{def:fsigma}
Suppose $\rho$ is a growth function. For each integer $m\geq 0$ and $\sigma\in \{-1,1\}^m$, we define a value $f^{\rho}_{\sigma}\in \mathbb{R}\times \mathbb{R}$ by induction as follows.
\begin{enumerate}
\item Define $f_{<>}^{\rho}=(0,0)$.
\item Suppose $m>0$ and we have defined $f_{\sigma}^{\rho}\in   \mathbb{R}\times \mathbb{R}$ for each $\sigma\in \{-1,1\}^{m}$.   Then for each $\sigma\in \{-1,1\}^m$, define $f^{\rho}_{\sigma\wedge 1}$ and $f^{\rho}_{\sigma\wedge -1}$ as follows, in terms of $(a,b)=f_{\sigma}^{\rho}$:
$$
f^{\rho}_{\sigma\wedge 1}=(a+1,b+1)\text{ and }f^{\rho}_{\sigma\wedge -1}=(a+\rho(a+b),b-1).
$$ 
\end{enumerate}
\end{definition}

Our next lemma shows Definition \ref{def:fsigma} can be used to bound complexities of factors arising from $(\rho,\sigma)$-chains. It moreover shows the discrepancy of $\sigma$ can be used to bound the quadratic complexity of the final factor.  We note this will implicitly tell us that only certain $\sigma$ give rise to $(\rho,\sigma)$-chains, as not all $\sigma$ have non-negative discrepancy.

\begin{lemma}\label{lem:chain}
Let $\rho$ be a growth function, let $m\geq 0$ be an integer, and let $\sigma \in \{-1,1\}^m$.  Suppose 
$$
\calB_0=(\calL_0,\calQ_0)\preceq \cdots \preceq \calB_m=(\calL_m,\calQ_m)
$$
is a $(\rho,\sigma)$-chain on $\F_p^n$ and for each $i\in \{0,\ldots, m\}$, $(\ell_i,q_i)$ denotes the complexity of $\calB_i$.  Then the following hold.
\begin{enumerate}[(a)]
\item For all $0\leq i\leq m$, $0\leq q_i\leq \disc(\sigma|_{[i]})$, 
\item For all $0\leq i\leq m$, if $(a_i,b_i)=f_{\sigma|_{[i]}}^{\rho}$, then $\ell_i\leq a_i$ and $q_i\leq b_i$.
\end{enumerate}
\end{lemma}
\begin{proof}
Fix a growth function $\rho$.  We show the desired conclusion holds for this $\rho$ and all $\sigma\in \{-1,1\}^m$ by induction on $m\geq 0$.

\noindent{\underline{\bf Base Case:}} If $m=0$, the desired conclusions hold trivially.

\noindent{\underline{\bf Induction Step:}} Suppose now $m>0$, and assume by induction the claim holds for all $0\leq m'<m$.  Fix $\sigma \in \{-1,1\}^m$ and assume 
$$
\calB_0=(\calL_0,\calQ_0)\preceq \cdots \preceq \calB_m=(\calL_m,\calQ_m)
$$
is a $(\rho,\sigma)$-chain on $\F_p^n$. For each $0\leq i\leq m$, let $(\ell_i,q_i)$ be the complexity of $\calB_i$ and let $(a_i,b_i)=f_{\sigma|_{[i]}}^{\rho}$. By Observation \ref{ob:sigma}, there exists some $\sigma'\in \{-1,1\}^{m-1}$ and $u\in \{-1,1\}$ such that $\sigma=\sigma'\wedge u$.  Given $0\leq i\leq m-1$, since $\sigma'|_{[i]}=\sigma|_{[i]}$, the induction hypothesis implies that $0\leq q_i\leq  \disc(\sigma|_{[i]})$, and $\ell_i\leq a_i$ and $q_i\leq b_i$.  We deal with the $i=m$ case using separate arguments based on whether $u=-1$ or $u=1$.  

Suppose first $u=-1$. Then by definition of a $(\rho,\sigma)$-chain, we must have that 
\begin{align*}
q_m=q_{m-1}-1\text{ and } \ell_m\leq \ell_{m-1}+\rho(\ell_{m-1}+q_{m-1}),\\
 b_m=b_{m-1}-1\text{ and }a_m=a_{m-1}+\rho(a_{m-1}+b_{m-1}).
 \end{align*}
 Combining these inequalities with $b_{m-1}\leq q_{m-1}$ and $\ell_{m-1}\leq a_{m-1}$ (which hold by induction), and the fact that $\rho$ is increasing, we have
$$
q_m= q_{m-1}-1\leq b_{m-1}-1=b_m\text{ and }\ell_m\leq \ell_{m-1}+\rho(\ell_{m-1}+q_{m-1})\leq a_{m-1}+\rho(a_{m-1}+b_{m-1})=a_m.
$$
Since $(\ell_m,q_m)$ is the complexity of $\calB_m$, we obviously have $q_m\geq 0$. Since $\sigma=\sigma'\wedge -1$,  $\disc(\sigma)=\disc(\sigma')-1$. By induction, we have $0\leq q_{m-1}\leq \disc(\sigma')$. Thus,  
$$
0\leq q_m=q_{m-1}-1\leq \disc(\sigma')-1=\disc(\sigma).
$$
This finishes the case $u=-1$.

Suppose now $u=1$. Then by definition of a $(\rho,\sigma)$-chain, and Definition \ref{def:fsigma}, we must have that 
\begin{align*}
q_m\leq q_{m-1}+1\text{ and } \ell_m\leq \ell_{m-1}+1,\\
 b_m=b_{m-1}+1\text{ and }a_m=a_{m-1}+1.
 \end{align*}
 Combining these inequalities with  $b_{m-1}\leq q_{m-1}$ and $\ell_{m-1}\leq a_{m-1}$ (which hold by induction), we have
$$
q_m\leq q_{m-1}+1\leq b_{m-1}+1=b_m\text{ and }\ell_m\leq \ell_{m-1}+1\leq a_{m-1}+1=a_m.
$$
Again, we must have $0\leq q_m$ as  $(\ell_m,q_m)$ is the complexity of the factor $\calB_m$. On the other hand, since $\sigma=\sigma'\wedge 1$,  $\disc(\sigma)=\disc(\sigma')+1$. By induction, we have $0\leq q_{m-1}\leq \disc(\sigma')$. Thus, 
$$
0\leq q_m=q_{m-1}+1\leq \disc(\sigma')+1=\disc(\sigma).
$$
This finishes the case $u=1$.   
\end{proof}

 In light of Lemma \ref{lem:chain}, we need to obtain bounds on the coordinates of $f_{\sigma}^{\rho}$ from Definition \ref{def:fsigma}.  We accomplish this through a series of lemmas.  Our first lemma in this direction, Lemma \ref{lem:seq1} below, says that if $\sigma_1$ and $\sigma_2$ are two binary sequences of the same length  such that $f_{\sigma_1}^{\rho}$ is (coordinatewise) bounded by $f_{\sigma_2}^{\rho}$, then this remains true after appending a binary string $\mu$ to both $\sigma_1$ and $\sigma_2$.  This statement appears technical at first, but the idea is that when the sequence of matrix deletions and  additions dictated by $\sigma_2$ generates a factor with worse bounds than the one generated by $\sigma_1$, then this cannot be reversed by applying identical deletion and addition steps to the two resulting factors.  

\begin{lemma}\label{lem:seq1}
Suppose $\rho$ is a growth function, $t\geq 0$ is an integer,  and $\sigma_1,\sigma_2\in \{-1, 1\}^{t}$ are such that  $a_1\leq a_2$ and $b_1\leq b_2$ where $(a_1,b_1)=f_{\sigma_1}^{\rho}$ and $(a_2,b_2)=f_{\sigma_2}^{\rho}$.

Then for all integers $s\geq 0$ and all $\mu\in \{-1,1\}^s$, if $f_{\sigma_1\wedge \mu}^{\rho}=(c_1,d_1)$  and $f_{\sigma_2\wedge \mu}^{\rho}=(c_2,d_2)$, then $c_1\leq c_2$ and $d_1\leq d_2$.  Moreover, if $b_1=b_2$ then  $d_1=d_2$.
\end{lemma}
\begin{proof}
Fix an integer $t\geq 0$ and $\sigma_1,\sigma_2\in \{-1, 1\}^{t}$ such that  $a_1\leq a_2$ and $b_1\leq b_2$ where $(a_1,b_1)=f_{\sigma_1}^{\rho}$ and $(a_2,b_2)=f_{\sigma_2}^{\rho}$.   We prove the claim holds for this $\sigma_1$ and $\sigma_2$ and  all $\mu\in \{-1,1\}^s$ by induction on $s\geq 0$.   

\noindent{\bf Base Case:} If $s=0$ there is nothing to show (since $\sigma_1\wedge<>=\sigma_1$ and $\sigma_2\wedge <>=\sigma_2$).

\noindent{\bf Induction Step:} Suppose $s\geq 0$, and assume by induction the claim holds for $s$.  Fix $\mu\in \{-1,1\}^{s+1}$, and let $(c_1,d_1)=f_{\sigma_1\wedge \mu}^{\rho}$  and $(c_2,d_2)=f_{\sigma_2\wedge \mu}^{\rho}$.  By Observation \ref{ob:sigma},  $\mu$ can be written as $\mu'\wedge v$ for some $\mu'\in \{-1,1\}^s$ and $v\in \{-1,1\}$. Let  $(c'_1,d'_1)=f_{\sigma_1\wedge \mu'}^{\rho}$  and $(c'_2,d'_2)=f_{\sigma_2\wedge \mu'}^{\rho}$.  By the induction hypothesis, $c_1'\leq c_2'$ and $d_1'\leq d_2'$.  We now deal proceed in cases based on whether $v=1$ or $v=-1$.

Suppose first $v=1$. Then 
$$
(c_1,d_1)=f_{\sigma_1\wedge \mu'\wedge v}^{\rho}=(c_1'+1,d_1'+1)\text{ and }(c_2,d_2)=f_{\sigma_2\wedge \mu'\wedge v}^{\rho}=(c_2'+1,d_2'+1).
$$
Combining with the induction hypothesis, this implies $c_1=c_1'+1\leq c_2'+1= c_2$ and $d_1=d_1'+1\leq d_2'+1= d_2$, as desired. If we also knew that $b_1=b_2$ holds, then the induction hypothesis implies $d_1'=d_2'$, and consequently, $d_1=d_1'+1=d_2'+1=d_2$. This finishes the $v=1$ case. 

Suppose now $v=-1$. Then 
$$
(c_1,d_1)=f_{\sigma_1\wedge \mu'\wedge v}^{\rho}=(c_1'+\rho(c_1'+d_1'),d_1'-1)\text{ and }(c_2,d_2)=f_{\sigma_2\wedge \mu'\wedge v}^{\rho}=(c_2'+\rho(c_2'+d_2'),d_2'-1).
$$
Combining this with the fact $\rho$ is increasing and the induction hypothesis, we have 
$$
c_1=c_1'+\rho(c_1'+d_1')\leq c_2'+\rho(c_2'+d_2')=c_2.
$$
This also implies (again using the induction hypothesis) that $d_1=d_1'-1\leq d_2'-1=d_2$.  We have now shown $c_1\leq c_2$ and $d_1\leq d_2$.  Suppose that  $b_1=b_2$ was also true.  Then, by the induction hypothesis, $d_1'=d_2'$.  Consequently, $d_1=d_1'-1=d_2'-1=d_2$. This finishes case $v=-1$, and thus the proof.
\end{proof}

Using Lemma \ref{lem:seq1}, we now prove that given a binary string $\sigma$, altering $\sigma$ by replacing an instance of $(-1,1)$ with $(1,-1)$  does not change the second coordinate of $f_{\sigma}^{\rho}$, and can only make the first coordinate go up.

\begin{lemma}\label{lem:seq2}
Let $\rho$ be a growth function, and let $m\geq 2$ be an integer. Suppose that $\sigma=(\sigma_1,\ldots, \sigma_m)\in \{-1, 1\}^m$ is such that for some $1\leq i\leq m-1$,  $\sigma_i = -1$ and $\sigma_{i+1} = 1$. Let $\phi=(\phi_1,\ldots, \phi_m)\in \{-1,1\}^m$ be defined by setting $\phi_i = 1$, $\phi_{i+1} =-1$, and $\phi_j = \sigma_j$ for all $j\in[m]\setminus \{i, i+1\}$. 

Suppose $f^\rho_{\sigma} = (a_\sigma, b_\sigma)$ and $f^\rho_\phi =(a_\phi, b_\phi)$. Then $a_\sigma\leq a_\phi$ and $b_\sigma = b_\phi$.
\end{lemma}
\begin{proof}
    Let $a,b\in \mathbb{R}$ be such that $f^\rho_{\sigma|_{[i-1]}} = f^\rho_{\phi|_{[i-1]}} = (a, b)$ (these values agree by definition of $\phi$). Since $\sigma_i=0$ and $\phi_i=1$, we have $f^\rho_{\sigma|_{[i]}}= (a+\rho(a+b), b-1)$ and $f^\rho_{\phi|_{[i]}} = (a+1, b+1)$. Further, since $\sigma_{i+1}=1$ and $\phi_{i+1}=0$, we have 
    $$
    f^\rho_{\sigma|_{[i+1]}} = (a+\rho(a+b)+1, b)\text{ and }f^\rho_{\phi|_{[i+1]}} = (a+\rho(a+b+2)+1, b).
    $$
    Since $\rho$ is increasing,  $a+\rho(a+b)+1\leq a+\rho(a+b+2)+1$. By Lemma \ref{lem:seq1}, we can conclude $a_{\sigma}\leq a_{\phi}$ and $b_{\sigma}= b_{\phi}$.
   
\end{proof}

By iterating Lemma \ref{lem:seq2}, one can show that, among the sequences with a given number of $1$'s appearing, the coordinates of $f_{\sigma}^{\rho}$ are maximized when $\sigma$ begins with all $1$'s and ends with all $-1$'s.  

\begin{lemma}\label{lem:seq3}
Let $\rho$ be a growth function, let $m\geq 1$ and $m\geq k\geq 0$ be integers, and let $\sigma\in \{-1,1\}^m$ satisfy $|\sigma|=k$ (see Definition \ref{def:disc}).  

Define $\theta=(\theta_1,\ldots, \theta_m) \in \{-1, 1\}^m$ by setting $\theta_i=1$ if $1\leq i\leq k$ and $\theta_i=-1$ if $k+1\leq i\leq m$.

If $f^\rho_{\sigma} = (a_\sigma, b_\sigma)$ and $f^\rho_{\theta} = (a_\theta, b_\theta)$, then   $a_\sigma\leq a_\theta$ and $b_\sigma = b_\theta$.
\end{lemma}
\begin{proof}
This follows from Lemma \ref{lem:seq2} and the fact that $\theta$ can be obtained from $\sigma$ by a sequence of transpositions, each of which switches an instance of $(-1,1)$ to $(1,-1)$.  
\end{proof}

We are almost ready to state the main result needed for the analysis of our bounds in Theorem \ref{thm:cylinder}.  While the preceding three lemmas considered arbitrary binary sequences, in what follows we need only consider those $\sigma$ which can actually give rise to $(\rho,\sigma)$-chains.  As mentioned above, by Lemma \ref{lem:chain},  we can restrict our attention to $\sigma$ with non-negative discrepancy.

We now compute bounds for $f_{\sigma}^{\rho}$ when $\sigma$ has non-negative discrepancy.   The idea is that,  by Lemma \ref{lem:seq3}, it suffices to compute bounds for the case where $\sigma$ starts with all $1$'s then switches to all $-1$'s.  Since $\sigma$ has non-negative discrepancy, the number of $-1$'s at the end is at most the number of $1$'s at the start. This means $f_{\sigma}^{\rho}$ roughly computes the bounds arising from $|\sigma|$-many $\rho$-matrix deletions applied to some factor of with complexity $(|\sigma|,|\sigma|)$. This allows us to apply Lemma \ref{lem:tau}, which is designed to compute bounds in exactly this scenario.

\begin{lemma}\label{lem:seq4}
Let $d\geq 1$ be an integer, let $C>1$ be a real, and let $\rho$ be a growth function satisfying $\rho(x)\geq x$ and $\rho(x)\leq C x^d$ for all $x\geq 1$. 

Let $m\geq 1$ and $m\geq k\geq 0$ be integers, and let $\sigma\in \{-1,1\}^m$ satisfy $|\sigma|=k$ and $\disc(\sigma)\geq 0$.  If $(a,b)=f_{\sigma}^{\rho}$, then 
\begin{align*}
0&\leq b=2k-m\leq k\text{ and }\\
0&\leq a \leq \tau^\rho_{m-k}(k, k)\leq (2C)^{(m-k)d^{m-k}}(2k)^{d^{m-k}}.
\end{align*}
\end{lemma}
\begin{proof}
We first observe that since $\disc(\sigma)\geq 0$, $k\geq m/2$ and $m-k\leq k$.  This implies $\tau^\rho_{m-k}(k, k)$ is defined and can be bounded using Lemma \ref{lem:tau}. 

As in Lemma \ref{lem:seq3}, define $\theta=(\theta_1,\ldots, \theta_m) \in \{-1, 1\}^m$ by setting $\theta_i=1$ for all $1\leq i\leq k$ and $\theta_i=-1$ for all  $k+1\leq i\leq m$.  Say $(a_{\theta},b_{\theta})=f_{\theta}^{\rho}$. By Lemma \ref{lem:seq3}, $a\leq a_{\theta}$ and $b=b_{\theta}$.  

Since $\theta$ begins with $k$ many $1$'s, $f_{\theta|_{[k]}}^{\rho}=(k,k)$.  Then, since the remaining $m-k$ coordinates of $\theta$ are $-1$'s, we have  $f_{\theta}^{\rho}=(\tau^\rho_{m-k}(k, k), 2k-m)$.  We can immediately conclude $b=b_{\theta}=2k-m\geq 0$, and, by Lemma \ref{lem:tau},
$$
a\leq a_{\theta}= \tau^\rho_{m-k}(k, k)\leq (2C)^{(m-k)d^{m-k}}(2k)^{d^{m-k}}.
$$

\end{proof}

We can finally combine our lemmas to bound the complexity of a factor arising at the end of $(\rho,\sigma)$-chain.

\begin{corollary}\label{cor:chain}
Let $d\geq 1$ be an integer, let $C>1$ be a real, and Let $\rho$ be a growth function satisfying $\rho(x)\geq x$ and $\rho(x)\leq C x^d$ for all $x\geq 1$. 

Let $m\geq 1$ and $m\geq k\geq 0$ be integers, and let $\sigma\in \{-1,1\}^m$ satisfy  $|\sigma|=k$.  Assume 
$$
\calB_0=(\calL_0,\calQ_0)\preceq \cdots \preceq \calB_m=(\calL_m,\calQ_m)
$$
is a $(\rho,\sigma)$-chain  on $\F_p^n$. Then the following hold, where $(\ell,q)$ denotes  the complexity of $\calB_m$.
 \begin{align*}
0&\leq q=2k-m\leq k\text{ and }\\
0&\leq \ell \leq \tau^\rho_{m-k}(k, k)\leq (2C)^{(m-k)d^{m-k}}(2k)^{d^{m-k}}.
\end{align*}
\end{corollary}
\begin{proof}
By Lemma \ref{lem:chain}(a), $\disc(\sigma)\geq 0$.  The conclusion now follows from Lemma \ref{lem:chain}(b) and Lemma \ref{lem:seq4}.
\end{proof}

\section{Warm up: quadratic arithmetic regularity lemma}\label{sec:warmup}

In this section,  we reprove a known quadratic regularity lemma using an energy increment argument and the local inverse theorem (Corollary \ref{cor:inverse}), arriving at a conclusion stated in terms of the local $U^3$ norm from Definition \ref{def:P} (see Theorem \ref{thm:reg1}).  The fact that such a theorem holds will be obvious to the reader familiar with arithmetic regularity lemmas and \cite{P}, and in light of Theorem \ref{thm:relatenorms}, is also equivalent to a formulation appearing in \cite{Terry.2021d}. We reprove this theorem here because such a proof has not yet appeared explicitly in the literature (to our knowledge), and it is closely related to our proof of Theorem \ref{thm:cylinder}.  We also clarify the relationship between this result and other quadratic arithmetic regularity lemmas in the literature, and discuss the difference in the bounds in Theorem \ref{thm:reg1}  versus  Theorem \ref{thm:cylinder}.

For context, we begin by stating a $U^3$ arithmetic regularity lemma due to Green and Tao.  We will use the following notation.

\begin{notation}\label{not:cond}
Given a partition $\calP$  of a group $G$ and a function $f:G\rightarrow [-1,1]$, define $\E(f|\calP)$ by setting $\E(f|\calP)(x)=\E_{y\in P}f(y)$, where $P$ is the element of $\calP$ containing $x$.  

When $\calP=\At(\calB)$ for some quadratic factor $\calB$ on $\F_p^n$, we write $\E(f|\calB)$ to mean $\E(f|\At(\calB))$, where we recall $\At(\calB)$ denotes the partition of $\F_p^n$ into the atoms of $\calB$.
\end{notation}

 \begin{theorem}[Proposition 3.9 in \cite{Green.2007}]\label{thm:globalreg}
Let $\rho,\omega$ be  growth functions with $\omega(x)>0$ for all $x\in \mathbb{R}$, and let $\e\in (0,1)$.  There exists a constant $M=M(p,\rho,\omega,\e)$ such that the following holds. For any function $f:G=\F_p^n\rightarrow [-1,1]$, there are integers $0\leq \ell,q\leq M$ and  a quadratic factor $\calB = (\calL, \calQ)$ of complexity $(\ell, q)$ such that 
$$
f=f_1+f_2+f_3,
$$
where $f_1=\E(f|\calB)$, $\|f_2\|_2\leq \e$, and $\|f\|_{U^3}<1/\omega(\ell+q)$.
 \end{theorem}

 This result is proved using the global $U^3$ inverse theorem of \cite{GT} (Theorem \ref{thm:globalinverse}) and several energy incrementing arguments.  When the rank function $\rho$ is a polynomial and the function $\omega$ is exponential (the case of interest to us), one can check the proof of Theorem \ref{thm:globalreg} in \cite{Green.2007} gives a tower type bound.  Bounds of this form were recently shown to be necessary by Gladkova in \cite{Gladkova.2025}.
 
 In \cite{Terry.2021d}, Theorem \ref{thm:globalreg} was used by the second author and Wolf to prove the following quadratic arithmetic regularity lemma, phrased in terms of the local $U^3$ norm of Definition \ref{def:TW}.

\begin{theorem}\label{thm:reg1}
    There is a constant $K>0$ so that the following holds. Let $\delta\in (0,1)$ and let $\rho$ be a growth function satisfying $\rho(x)>K(x+\log_p(\delta^{-1}))$.  For all $A\subseteq \F_p^n$,  there is a quadratic factor $\calB = (\calL, \calQ)$ on $\F_p^n$ of complexity $(\ell, q)$ and rank at least $\rho(\ell+q)$ such that 
    \begin{enumerate}
    \item $\ell, q\leq M(p,\delta,\rho,K)$
    \item for at least $(1-\delta)$-fraction of the choice of labels $d\in (\F_p^{\ell}\times \F_p^q)^3\times (\F_p^q)^3$, we have 
     $$
    \|1_A-\alpha_{B(\Sigma(d))}\|_{U^3(d)}^{TW}<\delta,
    $$
     where $\alpha_{B(\Sigma(d))}$ is the density of $A$ on the atom $B(\Sigma(d))$ (see Definitions \ref{def:sigma1} and \ref{def:TW}). 
    \end{enumerate}
\end{theorem}

 The proof of  Theorem \ref{thm:reg1} takes as its starting point Theorem \ref{thm:globalreg}  with an exponential choice of $\omega$, and thus obtains tower type bounds when the rank function $\rho$ is a polynomial.  In light of Theorem \ref{thm:relatenorms}, Theorem \ref{thm:reg1} is equivalent to the analogous statement in terms of the local $U^3$ norm from Definition \ref{def:P}.  

\begin{theorem}\label{thm:reg2}
     There is a constant $K>0$ so that the following holds. Let $\delta\in (0,1)$ and let $\rho$ be a growth function satisfying $\rho(x)>K(x+\log_p(\delta^{-1}))$.  For all $A\subseteq \F_p^n$,  there is a quadratic factor $\calB = (\calL, \calQ)$ on $\F_p^n$ of complexity $(\ell, q)$ and rank at least $\rho(\ell+q)$ such that 
    \begin{enumerate}
    \item $\ell, q\leq M(p,\delta,\rho,K)$
    \item for at least $(1-\delta)$-fraction of the choice of labels $b\in \F_p^{\ell}\times \F_p^q$, we have 
     $$
    \|1_A-\alpha_{B(b)}\|_{U^3(b)}^{P}<\delta,
    $$
     where $\alpha_{B(b)}$ is the density of $A$ on the atom $B(b)$ (see Definition \ref{def:P}).
    \end{enumerate}
\end{theorem}

The goal of this section is to give another proof of Theorem \ref{thm:reg2} using a straightforward energy incrementing argument and Prendiville's local inverse theorem, Corollary \ref{cor:inverse}. The bounds we obtain will still be tower. It seems likely these are the correct bounds.  Indeed, Theorem \ref{thm:reg2} implies the conclusion of Theorem \ref{thm:globalreg} holds for the function $1_A$ and the factor $\calB$, with an exponential choice of $\omega$, and it is likely Theorem \ref{thm:globalreg} requires tower type bounds in this case.  The reason this seems likely is that it is known to be the case in the setting of functions (rather than sets) by work of Gladkova \cite{Gladkova.2025}, and one can often (with some work) turn a function example into a set example. 

We will use the same index function as \cite{Green.2007} to track energy  increments.

\begin{definition}\label{def:index}
    For  $A\subseteq G=\F_p^n$ and a partition $\calP$ of $G$, define
    $$
    \ind(A, \calP) = \frac{1}{p^n}\sum_{P\in \calP}\left(\frac{|A\cap P|}{|P|}\right)^2|P|.
    $$
\end{definition}

We will frequently be applying the above definition in the case where $\calP$ arises as the set of atoms of a quadratic factor. In this case, the following notation will be more convenient.

\begin{notation}
Suppose $A\subseteq \F_p^n$ and $\calB$ is a quadratic factor on $\F_p^n$.  We write $\ind(A,\calB)$ to mean $\ind(A,\At(\calB))$, where we recall $\At(\calB)$ is the partition of $\F_p^n$ into the atoms of $\calB$.
\end{notation}

We will use the following standard fact  about this type of index function (see page 21 of \cite{Green.2007}). Specifically, it helps us understand how the index function goes up under refinements.

\begin{fact}[Pythagoras's theorem]\label{fact:index}
Suppose $A\subseteq \F_p^n$ and $\calP'\preceq \calP$ are partitions of $\F_p^n$.  Then
$$
\ind(A, \calP')=\ind(A, \calP)+\frac{1}{p^n}\sum_{P\in \calP}\sum_{\{P'\in \calP': P'\subseteq P\}}(\alpha_P-\alpha_{P'})^2|P'|,
$$
where $\alpha_X$ denotes the density of $A$ on a set $X\subseteq \F_p^n$.
\end{fact}

We now ready to reprove Theorem \ref{thm:reg2}.

\vspace{2mm}

\begin{proofof}{Theorem \ref{thm:reg2}}
Let $K>0$ be sufficiently large.  Fix $\delta\in (0,1)$ and a growth function $\rho$ satisfying $\rho(x)\geq K(x+\log_p(\delta^{-1}))$. Let $C=C(p)$ be as in Corollary \ref{cor:inverse}.  Assume $n\geq 1$ and $A\subseteq G=\F_p^n$.  Let $\alpha=|A|/|G|$, and given a set $X\subseteq G$, let $\alpha_X=|A\cap X|/|X|$.  

We inductively construct a sequence of quadratic factors $\calB_i=(\calL_i,\calQ_i)$ of complexity $(\ell_i,q_i)$ as follows.

    \noindent \underline{Step $0$:} Let $\calB_0 = (\calL_0, \calQ_0)$ be the trivial factor and let $\sigma(0)$ be the empty string. If $\|1_A-\alpha\|_{U^3}^P<\delta$, then let $\calB=\calB_0$ and end the proof.  Otherwise, go to the next step.

    \noindent \underline{Step $i+1$:} Suppose now $i\geq 0$, and assume by induction we have constructed integers $\ell_i,q_i\geq 0$ and a quadratic factor $\calB_i=(\calL_i,\calQ_i)$ on $G$ of complexity $(\ell_i,q_i)$ and rank at least $\rho(\ell_i+q_i)$ such that $\ind(A,\calB_i)\geq iC^{-2}\delta^{2C+2}$ and such that $|\bigcup_{B\in \calJ_i}B|>\delta |G|$, where
    $$
    \calJ_i:=\{B(b)\in \At(\calB_i): \|1_A-\alpha_{B(b)}\|^P_{U^3(b)}> \delta \}.
    $$
    For clarity, in the definition of $\calJ_i$, the norms $\|1_A-\alpha_{B(b)}\|^P_{U^3(b)}$ are all computed relative to the factor $\calB_i$.    By Corollary \ref{cor:inverse}, for each $B\in \calJ_i$ there is a quadratic polynomial $q_B(x) = x^TM_Bx+r_B^Tx+c_B$  such that 
    \begin{align}\label{inv1}
    \left|\sum_{x\in B}(1_A(x)-\alpha_B)e^{2\pi iq_B(x)/p}\right|> C^{-1}\delta^C|B|.
    \end{align}
    Define $\calQ_i' = \calQ_i\cup \{M_B:B\in \calJ_i\}$, let  $\calL_i'$ be a minimal linearly independent set containing $\calL_i$ and spanning $ \{r_B:B\in \calJ_i\}$, and set $\calB_i'=(\calL_i',\calQ_i')$.  Let $(\ell_i',q_i')$ be the complexity of $\calB_i'$. Clearly
    $$
    \ell_i'= |\calL_i'|\leq \ell_i+|\calJ_i|\leq \ell_i+p^{\ell_i+q_i}\text{ and }q_i'=|\calQ_i'|\leq q_i+|\calJ_i|\leq q_i+p^{\ell_i+q_i}.
    $$

    Fix $B\in \At(\calB_i)$. Clearly $q_B(x)$ is constant on all atoms $B'$ of $\calB_i'$.  Given $B'\in \At(\calB_i')$, let $q_B(B')$ denote this constant value.  Since  $|\bigcup_{B\in \calJ_i}B|>\delta |G|$ and by (\ref{inv1}), we have that 
    \begin{align}\label{ineq11}
        \nonumber\delta^{C+1} C^{-1} |G|\leq \sum_{B\in \calJ_i}\delta^C C^{-1}|B|&\leq \sum_{B\in \calJ_i}\left|\sum_{x\in B}(1_A(x)-\a_B)e^{2\pi iq_B(x)/p}\right|\\
        &=\sum_{B\in \calJ_i}\left|\sum_{B'\in \At(\calB_i')}\sum_{x\in B'}(1_A(x)-\a_B)1_B(x)e^{2\pi iq_B(x)/p}\right|\nonumber\\
        &=\sum_{B\in \calJ_i}\left|\sum_{B'\in \At(\calB_i')}e^{2\pi iq_B(B')/p}\sum_{x\in B'}(1_A(x)-\a_B)1_B(x)\right|\nonumber\\
        &=\sum_{B\in \calJ_i}\left|\sum_{\{B'\in \At(\calB_i'): B'\subseteq B\}}e^{2\pi iq_B(B')/p}(\a_{B'}-\a_B)|B'|\right|\nonumber\\
        &\leq \sum_{B\in \calJ_i}\sum_{\{B'\in \At(\calB_i'): B'\subseteq B\}}|B'||\a_{B'}-\a_B|,
    \end{align}
where the last inequality is by the triangle inequality.  By Fact \ref{fact:index}, we have
    \begin{align*}
        \ind(A, \calB_i')-\ind(A, \calB_i) &= \frac{1}{p^n}\sum_{B\in \At(\calB_i)}\sum_{\{B'\in \At(\calB_i'): B'\subseteq B\}}(\alpha_{B}-\alpha_{B'})^2|B'|\\
        &=\E_{x\in G}|\E(1_A|\calB_i')(x)-\E(1_A|\calB_i)(x)|^2\text{(see Notation \ref{not:cond})}.
        \end{align*}
        By Jensen's inequality, this is at least
        \begin{align*}
        \Big(\E_{x\in G}|\E(1_A|\calB_i')(x)-\E(1_A|\calB_i)(x)|\Big)^2&=\left(\frac{1}{p^{n}}\sum_{B\in \At(\calB_i)}\sum_{\{B'\in \At(\calB_i'): B'\subseteq B\}}|\alpha_{B}-\alpha_{B'}||B'|\right)^2\\
        &\geq \left(\frac{1}{p^{n}}\sum_{B\in \calJ_i}\sum_{\{B'\in \At(\calB_i'): B'\subseteq B\}}|\alpha_{B}-\alpha_{B'}||B'|\right)^2\\
        &\geq \left( \delta^{C+1}C^{-1} \right)^2\\
        &=\delta^{2C+2}C^{-2},
        \end{align*}
where the third inequality is by (\ref{ineq11}).        Combining, we have that
       \begin{align*}
\ind(A,\calB_i')-\ind(A,\calB_i)\geq \delta^{2C+2}C^{-2}.
       \end{align*}

   Apply Lemma \ref{lem:rank2} to obtain a factor $\calB_{i+1} = (\calL_{i+1}, \calQ_{i+1})$ refining $\calB_i'$, of complexity  $(\ell_{i+1}, q_{i+1})$ and rank at least $\rho(\ell_{i+1}+ q_{i+1})$, for some 
   $$
   q_{i+1}\leq q_i'\leq q_i+p^{\ell_i+q_i}
   $$
   and
   $$
   \ell_{i+1}\leq \tau_{q_i'}^{\rho}(\ell_i',q_i') \leq \tau_{q_i+p^{\ell_i+q_i}}(\ell_i+p^{\ell_i+q_i},q_i+p^{\ell_i+q_i}).
    $$
By the Fact \ref{fact:index}, the above, and our induction hypothesis, we have 
    $$
    \ind(A, \calB_{i+1})\geq \ind(A,\calB_i')\geq \ind(A, \calB_i)+\delta^{2C+2}C^{-2} \geq (i+1)C^{-2} \delta^{2C +2}.
    $$
    Let 
    $$
    \calJ_{i+1}=\{B(b)\in \At(\calB_{i+1}): \|1_A-\alpha_B\|^P_{U^3(b)}\geq \delta\},
    $$
     where the norms in $\calJ_{i+1}$ are computed relative to the factor $\calB_{i+1}$. If $|\bigcup_{B\in \calJ_{i+1}}B|\leq \delta |G|$, let $\calB'=\calB_{i+1}$ and end the proof.  Otherwise, go to the next step. 
    
    Clearly we may repeat this process at most some $t\leq C^{2} \delta^{-2C-2}$ many times. 
\end{proofof}

\vspace{2mm}

We now consider the bound generated by  the proof above, in the case where $\rho$ is polynomial.  Clearly the inductive process described in the proof ends at some stage $t\leq \poly(\delta^{-1})$ because of the increase in the index at each step.  We can see from the  proof  that  the sum $\ell_t+q_t$ is bounded by $\phi^{(t)}(0)$, where $\phi(x)=\tau_{x+p^{x}}^{\rho}(x+p^x, x+p^x)$.  When the rank function $\rho$ is a polynomial, Lemma \ref{lem:tau} implies 
$$
\tau_{x+p^{x}}^{\rho}(x+p^x, x+p^x)\leq  \exp_p(\exp_p(\exp_p(O(x)))),
$$
 in which case $\phi^{(t)}(0)$ is clearly bounded by a tower-type function in a power of $\delta^{-1}$. 
    
Our proof of the ``cylinder" version of this result (Theorem \ref{thm:cylinder}) will differ from the proof of Theorem \ref{thm:reg1} in several crucial ways, with the goal of increasing the complexity at most polynomially at every step (rather than exponentially as occurs above).   First, we will allow our partition to use atoms of distinct factors. Then, each time we apply the localized inverse theorem to an atom, we use refine only the factor involved with that specific atom.  This will allow  us to avoid taking the common refinements at each step, as occurs above. Second,   we avoid applying Lemma \ref{lem:rank} at any stage, instead choosing to interweave individual matrix deletions with applications of the inverse theorem.  A detailed analysis of the bounds generated by such a process will show they are similar to a single application of Lemma \ref{lem:rank}.

\section{``Cylinder" quadratic arithmetic regularity}\label{sec:cylinder}

In this section we prove our   ``cylinder" version of the quadratic arithmetic regularity lemma.   This result, Theorem \ref{thm:cylinder} below, partitions $\F_p^n$ into atoms of (possibly distinct) high rank quadratic factors, so that most of the group is covered by atoms which are uniform with respect to $A$ in the sense of Definition \ref{def:P}. We avoid in this section the decorated notation set out in Definition \ref{def:P}, instead working with the definition itself.  We do this as we will be working with partitions $\calP$, one part  of which we will frequently denote with a $P$.   

\begin{theorem}\label{thm:cylinder}
  For all $\d\in (0,1)$ and all growth functions $\rho$, there exists a constant $M=M(p,\d,\rho)$ so that the following holds. For all $A\subseteq \F_p^n = G$, then there is a partition $\calP$ of $G$ such that for each $P\in \calP$ there is a quadratic factor $\calB_P = (\calL_P, \calQ_P)$ of complexity $(\ell_P, q_P)$ and rank at least $\rho(\ell_P+q_P)$ such that $\ell_P, q_P\leq M$, such that $P\in \At(\calB_P)$, and such that $|\bigcup_{P\in \calJ}P|\leq \delta|G|$, where 
   $$
   \calJ=\{P\in \calP: \|(1_A-\alpha_{P})1_{P}\|_{U^3}> \delta\|1_{P}\|_{U^3},\text{ where   $\alpha_P$ denotes the density of $A$ on $P$}\}.
   $$   
   Moreover, if $\rho$ is  polynomial of degree $d>1$, then  for all $P\in \calP$, 
   \begin{align*}
q_P\leq O_p(\delta^{-O_p(1)}) \text{ and }\ell_P\leq \exp_p(\exp_p(\exp_p(O_{\rho,p}(\delta^{-O_p(1)})))).
   \end{align*}
   Moreover, if $\rho$ is  polynomial of degree $1$, then for all $P\in \calP$,
    \begin{align*}
q_P\leq O_p(\delta^{-O_p(1)}) \text{ and }\ell_P\leq \exp_p(\exp_p(O_{\rho,p}(\delta^{-O_p(1)}))).
   \end{align*}
\end{theorem}
 
 Before we proceed to the formal proof, we give an informal outline of the strategy, which consists of an algorithm generating at each step a partition of the group. 

The process begins by defining $\calP_0$ to be the trivial partition of $G$.  At step $i+1$, we will be given a partition $\calP_i$ of $G$ in which every $P\in \calP_i$  is an atom of a quadratic factor, $\calB_P$. If $\calP_i$ has all the  desired properties, we end the construction. Otherwise, we do one of two types of induction step  to generate a new partition $\calP_{i+1}$:

\vspace{2mm}

\noindent{\bf Inductive step type $1$:} If it is the case that all factors $\calB_P$ have high rank (as $P$ ranges over the elements of $\calP_i$), we can conclude that many atoms fail to be uniform with respect to $A$, since otherwise the algorithm would have already halted.   For each $P\in \calP_i$ that is already sufficiently uniform with respect to $A$, we leave $P$ unchanged.   For each $P\in \calP_i$ that fails to be sufficiently uniform with respect to $A$, we apply the inverse theorem localized to $P$ (see Corollary \ref{cor:inverse}), to obtain a quadratic polynomial correlating with $1_A-\alpha_P$ on $P$.  From this quadratic polynomial, we generate a refinement $\calB_P'\preceq \calB_P$ which adds at most one new linear term and at most one new quadratic term to $\calB_P$. We then delete the set $P$ from the partition, and in its place, we add the atoms of $\calB_P'$ which union to $P$.   We check that after doing this on all the non-uniform $P$, the index of the partition has gone up (see Definition \ref{def:index}), and then we go the next step.

\vspace{2mm}

\noindent{\bf Inductive step type $-1$:} If it is not the case that all the factors $\calB_P$ have high rank (as $P$ ranges over the elements of $\calP_i$), we do the following.  For each $P\in \calP$ such that $\calB_P$ has high rank, we leave $P$ unchanged.  For each $P\in \calP$ such that $\calB_P$ has low rank, we preform one $\rho$-matrix deletion (see Definition \ref{def:rhodeletion}) to obtain a new factor $\calB_P'$. We then delete $P$ from the partition, and in its place, we add the atoms of $\calB_P'$ which union to $P$.      We then go to the next step.

\vspace{2mm}

We keep track of how this process unfolds by associating a binary string $\sigma(P)$ (see Notation \ref{not:binary}) to each set $P$ in our partitions as follows.   

\begin{itemize}
\item Suppose at step $i+1$ we perform an inductive step of type 1. If $P\in \calP_{i}$ was left unchanged in $\calP_{i+1}$, we leave $\sigma(P)$ unchanged.  On the other hand, if $P\in \calP_{i}$ got deleted and replaced with subatoms in $\calP_{i+1}$,  then for each $P'\in \calP_{i+1}$ with $P'\subseteq P$, we define $\sigma(P')=\sigma(P)\wedge 1$.  By construction, $\calB_{P'}$ is obtained by adding at most one linear and at most one quadratic constraint to $\calB_P$ in this case.
\item Suppose at step $i+1$ we perform an inductive step of type $-1$. If $P\in \calP_{i}$ was left unchanged in $\calP_{i+1}$, we leave $\sigma(P)$ unchanged.  On the other hand, if $P\in \calP_{i}$ got deleted and replaced with subatoms in $\calP_{i+1}$,  then for each $P'\in \calP_{i+1}$ with $P'\subseteq P$, we define $\sigma(P')=\sigma(P)\wedge -1$.  By construction, $\calB_{P'}$ is a $\rho$-matrix deletion of $\calB_P$ in this case.
\end{itemize}

At the end of the process, we will obtain a partition $\calP$ of $G$ such that each $P\in \calP$ is associated to a binary string $\sigma(P)$ and a quadratic factor $\calB_P$, such that $P$ is an atom of $\calB_P$, and such that $\calB_P$ is the final factor in  a $(\rho,\sigma(P))$-chain (see Definition \ref{def:chain}).  We will see that each binary string $\sigma(P)$ cannot contain too many $1$'s, as each new $1$  corresponded to an increase in the index.   We can then bound the complexity of the  factors $\calB_P$ using Corollary \ref{cor:chain}. Indeed, Corollary \ref{cor:chain} was precisely designed to handle $(\rho,\sigma)$-chains where $\sigma$ has boundedly many $1$'s.   This ends our outline of the proof of Theorem \ref{thm:cylinder}.

Before proceeding to the formal proof, we prove a lemma allowing us to perform ``Inductive steps of type 1" quickly.

\begin{lemma}
\label{lem:energy}
    Let $C=C(p)$ be as in Corollary \ref{cor:inverse}. Let $\delta\in (0,1)$, let $A\subseteq \F_p^n = G$, and let $\calP$ be a partition of $G$ such that for each $P\in \calP$ there is a quadratic factor $\calB_P = (\calL_P, \calQ_P)$ of complexity at most $(\ell_P, q_P)$ and rank at least $C(\ell_P+q_P+\log_p(\delta^{-1}))$ such that $P\in \At(\calB_P)$.  Suppose $|\bigcup_{P\in \calJ}P|>\delta|G|$ where
   $$
   \calJ=\{P\in \calP: \|(1_A-\alpha_{P})1_{P}\|_{U^3}> \delta\|1_{P}\|_{U^3},\text{ where   $\alpha_P$ denotes the density of $A$ on $P$}\}.
   $$
    Then there is a partition $\calP'$ refining $\calP$ such that $\ind(A,\calP')\geq \ind(A, \calP)+C^2\d^{2C+2}$, and such that for each $P'\in \calP'$ one of the following hold.
    \begin{itemize}
    \item $P'\in \calP\setminus \calJ$,
    \item  $P'\subseteq P$ for some $P\in \calJ$, and morevoer, $P'$ is an atom of a factor $\calB_{P}' = (\calL_{P}', \calQ_{P}')\preceq \calB_P$ of complexity $(\ell_P',q_P')$ where $\calL_P\subseteq \calL_P'$, $\calQ_P\subseteq \calQ_P'$, $\ell_P'\leq \ell_P+1$, and $q_P'\leq q_P+1$. 
    \end{itemize}
\end{lemma}

\begin{proof}
    Let $C=C(p)$ be as in Corollary \ref{cor:inverse}. Assume $\delta,A,\calP,\calJ$ are as in the hypotheses.   By Corollary \ref{cor:inverse}, for each $P\in \calJ$ there is a quadratic polynomial $q_P(x) = x^TM_Px+r_P^Tx+c_P$  such that 
    \begin{align}\label{inv2}
        \left|\sum_{x\in G}(1_A-\alpha_P)1_Pe^{2\pi iq(x)/p}\right|\geq C^{-1}\delta^{C}|P|.
    \end{align}
    For each $P\in \calP$, define a quadratic factor $\calB_P' = (\calL_P' , \calQ'_P)$, where $\calL_P'$ is a minimal set of vectors containing $\calL_P$ and spanning $r_P$, and  $\calQ_P'=\calQ_P\cup \{M_P\}$. Note that for each $P\in \calP$, $q_P(x)$ is constant on all atoms $B'$ of $\calB_P'$, so given $B'\in \At(\calB_P')$, we can let $q_P(B')$ denote this constant value. We then have, since $|\bigcup_{P\in \calJ}P|\geq \delta |G|$, that 
    \begin{align}\label{ineq1}\begin{split}
        C^{-1}\delta^{C+1} |G|\leq C^{-1}\delta^C\sum_{P\in \calJ}|P|&=\sum_{P\in \calJ}C^{-1}\delta^C|P|\\
        &\leq \sum_{P\in \calJ}\left|\sum_{x\in P}(1_A(x)-\a_P)e^{2\pi iq_P(x)/p}\right|\\
        &=\sum_{P\in \calJ}\left|\sum_{B'\in \At(\calB_P')}\sum_{x\in B'}(1_A(x)-\a_P)1_P(x)e^{2\pi iq_P(x)/p}\right|\\
        &=\sum_{P\in \calJ}\left|\sum_{B'\in \At(\calB_P')}e^{2\pi iq_P(B')/p}\sum_{x\in B'}(1_A(x)-\a_P)1_P(x)\right|\\
        &=\sum_{P\in \calJ}\left|\sum_{\{B'\in \At(\calB_P'): B'\subseteq P\}}e^{2\pi iq_B(B')/p}(\a_{B'}-\a_P)|B'|\right|\\
        &\leq \sum_{P\in \calJ}\sum_{\{B'\in \At(\calB_P'): B'\subseteq P\}}|B'||\a_{B'}-\a_P|,\end{split}
    \end{align}
    where second inequality is by (\ref{inv2}), and the last inequality is by the triangle inequality.  Now, define $\calP'$ as follows:
    $$\calP' = \{P:P\in \calP\setminus \calJ\}\cup \bigcup_{P\in \calJ}\{B\in \At(\calB_P'):B\subset P\}.$$
    We estimate the difference between $\ind(A, \calP)$ and $\ind(A, \calP')$. By Pythagoras theorem (Fact \ref{fact:index}), we have the following (see also Notation \ref{not:cond})
    \begin{align*}
        \ind(A, \calP)-\ind(A, \calP') &=\frac{1}{p^n}\sum_{P\in \calP}\sum_{\{P'\in \At(\calB_P'): P'\subseteq P\}}(\alpha_{P}-\alpha_{P'})^2|P'|\\
        &=\E_{x\in G}|\E(1_A|\calP')(x)-\E(1_A|\calP)(x)|^2.
        \end{align*} 
By Jensen's inequality, this is at least 
\begin{align*}
\Big(\E_{x\in G}|\E(1_A|\calP')(x)-\E(1_A|\calP)(x)|\Big)^2&=\frac{1}{p^{2n}}\left(\sum_{P\in \calP}\sum_{\{P'\in \At(\calB_P'): P'\subseteq P\}}|\alpha_{P}-\alpha_{P'}||P'|\right)^2\\
&\geq\frac{1}{p^{2n}}\left(\sum_{P\in \calJ}\sum_{\{P'\in \At(\calB_P'): P'\subseteq P\}}|\alpha_{P}-\alpha_{P'}||P'|\right)^2\\
        &\geq \frac{1}{p^{2n}} \left(C^{-1}\delta^{C+1}|G|\right)^2\\
        &= C^{-2}\d^{2C+2},
    \end{align*}
    where the second inequality is by (\ref{ineq1}). 
    
\end{proof}

We are now ready to prove Theorem \ref{thm:cylinder}.  The reader may wish to review Notation \ref{not:binary} and Definition \ref{def:disc} before reading the proof.

\vspace{2mm}

 \begin{proofof}{Theorem \ref{thm:cylinder}}
    Let $C=C(p)$ be as in Corollary \ref{cor:inverse}. Fix $\delta\in (0,1)$ and a growth function $\rho$.  After replacing $\rho$ if necessary, we assume without loss of generality that $\rho(x)\geq C(x+\log_p(\delta^{-1}))$ for all $x\geq 1$.  Fix a set $A\subseteq G=\F_p^n$.  Throughout the proof, for a set $X\subseteq G$, $\alpha_X$ denotes the density of $A$ on the set $X$.
    
    We describe an inductive process, where at step $i$ we either obtain the desired partition $\calP$, or we generate a partition $\calP_i$ of our group, and to each $P\in \calP_i$, we  associate a binary string $\sigma(P)\in \{-1,1\}^{\leq i}$.  At each step $i$ in the process, we will ensure that for each $P\in G$, $\ind(A,\calP_i)\geq |\sigma(P)|C^2\delta^{-2C-2}$, and that $P$  is an atom of a factor $\calB_P$ which arises as the last element of a $(\rho,\sigma(P))$-chain.

    \vspace{2mm}

    \noindent \underline{\textbf{Base Case}}: Let $\calP_0=\{G\}$ be the trivial partition, and define $\sigma(G)=<>$ (the empty string). Trivially, we have $\ind(A,\calP_0)\geq 0=|\sigma(G)|C^2\delta^{-2C-2}$. Moreover, $G$  is an atom of the trivial factor, which is by definition the last (and first) element of any $(\rho,<>)$-chain.

    \noindent \underline{\textbf{Induction Step}}: Suppose $i\geq 0$, and assume by induction we have  a partition $\calP_i$ of $G$, and for each $P\in \calP_i$ a sequence $\sigma(P)\in \{-1,1\}^{\leq i}$ such that  $P$ is an atom of a quadratic factor $\calB_P$ with complexity $(\ell_P,q_P)$, such that $\calB_P$ is the last element in a $(\rho,\sigma(P))$-chain on $\F_p^n$, and such that $\ind(A,\calP_i)\geq |\sigma(P)|C^2\delta^{-2C-2}$.  Define two subsets of $\calP_i$ as follows.
    \begin{align*}
    \calJ_i&=\{P\in \calP: \|(1_A-\alpha_P)1_P\|_{U^3}\geq \delta \|1_P\|_{U^3}\}\text{ and }\\
    \calS_i&=\{P\in \calP: \rk(\calB_P)<\rho(\ell_P+q_P)\}.
    \end{align*}
    If $|\bigcup_{P\in \calJ_i}P_i|<\d|G|$ and $\calS_i=\emptyset$,  let $\calP = \calP_i$ and end the construction. Otherwise we proceed in two cases. 

    \vspace{2mm}

    \noindent{\bf Inductive step type $-1$:} Suppose first that $\calS_i\neq \emptyset$.  Fix $P\in \calS_i$. Then there are $\lambda_1, \ldots, \lambda_{q_P}$ not all zero such that 
    $$\rk(\lambda_1M_1+\ldots+\lambda_{q_P}M_{q_p})<\rho(\ell_P+ q_P).$$
    Without loss of generality, suppose $\lambda_{q_P}\neq 0$. Set $U := \lambda_1M_1+\ldots+\lambda_{q_P}M_{q_p}$, and let $\calL_{P}'$ be a minimal set of vectors containing $\calL_P$ and spanning $\ker(U)^\perp$, and let $\calQ_P' = \calQ_P\setminus \{M_{q_P}\}$. Then $\calB'_P := (\calL_P', \calQ_P')$ is a refinement of $\calB_P$ of complexity $(\ell_{P'},q_{P'})$, where $q_{P'}=q_P-1$ and $\ell_{P'}\leq \ell_P+\rho(\ell_P+q_P)$. By the inductive hypothesis, $\calB_P$ is the final factor appearing in a $(\rho,\sigma(P))$-chain. Thus, $\calB_P'$ is the final factor appearing in a $(\rho,\sigma(P)\wedge -1)$-chain.  For each $P'\in \At(\calB_P')$ satisfying $P'\subseteq P$, set $\sigma(P')=\sigma(P)\wedge -1$, set $\calB_{P'}=\calB_P'$, and let $(\ell_{P'},q_{P'})$ be the complexity of $\calB_{P'}$.  We note that, for such $P'$,  $|\sigma(P')|=|\sigma(P)|$, and thus our induction hypotheses imply 
    $$
    \ind(A,\calP_{i+1})\geq \ind(A,\calP_i)\geq |\sigma(P)|C^2\delta^{-2C-2}=|\sigma(P')|C^2\delta^{-2C-2},
    $$
   where the first inequality is because $\calP_{i+1}$ refines $\calP_i$, and thus has index at least that of $\calP_i$ by Fact \ref{fact:index}.  We then define our new partition to be  
    $$\calP_{i+1} = (\calP_i\setminus\calS_i)\cup \bigcup_{P\in \calS_i}\{P'\in \At(\calB'_P):P'\subseteq P\}.$$
   By construction, the inductive hypotheses, and the fact $\ind(A,\calP_{i+1})\geq \ind(A,\calP_i)$, we can conclude that for all $P\in \calP_{i+1}$, $P$ is an atom of a quadratic factor $\calB_P$ with complexity $(\ell_P,q_P)$ that arises as the last element in a $(\rho,\sigma(P))$-chain, and moreover, that $\ind(A,\calP_{i+1})\geq |\sigma(P)|C^2\delta^{-2C-2}$.

\vspace{2mm}

    \noindent{\bf Inductive step type 1:} Suppose now $\calS_i=\emptyset$. Then we must have $|\bigcup_{P\in \calJ_i}P|>\delta |G|$, so we can apply Lemma \ref{lem:energy}  to obtain a partition $\calP_{i+1}\preceq \calP_i$ such that the following hold.
    \begin{enumerate} 
    \item $\ind(A,\calP')\geq \ind(A, \calP)+C^{-2}\d^{2C+2},$\label{ind}
    \item for each $P'\in \calP_{i+1}'$, either $P'\in \calP_i\setminus \calJ_i$, or $P'\subseteq P$ for some $P\in \calJ_i$, and in this case, $P'$ is an atom of a factor $\calB_{P'} = (\calL_{P'}, \calQ_{P'})\preceq \calB_P$ of complexity $(\ell_{P'},q_{P'})$ satisfying $\calL_P\subseteq \calL_{P'}$, $\calQ_P\subseteq \calQ_{P'}$, $\ell_{P'}\leq \ell_P+1$, and $q_{P'}\leq q_P+1$. 
    \end{enumerate}
 For each $P\in \calJ_i$, and $P'\in \calP_{i+1}$ with $P'\subseteq P$,   define $\sigma(P')=\sigma(P)\wedge 1$.  Recall that by the induction hypothesis, $\calB_P$ is the final factor in a $(\rho,\sigma(P))$-chain, and thus, by construction, $\calB_{P'}$    is the final factor in a $(\rho,\sigma(P'))$-chain.  Moreover, by induction, by (1) above, and since $|\sigma(P')|=|\sigma(P)|+1$, we have
 $$
\ind(A,\calP_{i+1})\geq \ind(A, \calP)+C^{-2}\d^{2C+2}\geq |\sigma(P)|C^{-2}\d^{2C+2}+C^{-2}\d^{2C+2}=|\sigma(P')|C^{-2}\d^{2C+2}.
 $$
    Thus, by construction, the inductive hypotheses, and the fact $\ind(A,\calP_{i+1})\geq \ind(A,\calP_i)$, we can conclude that for all $P\in \calP_{i+1}$, $P$ is an atom of a quadratic factor $\calB_P$ with complexity $(\ell_P,q_P)$ that arises as the last element in a $(\rho,\sigma(P))$-chain, and moreover, that $\ind(A,\calP_{i+1})\geq |\sigma(P)|C^2\delta^{-2C-2}$.

    This completes our description of step $i+1$ of the construction.

    \vspace{2mm}

    We claim this process halts after at some finite number of steps $t$. First, we observe that at every step $i$, and for every $P\in \calP_i$, we have $\ind(A,\calP_i)\geq |\sigma(P)|C^{-2}\delta^{2C+2}$, and thus, $\sigma(P)$ can contain at most $C^{2}\d^{-2C-2}$ many $1$'s.   On the other hand, by construction, there also exists a $(\rho,\sigma(P))$-chain on $\F_p^n$. By Lemma \ref{lem:chain}(a), this implies we must have $\disc(\sigma(P))\geq 0$, so $\sigma(P)$ contains at most $C^{2}\d^{-2C-2}$ many $-1$'s (since it has at most that many $1$'s).   Consequently, no $\sigma(P)$ can have length longer than $2C^{2}\d^{-2C-2}$.  Since each step in the process adds  $-1$ or $1$ to some $\sigma(P)$,  there can be only finitely many steps  total (by K\H{o}nig's lemma).   
    
    Thus, at the end of the process, we arrive at a partition $\calP:=\calP_t$, and for each $P\in \calP$, a binary string $\sigma(P)\in \{-1,1\}^{\leq t}$ with the following properties.
    \begin{enumerate}[(i)]
    \item $|\bigcup_{P\in \calJ}P|<\d|G|$ where $\calJ = \{P\in \calP:\|(1_A-\alpha_P)1_P\|_{U^3}\geq \delta\|1_P\|_{U^3}\}$,
    \item for each $P\in \calP$, $P$ is an atom of a factor $\calB_P$ of complexity $(\ell_P,q_P)$ and rank at least $\rho(\ell_P+q_P)$, 
    \item for each $P\in \calP$, $\calB_P$ is the final factor in a $(\rho,\sigma(P))$-chain on $\F_p^n$,
    \item for each $P\in \calP$, $\ind(A,\calP)\geq |\sigma(P)|C^{-2}\d^{2C+2}$.
    \end{enumerate}

    Clearly item (iv) implies that for all $P\in \calP$, $|\sigma(P)|\leq C^2\delta^{-2C-C}$.  Combining this with (iii), we see that Corollary \ref{cor:chain} implies 
\begin{align*}
&0\leq q_P\leq  2C^2\delta^{-2C-2}\text{ and }\\
&0\leq \ell_P\leq (2K)^{2C^2\delta^{-2C-2}d^{2C^2\delta^{-2C-2}}}(4C^2\delta^{-2C-2})^{d^{2C^2\delta^{-2C-2}}}.
\end{align*}
Clearly the bound for $q_P$ above has the form $O_{p}(\delta^{-O_{p}(1)})$.  If $d=1$, the bound for $\ell_P$ above has the form $\exp_p(\exp_p(O_{p,\rho}(\delta^{-O_p(1)})))$, while if $d>1$, then the bound for $\ell_P$ above has the form $\exp_p(\exp_p(\exp_p(O_{p,\rho}(\delta^{-O_p(1)}))))$. This finishes the proof.
\end{proofof}

\section{Putting it all together}\label{sec:sumup}

We can now put things together to prove the main theorem. It will be convenient to use the following terminology from \cite{Terry.2021a,Terry.2021d} for an ``almost $\e$-homogeneous partition" relative to a distinguished subset of a group.

\begin{definition}[Almost homogeneous partitions]
Let $G$ be a finite group and let $A\subseteq G$. We say a partition $\calP$ of $G$ is \emph{almost $\e$-homogeneous with respect to $A$} if $|\bigcup_{P\in \Sigma}P|\geq (1-\e)|G|$, where 
$$
\Sigma=\{P\in \calP: |A\cap P|\geq (1-\e)|P|\text{ or }|A\cap P|\leq \e |P|\}.
$$
\end{definition}

We will use the following simple averaging fact (see Fact 4.25 in \cite{Terry.2021a} for a proof).

\begin{fact}\label{fact:hom}
Suppose $A\subseteq G=\F_p^n$. If $\calP$ is a partition of $G$ which is almost $\e$-homogeneous with respect to $A$ and $\calP'$ is a refinement of $\calP$, then $\calP'$ is almost $2\sqrt{\e}$-homogeneous with respect to $A$. 
\end{fact}

 For convenience, we state a simplified version of Theorem \ref{thm:cylinder} which holds for linear growth functions, as this is the version we will need. We use the letter $\calR$ to denote partitions to avoid confusion with the notation from Definition \ref{def:P}.

\begin{theorem}\label{thm:cylinderlinear}
   For all $K>0$ there exists $M=M(K,p)$ so that the following holds.  For all $A\subseteq G=\F_p^n$, there exists a partition $\calR$ of $G$ such that for each $R\in \calR$ there is a quadratic factor $\calB_R = (\calL_R, \calQ_R)$ of complexity $(\ell_R, q_R)$ and a label $b_R\in \F_p^{\ell_R}\times \F_p^{q_R}$ such that the following hold.
   \begin{enumerate}
   \item $R=B_R(b_R)\in \At(\calB_R)$, $\calB_R$ has rank at least $K(\ell_R+q_R+\log_p(\delta^{-1}))$, and
   $$
  \ell_R+q_R\leq p^{\delta^{-M}};
   $$
   \item  $|\bigcup_{R\in \calJ}R|\leq \delta|G|$, where $\calJ$ is the set of $R\in \calR$ such that  $\|1_A-\alpha_{R}\|_{U^3(b_R)}^P\geq  \delta$, where $\alpha_R$ denotes the density of $A$ on $R$, and the norm $\|1_A-\alpha_{R}\|_{U^3(b_R)}^P$ is computed relative to the factor $\calB_R$.
   \end{enumerate}
\end{theorem}
\begin{proof}
This is an immediate corollary of Theorem \ref{thm:cylinder} in the case of a linear growth function (i.e. $d=1$) and an inspection of Definition \ref{def:P}.
\end{proof}

Finally, we will use the following result, proved in \cite{Terry.2021f}, which says that given a set of bounded $\VC_2$-dimension, any high rank quadratic factor which is almost homogeneous with respect to $A$ can be replaced with a (possibly different) factor which is still almost homogeneous with respect to $A$, which has the same linear component and rank as the original factor, and which has small quadratic complexity.\footnote{The statement below is slightly more detailed than that appearing in \cite{Terry.2021f}, but follows directly from its proof there.}

\begin{theorem}[Theorem 5.4 of \cite{Terry.2021f}]\label{thm:quad}
For all integers $k\geq 0$, there exists constant $K=K(k,p)>0$ and $\delta_0=\delta_0(k)>0$ so that for all $0<\delta<\delta_0$ and all $0<\e\leq (\delta/120)^{k+2}$,  the following holds for all sufficiently large $n$.

Suppose $A\subseteq G=\mathbb{F}_p^n$ has $\VC_2$-dimension at most $k$, and suppose there exist  integers $\ell,q\geq 0$ and a quadratic factor $\calB=(\calL,\calQ)$ on $G$ of complexity $(\ell,q)$ and rank at least $K(\ell+q+\log_p(\delta^{-1}))$ which is almost $\e$-homogeneous with respect to $A$.  Then there exist an integer $q'\geq 0$ and a quadratic factor $\calB'=(\calL,\calQ')$ on $\F_p^n$ of complexity $(\ell,q')$ with $\rk(\calB')=\rk(\calB)$ such that  
\begin{enumerate} 
\item $q'\leq \log_p( \delta^{-k-o_{k,p}(1)})$;
\item there exists a union $Y$ satisfying $|A\Delta Y|\leq 16\delta^{1-o_{k,p}(1)}|G|$,where $o_{k,p}(1)$ tends to $0$ as $\delta$ tends to $0$ at a rate depending on $k$ and $p$.
\end{enumerate}
\end{theorem}

We now prove the main theorem.

\vspace{2mm}

\begin{proofof}{Theorem \ref{thm:main}}
Fix an integer $k\geq 1$ and let $\delta_0=\delta_0(k)$ be from Theorem \ref{thm:quad}. Fix any $0<\delta<\delta_o(k)$ and set $\e=(\delta/120)^{k+2}$.  To ease notation, let $\mu=(\e^2/8)^{k^22^{k^2}}2^{-1}$.  Choose $K>0$ sufficiently large compared to $k$ and $p$, and let $M=M(K,p)$ be as in Theorem \ref{thm:cylinderlinear}.  Let $\rho$ be a polynomial growth function. After possibly replacing $\rho$, we may assume without loss of generality that $\rho(x)\geq x$ for all $x\geq 1$. Let $C>1$ be such that for all $x\geq 1$, $\rho(x)\leq Cx^d$.  

Fix $A\subseteq G=\F_p^n$ a subset with $\VC_2$-dimension at most $k$.  Given $X\subseteq G$, we let $\alpha_X$ denote the density of $A$ on the set $X$, i.e. $\alpha_X=|A\cap X|/|X|$.  By Theorem \ref{thm:cylinder} there exists a partition $\calR$ of $G$ so that the following hold.
\begin{enumerate}
\item for each $R\in \calR$, there are integers $\ell_R,q_R\geq 0$, a quadratic factor $\calB_R=(\calL_R,\calQ_R)$ on $\F_p^n$ of complexity $(\ell_R,q_R)$, and $b_R\in \F_p^{\ell_R}\times \F_p^{q_R}$ such that 
\begin{enumerate}
\item $\calB_P$ has rank at least $K(\ell_R+q_R+\log_p(\mu^{-1}))$;
\item $\ell_R+q_R\leq p^{\mu^{-M}}$;
\item $R=B_R(b_R)\in \At(\calB_R)$.
\end{enumerate}
\item The union $|\bigcup_{R\in \calJ}R|$ has size at most $\mu|G|$, where $\calJ$ is the set of $R\in \calR$ satisfying $\|1_A-\alpha_R\|^P_{U^3(b_R)}\geq \mu$ (where the norm is computed relative to the factor $\calB_R$).
\end{enumerate}
We now make a few observations about this partition.  First, we observe that by (a), (c), and Lemma \ref{lem:sizeofatoms}, we have that for all $R\in \calR$,
$$
|R|\geq (1- \mu)p^{n-\ell_R-q_R}\geq (1-\mu)p^{n-p^{\mu^{-M}}}.
$$
Consequently, 
$$
|\calR|\leq \frac{|G|}{(1-\mu)p^{n-p^{\mu^{-M}}}}= p^{p^{\mu^{-M}}}(1-\mu)^{-1}.
$$
Further, we observe that (a), the definition of $\calJ$, and Corollary \ref{cor:key} imply that for all $R\in \calR\setminus \calJ$, $\alpha_R\in [0,\frac{\e^2}{2})\cup (1-\frac{\e^2}{2},1]$.  Thus, combining with (2) above, we have that $\calR$ is almost $\e^2/2$-homogeneous with respect to $A$ (recall $\mu<\e^2/2$ by definition).

We now define an initial quadratic factor by simply combining the individual factors generated by the $\calB_R$ as $R$ ranges over the elements in $\calR\setminus \calJ$.  Specifically, we define $\calB=(\calL,\calQ)$ where $\calQ=\bigcup_{R\in \calR\setminus \calJ}\calQ_R$, and $\calL$ is a minimal set of vectors spanning $\bigcup_{R\in \calJ}\calL_R$. By construction, $\calB$ has complexity $(\ell,q)$ for some integers $\ell,q\geq 0$ satisfying
$$
\ell\leq \sum_{R\in \calJ}\ell_R\leq |\calR|p^{\mu^{-M}}\leq p^{p^{\mu^{-M}}}(1-\mu)^{-1}p^{\mu^{-M}}=(1-\mu)^{-1}p^{p^{\mu^{-M}}+\mu^{-M}},
$$
and similarly,
$$
q\leq \sum_{R\in \calJ}q_R\leq |\calR|p^{\mu^{-M}}\leq(1-\mu)^{-1} p^{p^{\mu^{-M}}+\mu^{-M}}.
$$
It is not difficult to see from this, and our definition of $\mu$, that we can bound both $\ell$ and $q$ by $\exp_p(\exp_p(O_{k,p}(\delta^{-O_{k,p}(1)}))$.  

Apply Lemma \ref{lem:rank} to $\calB$ to obtain $\calB'=(\calL',\calQ')\preceq \calB$ of complexity $(\ell',q')$ and rank at least $\rho(\ell'+q')$ for some integers $\ell',q'$ satisfying $0\leq q'\leq q$ and 
$$
\ell'\leq \tau_q^{\rho}(\ell,q)\leq 2^{qd^q}C^{qd^q}(\ell+q)^{d^q},
$$
where the second inequality is by Lemma \ref{lem:tau}. Note that if $d=1$, this implies
$$
\ell'\leq 2^{q}C^{q}(\ell+q)\leq \exp_p(\exp_p(\exp_p(O_{k,p,C}(\delta^{-O_{k,p}(1)}))),
$$
while if $d>1$, this implies
$$
\ell'\leq   \exp_p(\exp_p(\exp_p(\exp_p(O_{k,p,C,d}(\delta^{-O_{k,p}(1)})))).
$$
Define
\begin{align*}
\Sigma_0=\{B\in \At(\calB'): \alpha_B\leq \e\}\text{ and }\Sigma_1=\{B\in \At(\calB'): \alpha_B\geq (1-\e)\}.
\end{align*}
Since $\At(\calB')$ refines $\calR$, Fact \ref{fact:hom} implies $|\bigcup_{B\in \Sigma_0\cup \Sigma_1}B|\geq (1-\e)|G|$. Thus, we see that $\calB'$ is almost $\e$-homogeneous with respect to $A$.  By Theorem \ref{thm:quad}, there exist an integer $q''\geq 0$ and a quadratic factor $\calB''=(\calL',\calQ'')$ on $\F_p^n$ of complexity $(\ell',q'')$ with $\rk(\calB'')=\rk(\calB')$ such that  
\begin{enumerate} 
\item $q''\leq \log_p( \delta^{-k-o_{k,p}(1)})$;
\item there exists a union $Y$ of atoms of $\calB''$ satisfying $|A\Delta Y|\leq 16\delta^{1-o_{k,p}(1)}|G|$.
\end{enumerate}
 This finishes the proof.
\end{proofof}

\section{Appendix}

In this appendix, we collect the proofs of some straightforward lemmas used earlier in the paper. We begin by proving Lemma \ref{lem:constraints}.

\begin{lemma}\label{lem:constraintsrepeat}
Suppose $\ell,q\geq 0$ are integers and $\calB=(\calL,\calQ)$ is a quadratic factor on $G=\F_p^n$ of complexity $(\ell,q)$, and $B\in \At(\calB)$.  Then for any $(x,h_1,h_2,h_3)\in G^4$, the following are equivalent.
\begin{enumerate}
\item  $(x,h_1,h_2,h_3)\in \Omega_B$,
\item the following hold: $x\in B$, $h_i\in L(0)$ for all $i\in [3]$,  $2\beta_{\calQ}(x,h_i)=-\beta_{\calQ}(h_i,h_i)$ for all $i\in [3]$,  and $\beta_{\calQ}(h_i,h_j)=0$ for all $i\neq j\in [3]$.
\end{enumerate}
\end{lemma}
 
\begin{proof}
Suppose first $(x,h_1,h_2,h_3)\in \Omega_B$. Since $1_B(x)\neq 0$, $x\in B$.  For each $i\in [3]$, since $1_B(x+h_i)\neq 0$,  $x+h_i\in B$, which, along with $x\in B$ implies we must have $h_i\in L(0)$.  Further,
\begin{align*}
\beta_{\calQ}(x,x)=\beta_{\calB}(x+h_i,x+h_i)&= \beta_{\calQ}(x,x)+2\beta_{\calQ}(x,h_i)+\beta_{\calQ}(h_i,h_i).
\end{align*}
This immediately implies $2\beta_{\calQ}(x,h_i)=-\beta_{\calQ}(h_i,h_i)$. Finally, given $i\neq j\in [3]$, since $1_B(x)\neq 0$ and $1_B(x+h_i+h_j)\neq 0$, we have
\begin{align*}
\beta_{\calQ}(x,x)&=\beta_{\calQ}(x+h_i+h_j,x+h_i+h_j)\\
&=\beta_{\calQ}(x,x)+2\beta_{\calQ}(x,h_i)+2\beta_{\calQ}(x,h_j)+2\beta_{\calQ}(h_i,h_j)+\beta_{\calQ}(h_i,h_i)+2\beta_{\calQ}(h_j,h_j)\\
&= \beta_{\calQ}(x,x)+2\beta_{\calQ}(h_i,h_j),
\end{align*}
where the last equality is because we have already shown $2\beta_{\calQ}(x,h_i)=-\beta_{\calQ}(h_i,h_i)$ and $2\beta_{\calQ}(x,h_j)=-\beta_{\calQ}(h_j,h_j)$.  The displayed equation above then   implies $\beta_{\calQ}(h_i,h_j)=0$.  This finishes the proof that (1) implies (2).

Conversely, assume $(x,h_1,h_2,h_3)\in G^4$ satisfies (2).  Since $x\in B$, $1_B(x)\neq 0$.  For each $i\in [3]$, $h_i\in L(0)$ tells us then that
$$
\beta_{\calL}(x)=\beta_{\calL}(x+h_1).
$$
Further, since $2\beta_{\calQ}(x,h_i)=-\beta_{\calQ}(h_i,h_i)$,
$$
\beta_{\calQ}(x+h_i,x+h_i)=\beta_{\calQ}(x,x),
$$
and thus, we can conclude $x+h_i\in B$ holds as well.  Fix now $i\neq j\in [3]$.  Again, since $h_i,h_j\in L(0)$, we know
$$
\beta_{\calL}(x)=\beta_{\calL}(x+h_i+h_j).
$$
Using that $2\beta_{\calQ}(x,h_i)=-\beta_{\calQ}(h_i,h_i)$, $2\beta_{\calQ}(x,h_j)=-\beta_{\calQ}(h_j,h_j)$, and $\beta_{\calQ}(h_i,h_j)=0$, we have 
\begin{align*}
\beta_{\calQ}(x+h_i+h_j,x+h_i+h_j)&=\beta_{\calQ}(x,x),
\end{align*}
and consequently, we can conclude $1_B(x+h_i+h_j)\neq 0$.  Finally, the fact $1_B(x+h_1+h_2+h_3)\neq 0$ follows since $h_1,h_2,h_3\in L(0)$ implies
$$
\beta_{\calL}(x)=\beta_{\calL}(x+h_1+h_2+h_3),
$$
and, since $\beta_{\calQ}(h_1,h_2)=\beta_{\calQ}(h_1,h_3)=\beta_{\calQ}(h_2,h_3)=0$,
\begin{align*}
\beta_{\calQ}(x+h_1+h_2+h_3,x+h_1+h_2+h_3)&=\beta_{\calQ}(x+h_1+h_2,x+h_1+h_2)+2\beta_{\calQ}(x, h_3)+2\beta_{\calQ}(h_3,h_3)\\
&=\beta_{\calQ}(x,x),
\end{align*}
where the last equality is because we have already shown $\beta_{\calQ}(x+h_1+h_2,x+h_1+h_2)=\beta_{\calQ}(x,x)$, and since by assumption, $2\beta_{\calQ}(x, h_3)+2\beta_{\calQ}(h_3,h_3)=0$.  We can now conclude $x+h_1+h_2+h_3\in B$, finishing the proof.
\end{proof}

As a lemma towards Lemma \ref{lem:omegagood}, we prove the following.

\begin{lemma}\label{lem:badcount1}
Let $\calB=(\calL,\calQ)$ a quadratic factor on $G=\F_p^n$ of complexity $(\ell,q)$ and rank $r$.  Suppose $k\geq 0$ and $S\subseteq G$ is a sst of size $k$ such that $\calL\cup \{Mw: M\in \calQ, w\in S\}$ is linearly independent.  Then
$$
|\{x\in G: \calL\cup \{Mw: M\in \calQ, w\in S\}\cup \{Mx: M\in \calQ\}\text{ is not linearly independent}\}|\leq p^{n+\ell+(k+1)q-r}.
$$
\end{lemma}
\begin{proof}
Observe, 
\begin{align*}
&|\{x\in G: \calL\cup \{Mw: M\in \calQ, w\in S\}\cup \{Mw: M\in \calQ\}\text{ is not linearly independent}\}|\\
&\leq \sum_{ v\in \Span(\calL\cup \{Mw: M\in \calQ, w\in S\})\setminus \{0\}}\sum_{M'\in \calQ}|\{x\in G: M'x=v\}|\\
&\leq \sum_{ v\in \Span(\calL\cup \{Mw: M\in \calQ, w\in S\})\setminus \{0\}}\sum_{M'\in \calQ}p^{n-r}\\
&\leq p^{n+\ell+(k+1)q-r}.
\end{align*}
\end{proof}

We can now prove Lemma \ref{lem:omegagood}, which we restate here for the convenience of the reader.
\begin{lemma} 
Let $\calB=(\calL,\calQ)$ a high rank factor of complexity $(\ell,q)$ and rank $r$.  Then 
$$
|\{(w_1,w_2,w_3,w_4)\in G^4: \calL\cup \{Mw_1, Mw_2,Mw_3,Mw_4: M\in \calQ\}\text{ is not linearly independent}\}|
$$
is at most $14p^{4n+\ell+4q-r}$.
\end{lemma}
\begin{proof}
To ease notation, let
$$
\Sigma_i=\{(w_1,\ldots, w_i)\in G^i:  \calL\cup \{Mw_1, Mw_2,Mw_3,Mw_4: M\in \calQ\}\text{ is linearly independent}\}.
$$
Observe
\begin{align*}
&|\{(w_1,w_2,w_3,w_4)\in G^4: \calL\cup \{Mw_1, Mw_2,Mw_3,Mw_4: M\in \calQ\}\text{ is not linearly independent}\}|\\
&\leq 4|G|^3|\{w\in G: \calL\cup \{Mw: M\in \calQ\}\text{ not linearly independent}\}|\\
&+6\sum_{w\in \Sigma_1}|G|^2|\{x\in G: \calL\cup \{Mw,Mx: M\in \calQ\}\text{ not linearly independent}\}|\\
&+4\sum_{(w_1,w_2)\in \Sigma_2}|G||\{x\in G: \calL\cup \{Mw_1,Mw_2,Mx: M\in \calQ\}\text{ not linearly independent}\}|\\
&\leq 4|G|^3p^{n+\ell+q-r}+6|G|^2p^{n+\ell+3q-r}+4|G|p^{n+\ell+4q-r}\\
&\leq 14 p^{4n+\ell+4q-r},
\end{align*}
where the second to last inequality is by Lemma \ref{lem:badcount1}.
\end{proof}

\providecommand{\bysame}{\leavevmode\hbox to3em{\hrulefill}\thinspace}
\providecommand{\MR}{\relax\ifhmode\unskip\space\fi MR }
\providecommand{\MRhref}[2]{%
  \href{http://www.ams.org/mathscinet-getitem?mr=#1}{#2}
}
\providecommand{\href}[2]{#2}

\end{document}